\documentclass[11pt]{amsproc}

\usepackage{fullpage}
\usepackage[utf8]{inputenc}
\usepackage[T1]{fontenc}
\usepackage{graphicx,subfigure}
\usepackage{amsmath,amsfonts,amssymb,mathtools}
\usepackage{stmaryrd}
\usepackage{mathrsfs,dsfont}
\usepackage{amsthm}
\usepackage{mathabx}
\usepackage{tabularx}
\usepackage{float}

\usepackage{hyperref}

\usepackage{enumerate}

\usepackage{color}

\newcommand{\E}{\mathbb{E}}
\newcommand{\R}{\mathbb{R}}
\newcommand{\N}{\mathbb{N}}

\newcommand{\D}{\mathcal{D}}

\newcommand{\HH}{H}%full space
\newcommand{\hh}{\mathbb H}%mean-zero subspace
\newcommand{\PP}{\mathbf{P}}%projection

\newtheorem{theo}{Theorem}[section]

\newtheorem{rem}[theo]{Remark}

\newtheorem{propo}[theo]{Proposition}
\newtheorem{lemma}[theo]{Lemma}

\newtheorem{ass}{Assumption}

\usepackage{algorithm}
\usepackage{algorithmic}

\begin{document}

\title{Weak error estimates of fully-discrete schemes for the stochastic Cahn--Hilliard equation}

\author{Charles-Edouard Br\'ehier}
\address{Univ Lyon, Université Claude Bernard Lyon 1, CNRS UMR 5208, Institut Camille Jordan, 43 blvd. du 11 novembre 1918, F--69622 Villeurbanne cedex, France}
\email{brehier@math.univ-lyon1.fr}

\author{Jianbo Cui}
\address{Department of Applied Mathematics, The Hong Kong Polytechnic University, Hung Hom, Kowloon, Hong Kong}
\email{jianbo.cui@polyu.edu.hk}

\author{Xiaojie Wang}
\address{School of Mathematics and Statistics and HNP-LAMA, Central South University, Changsha, China}
\email{x.j.wang7@csu.edu.cn}

\begin{abstract}
We study a class of fully-discrete schemes for the numerical approximation of solutions of stochastic Cahn--Hilliard equations with cubic nonlinearity and driven by additive noise. The spatial (resp. temporal) discretization is performed with a spectral Galerkin method (resp. a tamed exponential Euler method). We consider two situations: space-time white noise in dimension $d=1$ and trace-class noise in dimensions $d=1,2,3$. In both situations, we prove weak error estimates, where the weak order of convergence is twice the strong order of convergence with respect to the spatial and temporal discretization parameters. To prove these results, we show appropriate regularity estimates for solutions of the Kolmogorov equation associated with the stochastic Cahn--Hilliard equation, which have not been established previously and may be of interest in other contexts.
\end{abstract}

\maketitle

\section{Introduction}

Over the last decades, a large number of research works have been devoted to the numerical study of stochastic partial differential equations (SPDEs). The classical convergence theory of numerical approximations often requires nonlinearities of SPDEs to be globally Lipschitz continuous (see, e.g., \cite{lord2014introduction}). However, nonlinearities of many practical SPDE models are only locally Lipschitz continuous and the classical convergence results for standard explicit numerical schemes fail to hold (see, e.g. \cite{hutzenthaler2011strong,beccari2019strong,cui2021}). Prominent examples of SPDEs with non-globally Lipschitz continuous nonlinearities include
stochastic phase-field models such as the  stochastic Allen--Cahn equation and the stochastic Cahn--Hilliard equation. In the last decade, numerical analysis of such stochastic phase-field models has been a very active field of research, see, e.g., \cite{KLM11,MP17,FLZ17,KLL18,CCZZ18,FKLL18,BG18,BK19,FLZ20} and references therein. Progress on the convergence aspects of numerical approximations for stochastic phase-field model has been made recently, including strong and weak convergence rate analysis for the stochastic Allen-Cahn equation (see, e.g., \cite{BJ19,BCH19,CH19,QW19,BG20,MR3536036,LQ20,Wang2020,CGW21,CHS21a,Bre22}) and the strong convergence rate analysis for the stochastic Cahn--Hilliard equation (see, e.g., \cite{QW20,CH20,ABNP21,CHS21,HDS22,CQW22,QCW22}). 

In this paper, we consider the stochastic Cahn–-Hilliard equation driven by additive noise and with homogeneous Neumann boundary conditions
\begin{equation}\label{eq:SPDEintro}
\left\lbrace
\begin{aligned}
&du-\Delta wdt=dW^Q(t),\quad (t,x)\in (0,T]\times\mathcal{D},\\
&w=-\Delta u+u^3-u,\quad (t,x)\in(0,T]\times\mathcal{D},\\
&\frac{\partial u}{\partial n}=\frac{\partial w}{\partial n}=0, \quad (t,x)\in(0,T]\times\partial\mathcal{D},\\
&u(0,x)=u_0(x),\quad x\in\mathcal{D},
\end{aligned}
\right.
\end{equation}
where $\mathcal D=(0,1)^d$, with $d\in\{1,2,3\}$ and $\bigl(W^Q(t)\bigr)_{t\ge 0}$ is a suitable $Q$-Wiener process. We consider two situations: space-time white noise in dimension $d=1$ and trace-class noise in dimensions $d=1,2,3$. In the sequel, the SPDE~\eqref{eq:SPDEintro} is interpreted in the sense of stochastic evolution equations, see~\cite{DZ14}, with values in $H=L^2(\mathcal{D})$:
\begin{equation}\label{sche}
dX(t)
+
A\bigl(AX(t) +F(X(t))\bigr)dt
=
dW^Q(t),
\quad
X(0)=x_0,
\end{equation}
where the linear operator $-A$ is the Laplace operator on the bounded domain $\mathcal D$ equipped with homogeneous Neumann boundary condition. The nonlinearity $ F(u)=u^3-u$ corresponds to the double-well potential. The driving Wiener process $W^Q$ can be used to describe impurities in the alloy, effects of external fields, or thermal fluctuations. As a stochastic phase-field model for studying the spinodal decomposition \cite{Cook70,Lan71},  the stochastic 
Cahn-Hilliard equation (also called Cahn–Hilliard–Cook equation) shows a better agreement with
experimental data in the presence of noise \cite{GMW05}.

The objective of this manuscript is to give weak error estimates for numerical schemes applied to~\eqref{sche}. Recall that in contrast  with strong error estimates, which mainly deal with mean-square convergence of the numerical schemes, weak error estimates are associated with the convergence in distribution, and that in many situations the weak order of convergence is twice the strong order of convergence for such SPDE systems. This result has been established recently for the stochastic Allen--Cahn equation, see~\cite{CH19,BG20,CGW21,CHS21a,Bre22} (and previously for parabolic semilinear SPDEs with globally Lipschitz continuous nonlinearities). We intend to fill a gap in the existing literature and prove that the result also holds for the stochastic Cahn--Hilliard equation. Since the strong order of convergence of the numerical scheme has been studied in previous works, we focus on the weak error estimates. The problem is more challenging than in the stochastic Allen--Cahn case since the nonlinearity does not satisfy a monotony or one-sided Lipschitz continuity property in $H=L^2(\mathcal{D})$. Note that it is relevant to study carefully the weak convergence rate, since owing to~\cite{Lang16} it has an influence on the required number of samples in multilevel Monte Carlo schemes~\cite{Gil15}.

In order to describe the main results of this work, let us first describe the considered numerical method. We propose and analyze an explicit fully discrete scheme
which combines the spectral Galerkin method and a tamed exponential Euler scheme, which reads 
\begin{equation}\label{eq:schemeintro}
X_{N,k+1}
=
e^{-\Delta tA^2} X_{N,k}
+( A^2 )^{-1} ( I - e^{-\Delta tA^2} ) \frac {-A P_N F (X_{N,k} ) } { 1 + \Delta t \| P_N F(X_{N, k})\|}
+ e^{-\Delta tA^2} P_N \Delta W_k^Q.
\end{equation}
with initial value $X_{N,0}=P_Nx_0=X_N(0)$ and Wiener increments
\[
\Delta W_k^Q=W^Q(t_{k+1})-W^Q(t_k).
\]
In the numerical scheme, $N\in\N$ is the parameter of the spectral Galerkin method, $P_N$ are the associated orthogonal projections, $\Delta t=T/K$ is the time step size, $T\in(0,\infty)$, $K\in\N$ and $t_k=k\Delta t$ for $k\in\{0,\ldots,K-1\}$.

The main result of this article is the following weak error estimate: for any function $\varphi:H\to\R$ of class $\mathcal{C}^2$ with bounded first and second order derivatives and all $\gamma\in[0,\Gamma)$, one has
\begin{equation}\label{eq:resultintro}
\big|\E[\varphi(X_{N,K})]-\E[\varphi(X(T))]\big|\le C(\gamma,T,x_0,\phi)(\Delta t^{\frac{\gamma}{2}}+\lambda_N^{-\gamma }),
\end{equation}
where the value of the parameter $\Gamma$ depends on the regularity of the noise: $\Gamma=3/2$ in the space-time white noise case ($d=1$) and $\Gamma=2$ in the trace-class noise case ($d\in\{1,2,3\}$). We refer to Theorems~\ref{theo:Galerkin} and~\ref{theo:tamed} below for precise statements. As mentioned above, we thus prove that for the considered scheme~\eqref{eq:schemeintro}, the weak order of convergence is twice the strong order (see the discussion below). We refer to the recent preprint~\cite{CQW22} for the strong error estimates satisfied by the scheme~\eqref{eq:schemeintro} considered in this article. We would like to mention that the recent preprint~\cite{CGH21} gives preliminary weak error estimates for implicit fully discrete schemes applied to stochastic Cahn--Hilliard equation~\eqref{sche}, however the authors of that preprint employ different techniques and they have not been able to prove that the weak order is twice the strong order. Our work thus substantially improves~\cite{CGH21}.

Let us review other recent relevant works on numerical approximations for the stochastic Cahn--Hilliard equation. The articles~\cite{FKLL18,KLM11} obtain the strong convergence and the semi-strong error estimates (in a large subsample space) of the finite element method and its implicit discretization for equation \eqref{sche} driven by a spatial regular noise. 
The manuscript~\cite{HJ14} studies both the exponential integrability and the strong convergence rate of the spectral Galerkin method for \eqref{sche} driven by trace class noise. By exploiting the equivalence of the original problem and a system consisting of a random PDE and stochastic convolution, the article~\cite{QW20} presents the strong convergence rate of the finite element method and a fully discrete scheme in the trace class noise case, and the article~\cite{CHS21} obtains strong error estimates with higher order for an accelerated implicit full discretization in the space-time white noise case. Finally, the article~\cite{CH20} provides a strong convergence analysis of numerical schemes for the stochastic Cahn--Hilliard equation driven by multiplicative noise with unbounded diffusion coefficient.

Let us also explain why it is relevant to consider the tamed numerical scheme~\eqref{eq:schemeintro}. It is known that applying standard explicit schemes (e.g., Euler-Maruyama) to stochastic differential equations (SDEs) with non-globally Lipschitz continuous coefficients may lead to divergence approximations due to the lack of moment bounds, see for instance \cite{hutzenthaler2011strong}. Using a tamed Euler method~\cite{hutzenthaler2012strong} is one of the many possible strategies to overcome this issue, and avoid implicit schemes (that would be admissible but may lead to higher computational costs), see the monograph~\cite{hutzenthaler2015numerical}  and references therein for further examples. Recently, tamed and truncated Euler schemes have been extended to treat  SPDEs with non-globally Lipschitz continuous coefficients, see for instance~\cite{gyongy2016convergence,jentzen2020strong,Wang2020,CQW22,CGW21,BJ19,CQW22}.
Most of the references are concerned with strong and weak convergence rates of explicit time-stepping schemes for Allen--Cahn type SPDEs. The tamed exponential explicit Euler scheme~\eqref{eq:schemeintro} has been introduced and studied for \eqref{sche} in the recent preprint~\cite{CQW22}, where only the strong convergence is studied and the $Q$-Wiener process is chosen to ensure mass preservation of the solution. Instead the present work considers
a general $Q$-Wiener process and for completeness we thus give a detailed proof of the required moment bounds property, see Theorem~\ref{theo:momentbounds-tamed} below, for the scheme~\eqref{eq:schemeintro}. In addition, the proof proposed in Section~\ref{sec:proofmomentsfull} is slightly different from the proof of~\cite[Corollary~3.6]{CQW22}.

In order to prove the weak error estimates~\eqref{eq:resultintro} (see Theorem~\ref{theo:Galerkin} and~\ref{theo:tamed}), one needs to prove some auxiliary regularity properties for solutions of the (infinite dimensional) Kolmogorov equations associated with~\eqref{sche}. We refer to Theorem~\ref{theo:Kolmogorov} below for a precise statement. Such results have been obtained for parabolic semilinear equations, such as the stochastic Allen--Cahn equation, however it seems that Theorem~\ref{theo:Kolmogorov} is the first result of this type for the stochastic Cahn--Hilliard equation. One of the major difficulty we need to deal with is the fact that the nonlinearity $AF$ in~\eqref{sche} does not satisfy a monotonicity property in $H$ (whereas the nonlinearity $F$ in the Allen--Cahn case is one-sided Lipschitz continuous). To overcome this issue, we study energy estimates in $L^2$ and $H^{-1}$ norms, which are similar to the tools used in~\cite{CQW22,QCW22} for instance to prove strong error estimates. The proofs of Theorem~\ref{theo:Galerkin} and~\ref{theo:tamed} are then based on standard arguments which need to combine carefully the moment bounds on the numerical scheme and the regularity results on the Kolmogorov equation to obtain the expected result, that the weak order is twice the strong order. Note that the proof does not require Malliavin calculus techniques, due to the choice of a scheme based on an exponential integrator. Theorem~\ref{theo:Kolmogorov} gives regularity results on the Kolmogorov equation associated with~\eqref{sche} which may be useful in other contexts.

In future works, it may be interesting to study weak error estimates for variants of the scheme~\eqref{eq:schemeintro}. Instead of using a spectral Galerkin method, the spatial approximation can be performed using a finite element method (see for instance~\cite{KLM11} and the strong error estimates in~\cite{QW20}). For the temporal discretization, one may use implicit methods for both the linear and the nonlinear parts. In addition, it may be interesting to study the long time behavior of the scheme and of the weak error estimates and the issue of approximation of invariant distributions.

The article is organized as follows. Section \ref{sec:setting} is devoted to giving the notation, the assumptions in order to be able to handle the stochastic evolution equation~\eqref{sche}.
In Section~\ref{sec:main}, we then provide a detailed presentation of the numerical scheme~\eqref{eq:schemeintro}, and separate the analysis of the spatial and temporal discretization errors. We state the main results of this article: Theorem~\ref{theo:Galerkin} and~\ref{theo:tamed}. Section~\ref{sec:Kolmogorov} is concerned with the statement and the proof of the regularity properties of the Kolmogorov equation. The proof of Theorem~\ref{theo:Galerkin} on the spatial discretization error is given in Section~\ref{sec:proofGalerkin} whereas the proof of Theorem~\ref{theo:tamed} on the temporal discretization error is given in Section~\ref{sec:prooffull}. Section~\ref{sec:proofmomentsfull} gives the proof of the moment bounds stated in Theorem~\ref{theo:momentbounds-tamed}.

\section{Setting}\label{sec:setting}

\subsection{Notation}

The set of integers is denoted by $\N=\{1,\ldots,\}$, and $\N_0=\{0\}\cup\N$. In the sequel, we often use the notation $j\ge 1$ (resp. $j\ge 0$) instead of $j\in\N$ (resp. $j\in\N_0$).

Let $d\in\{1,2,3\}$ and let $\D=(0,1)^d\subset \R^d$ be the standard $d$-dimensional unit cube. For all $p\in[1,\infty]$, let $L^p$ denote the standard Banach space $L^p(\D)$ of the $p$-fold integrable real-valued mappings defined on $\D$, with the norm denoted by $\|\cdot\|_{L^p}$. In this article, we denote the space of square-integrable real-valued mappings defined on $\D$ ($p=2$) by $\HH=L^2=L^2(\D)$. Note that $\HH$ is an Hilbert space, with norm and inner product denoted by $\|\cdot\|$ and $\langle\cdot,\cdot\rangle$ respectively. For all $x,y\in \HH$, one has $\langle x,y\rangle=\int_{\D}x(\xi)y(\xi)d\xi$, and for all $p\in(1,\infty)$ one also use the notation $\langle x,y\rangle=\int_{\D}x(\xi)y(\xi)d\xi$ for all $x\in L^p(\D)$ and $y\in L^{\frac{p}{p-1}}(\D)$. If $L:\HH\to\HH$ is a linear operator, its trace (when it is finite) is denoted by ${\rm Tr}(L)$. The set of bounded linear operators from $\HH$ to $\HH$ is denoted by $\mathcal{L}(\HH)$, and the set of Hilbert--Schmidt linear operators from $\HH$ to $\HH$ is denoted by  $\mathcal{L}_2(\HH)$, with associated norm denoted by $\|\cdot\|_{\mathcal{L}_2(\HH)}$.

For all $s\in[0,2]$, let $W^{s,2}(\D)$ denote the standard fractional Sobolev spaces, with norm denoted by $\|\cdot\|_{W^{s,2}}$, see for instance~\cite[Section~1.11]{Yagi10}.

If $\varphi:\HH\to\R$ is a mapping of class $\mathcal{C}^2$ with bounded first and second order derivatives, set
\begin{equation}\label{eq:varphi-notation}
\vvvert\varphi\vvvert_2=\underset{x\in \HH}\sup~\underset{h\in\HH\setminus\{0\}}\sup~\frac{D\varphi(x).h}{|h|}+\underset{x\in\HH}\sup~\underset{h_1,h_2\in\HH\setminus\{0\}}\sup~\frac{D^2\varphi(x).(h_1,h_2)}{|h_1||h_2|}.
\end{equation}

In the sequel, the values of constants $C\in(0,\infty)$ appearing in the statements and proofs of the moment bounds and of the error estimates may change from line to line. Similarly, the value of the integer $q\in\N$ may also change from line to line.

\subsection{Assumptions}\label{sec:ass}

\subsubsection{Linear operator}

In this article, $-A$ is the Laplace operator $\Delta$ considered with homogeneous Neumann boundary conditions on $\partial\D$. This means that the linear operator $A$ is unbounded and self-adjoint, with domain $D(A)=\{v\in W^{2,2}(\D);~\frac{\partial v}{\partial n}=0~\text{on}~\partial\D\}\subset \HH$. Moreover, there exists a complete orthonormal system $\bigl(e_j\bigr)_{j=0,1,\ldots}$ of the Hilbert space $\HH$ and a non-decreasing sequence of nonnegative real numbers $\bigl(\lambda_j\bigr)_{j\ge 0}$, such that for all $j\in\N_0$ one has
\[
Ae_j=\lambda_je_j.
\]
One has $e_0(\cdot)=1$ and $\lambda_0=0$. In addition, one has $\lambda_j\ge \lambda_1>0$ for all $j\in\N$, and there exists $c_d\in(0,\infty)$ such that $\lambda_j\sim c_dj^{2/d}$ when $j\to\infty$.

Introduce the Hilbert space $\hh\subset\HH$ defined by
\begin{equation}\label{eq:hh}
\hh={\rm span}\{e_j;~j\ge 1\}
\end{equation}
and let $\PP$ be the associated orthogonal projection operator: for all $x\in\HH$, one has
\begin{equation}\label{eq:PP}
\PP x=\sum_{j\ge 1}\langle x,e_j\rangle e_j=x-\langle x,e_0\rangle e_0.
\end{equation}

For all $\alpha\in\R^+$ and all $x\in\hh$, set
\begin{equation}\label{eq:seminorm_alpha}
|x|_\alpha^2=\sum_{j=1}^{\infty}\lambda_j^\alpha \langle x,e_j\rangle^2\in[0,\infty],
\end{equation}
and for all $\alpha\in\R^+$ introduce the space
\begin{equation}\label{eq:hh_alpha}
\hh^\alpha=\{x\in \hh;~|x|_\alpha<\infty\}.
\end{equation}
The space $\hh^\alpha$ may be considered both as a dense subspace of the Hilbert space $\hh$, and as an Hilbert space with norm $|\cdot|_\alpha$. Note that one obtains $\hh=\hh^0$ when $\alpha=0$, and in that case the notation $|\cdot|=|\cdot|_0$ is used in the sequel.

For any $\alpha\in\R^+$, the linear unbounded self-adjoint, operator $A^{\frac{\alpha}{2}}$ is then defined with the domain $D(A^{\frac{\alpha}{2}})=\hh^\alpha$ and such that for all $x\in \hh^\alpha$ one has
\begin{equation}\label{eq:A_alpha}
A^{\frac{\alpha}{2}}x=\sum_{j\ge 1}\lambda_j^{\frac{\alpha}{2}}\langle x,e_j\rangle e_j.
\end{equation}
Finally, for all $\alpha\in\R^+$ set
\begin{equation}\label{eq:HH_alpha}
\HH^\alpha=\{x\in \HH;~\PP x\in\hh^\alpha\},
\end{equation}
and for all $x\in \HH^\alpha$ set
\begin{equation}\label{eq:norm_alpha}
\|x\|_\alpha^2=|\PP x|_\alpha^2+\langle x,e_0\rangle^2.
\end{equation}
For any $\alpha\in\R^+$, one defines the bounded self-adjoint linear operator $A^{-\frac{\alpha}{2}}$ on $\hh$ as follows: for all $x\in\hh$, set
\begin{equation}\label{eq:A_alpha-neg}
A^{-\frac{\alpha}{2}}x=\sum_{j\ge 1}\lambda_j^{-\frac{\alpha}{2}}\langle x,e_j\rangle e_j.
\end{equation}

For all $N\in\N$, define the finite dimensional subspace
\begin{equation}\label{eq:HH_N}
\HH_N={\rm span}\{e_0,\ldots,e_N\}
\end{equation}
of $\HH$. In addition, define the orthogonal projection operator $P_N$ as follows: for all $x\in H$, set
\begin{equation}\label{eq:P_N}
P_Nx=\sum_{j=0}^{N}\langle x,e_j\rangle e_j \in\HH_N.
\end{equation}
Let also $\hh_N={\rm span}\{e_1,\ldots,e_N\}$. Note that $\HH_N\subset \HH^\alpha$ and $\hh_N\subset \hh^\alpha$ for all $\alpha\ge 0$ and $N\in\N$.

It remains to introduce the semigroup $\bigl(e^{-tA^2}\bigr)_{t\ge 0}$ of bounded linear operators defined on $\HH$. For all $t\ge 0$ and all $x\in \HH$, set
\begin{equation}\label{eq:semigroup}
e^{-tA^2}x=\sum_{j\ge 0}e^{-t\lambda_j^2}\langle x,e_j\rangle.
\end{equation}
Note that $e^{-tA^2}e_0=e_0$ for all $t\ge 0$, and that for all $x\in\hh$ and $t\ge 0$, one has $e^{-tA^2}x\in\hh$. For all $t\ge 0$, $e^{-tA^2}$ is a bounded linear operator from $\HH$ to $\HH$, and one has
\begin{equation}\label{eq:bounds-semigroup}
\begin{aligned}
\|e^{-tA^2}x\|&\le \|x\|,\\
\|e^{-tA^2} \PP x\|&\le e^{-t\lambda_1^2}\|\PP x\|,
\end{aligned}
\end{equation}
for all $t\ge 0$ and $x\in\HH$.

The following smoothing property is employed frequently in this article: for all $\alpha\in[0,4]$, there exists $C_\alpha\in(0,\infty)$ such that for all $t\in(0,\infty)$ and all $x\in\HH$, one has $e^{-tA^2}x\in \HH^\alpha$ with
\begin{equation}\label{eq:smoothing-HH}
\|e^{-tA^2}x\|_\alpha\le C_\alpha (1+t^{-\frac{\alpha}{4}})\|x\|.
\end{equation}
In fact, more precisely one has $e^{-tA^2}\PP x\in\hh^\alpha$ and the inequality
\begin{equation}\label{eq:smoothing-hh}
|e^{-tA^2}\PP x|_\alpha\le C_\alpha t^{-\frac{\alpha}{4}}|\PP x|.
\end{equation}
In addition, the following regularity property holds: for all $\alpha\in[0,4]$, there exists $C_\alpha\in(0,\infty)$ such that for all $t\ge 0$ and all $x\in\HH^\alpha$, one has
\begin{equation}\label{eq:regul}
\|\bigl(e^{-tA^2}-I\bigr)x\|\le C_\alpha t^{\frac{\alpha}{4}}\|x\|_\alpha.
\end{equation}
In fact, more precisely, one has the equality $\langle e^{-tA^2}x,e_0\rangle=\langle x,e_0\rangle$ and the inequality
\begin{equation}\label{eq:regul2}
\|\bigl(e^{-tA^2}-I\bigr)\PP x\|\le C_\alpha t^{\frac{\alpha}{4}}|\PP x|_\alpha.
\end{equation}

\subsubsection{Nonlinearity}

Let $f:\R\to\R$ be the polynomial function defined by $f(z)=z^3-z$ for all $z\in\R$. For all $p\in[1,\infty]$, the mapping $F:L^{3p}\to L^p$ is defined by the equality
\begin{equation}\label{eq:F}
F(x)=f(x(\cdot))=x^3-x,
\end{equation}
for all $x\in L^{3p}$.

Observe that for all $y\in L^4$, one has $F(y)\in L^{\frac43}$, therefore the expression $\langle F(y),y\rangle$ is well-defined. In addition, for all arbitrarily small $\epsilon\in(0,1)$, there exists $C_\epsilon\in(0,\infty)$ such that for all $y\in L^4$, one has
\begin{equation}\label{eq:ineqF1}
-\langle F(y),y\rangle=-\|y\|_{L^4}^4+\|y\|_{L^2}^2\le -(1-\epsilon)\|y\|_{L^4}^4+C_\epsilon
\end{equation}
and for all $y,z\in L^4$, one has
\begin{equation}\label{eq:ineqF2}
|\langle F(y+z)-F(y),y\rangle|\le \epsilon\|y\|_{L^4}^4+C_\epsilon\bigl(1+\|z\|_{L^4}^4\bigr).
\end{equation}

\subsubsection{Wiener process}

Let $\bigl(\beta_j\bigr)_{j\ge 0}$ be a sequence of independent standard real-valued Wiener processes, defined on a probability space $\bigl(\Omega,\mathcal{F},\mathbb{P}\bigr)$, equipped with a filtration $\bigl(\mathcal{F}_t\bigr)_{t\ge 0}$ which satisfies the usual conditions.

The cylindrical Wiener process is defined formally by
\begin{equation}\label{eq:W}
W(t)=\sum_{j\ge 0}\beta_j(t)e_j,
\end{equation}
where we recall that $\bigl(e_j\bigr)_{j\in\N_0}$ is the complete orthonormal system of $\HH$ associated with the linear operator $A$. Note that
\[
\E[\|A^{-\frac{\alpha}{2}}W(t)\|^2]=\sum_{j\ge 1}\lambda_j^{-\alpha}<\infty
\]
if and only if $\alpha>d/2$ (since one has $\lambda_j\sim c_dj^{2/d}$ when $j\to\infty$), where we recall that $(-A)^{-\frac{\alpha}{2}}$ is defined by~\eqref{eq:A_alpha-neg}. In particular, $W(t)$ does not take values in $\HH$.

Two cases are treated in this article. On the one hand, we consider the stochastic Cahn--Hilliard equation driven by space-time white noise in dimension $d=1$: in this case, set $Q=I$ the identity operator, and $W^Q(t)=W(t)$. On the other hand, we consider the stochastic Cahn--Hilliard equation driven by trace-class noise in dimension $d\in\{1,2,3\}$: in this case, let $Q$ be a self-adjoint, nonnegative, trace-class linear operator on $\HH$, and set $W^Q(t)=Q^{\frac12}W(t)$, where $Q^{\frac12}$ is the square-root of $Q$. Precisely, there exists a sequence of nonnegative real numbers $\bigl(q_j\bigr)_{j\in\N_0}$, and a complete orthonormal system $\bigl(e_j^Q\bigr)_{j\in\N_0}$ of $\HH$, such that one has $\sum_{j\ge 0}q_j<\infty$, and for all $x\in\HH$ one has
\[
Qx=\sum_{j\ge 0}q_j\langle x,e_j^Q\rangle e_j^Q.
\]
The operator $Q^{\frac12}$ is defined such that for all $x\in\HH$ one has
\[
Q^{\frac12}x=\sum_{j\ge 0}\sqrt{q_j}\langle x,e_j^Q\rangle e_j^Q.
\]
The operator $Q$ is trace-class by assumption, which implies that $Q^{\frac12}$ is an Hilbert--Schmidt operator:
\begin{align*}
{\rm Tr}(Q)&=\sum_{j\ge 0}\langle Qe_j,e_j\rangle=\sum_{j\ge 0}\langle Qe_j^Q,e_j^Q\rangle=\sum_{j\ge 0}q_j<\infty\\
\|Q^{\frac12}\|_{\mathcal{L}_2(\HH)}^2&=\sum_{j\ge 0}\|Q^{\frac12}e_j\|^2=\sum_{j\ge 0}\|Q^{\frac12}e_j^Q\|^2=\sum_{j\ge 0}q_j<\infty.
\end{align*}
The $Q$-Wiener process $\bigl(W^Q(t)\bigr)_{t\ge 0}$ can be expressed as
\begin{equation}\label{eq:WQ_1}
W^Q(t)=\sum_{j\ge 0}\sqrt{q}_j\beta_j^Q(t)e_j^Q,
\end{equation}
where $\bigl(\beta_j^Q(t)\bigr)_{j\in\N_0}$ is a sequence of independent standard real-valued Wiener processes. In the sequel, we only use the expression
\begin{equation}\label{eq:WQ_2}
W^Q(t)=\sum_{j\ge 0}\beta_j(t)Q^{\frac12}e_j=Q^{\frac12}W(t),
\end{equation}
where $W(t)$ is given by~\eqref{eq:W}, which is well-defined in $\HH$: for all $t\ge 0$ one has $\E[\|W^Q(t)\|^2]=t\sum_{j\ge 0}q_j=t{\rm Tr}(Q)<\infty$. The two formulations~\eqref{eq:WQ_1} and~\eqref{eq:WQ_2} for $W^Q(t)$ are equivalent.

Note that if the range of $Q$ is included in $\hh$, meaning that $Qx\in\hh$ for all $x\in\HH$, or equivalently that $\langle Qx,e_0\rangle=0$ for all $x\in\HH$, then the $Q$-Wiener process $W^Q$ takes values in $\hh$: $W^Q(t)\in\hh$ for all $t\ge 0$. The condition above also implies the mass conservation for solutions of~\eqref{eq:SPDEintro} (see \cite{FKLL18,CHS21} for instance).

Let us mention that it is not necessary to assume that the linear operators $A$ and $Q$ commute.

\subsection{Auxiliary inequalities}\label{sec:auxineq}

Owing to the Cauchy--Schwarz inequality, for all $x\in\HH^1$, one has
\begin{equation}\label{eq:ineq1}
\|x\|^2\le |A^{-\frac{1}{2}}\PP x|~|A^{\frac{1}{2}}\PP x|+\langle x,e_0\rangle^2=\|A^{-\frac{1}{2}}\PP x\| \|\PP x\|_1+\langle x,e_0\rangle^2.
\end{equation}

Let us recall the following result concerning the spaces $\HH^\alpha=D(A^{\frac{\alpha}{2}})$ and the associated norms $\|\cdot\|_\alpha$ for $\alpha\in[0,2]$, see for instance~\cite[Theorem~16.9]{Yagi10}. On the one hand, for all $\alpha\in[0,\frac32)$, one has
\[
\HH^\alpha=W^{\alpha,2}(\D).
\]
On the other hand, for all $\alpha\in(\frac32,2]$, one has
\[
\HH^\alpha=W_N^{\alpha,2}(\D)=\{v\in W^{\alpha,2}(\D);~\frac{\partial v}{\partial n}=0~\text{on}~\partial\D\}\subset W^{\alpha,2}(\D).
\]
In addition, for all $\alpha\in[0,2]\setminus\{\frac32\}$, the norms $\|\cdot\|_\alpha$ and $\|\cdot\|_{W^{\alpha,2}}$ are equivalent on $\HH^\alpha$: there exists $C_\alpha\in[1,\infty)$ such that for all $x\in\HH^\alpha$ one has
\begin{equation}\label{eq:norm-equivalence}
C_\alpha^{-1}\|x\|_{W^{\alpha,2}}\le \|x\|_{\alpha}\le C_\alpha\|x\|_{W^{\alpha,2}}.
\end{equation}
Recall the following Sobolev embedding property: for all $\alpha>d/2$, one has $W^{\alpha,2}(\D)\subset L^\infty(\D)$, and there exists $C_{\alpha,d}\in(0,\infty)$ such that
\begin{equation}\label{eq:Sobolev}
\|x\|_{L^\infty}\le C_{\alpha,d}\|x\|_{W^{\alpha,2}}.
\end{equation}
As a consequence of~\eqref{eq:norm-equivalence} and~\eqref{eq:Sobolev}, for all $\alpha\in(d/2,2]\setminus\{3/2\}$, there exists $C_\alpha\in(0,\infty)$ such that one has the inequality
\begin{equation}\label{eq:Sobolev-norm_alpha}
\|x\|_{L^\infty}\le C_{\alpha}\|x\|_{\alpha}
\end{equation}
for all $x\in\HH^\alpha$. In addition, for all $\alpha\in(d/2,2]\setminus\{3/2\}$, the space $\HH^\alpha$ is an algebra: there exists $C_\alpha\in(0,\infty)$ such that for all $x,y\in\HH^\alpha$, one has $xy\in\HH^\alpha$ and the inequality
\begin{equation}\label{eq:algebra}
\|xy\|_\alpha\le C_\alpha\|x\|_\alpha\|y\|_\alpha.
\end{equation}
Note also that for all $x\in\HH$ and $y\in\HH^\alpha$ with $\alpha\in(d/2,2]\setminus\{3/2\}$, one has
\begin{equation}\label{eq:ineq2}
\|xy\|\le C_\alpha\|x\|\|y\|_\alpha.
\end{equation}
A duality argument implies the following additional inequality: for all $\alpha\in(d/2,2]\setminus\{3/2\}$ and $x\in L^1$, one has
\begin{equation}\label{eq:ineq3}
\|A^{-\frac{\alpha}{2}}\PP x\|\le C_\alpha\|x\|_{L^1}.
\end{equation}
Note that $\HH^1$ is an algebra only when $d=1$, whereas $\HH^2$ is an algebra for all values of $d\in\{1,2,3\}$.

In order to prove the moment bounds for the exact and numerical solutions stated below, the cases $d=1$ and $d\in\{2,3\}$ are treated using different techniques. On the one hand, when $d=1$, one has the Sobolev embedding $H^{\frac13}=W^{\frac13,2}(\D)\subset L^6(\D)$ (see for instance~\cite[Theorem~6.7]{MR2944369}), and there exists $C\in(0,\infty)$ such that for all $x\in\HH^{\frac13}$ one has
\begin{equation}\label{eq:L6-d1}
\|x\|_{L^6}\le C\|x\|_{W^{\frac13,2}}\le C\|x\|_{\frac13}\le C\bigl(|A^{\frac16}\PP x|+|\langle x,e_0\rangle\bigr|).
\end{equation}
Moreover, when $d=1$, one also has the following Gagliardo--Nirenberg inequality, see, e.g., \cite[Theorem 1]{MR3813967}: there exists $C\in(0,\infty)$ such that for all $x\in\HH^2$ one has
\begin{equation}\label{eq:L6-d1-GN}
\|x\|_{L^6}\le C\|x\|_{W^{2,2}}^{\frac16}\|x\|^{\frac56}\le C\|x\|_2^{\frac16}\|x\|^{\frac56}\le C\bigl(|A\PP x|+|\langle x,e_0\rangle\bigr|)^{\frac16}\|x\|^{\frac56}.
\end{equation}
On the other hand, when $d\in\{2,3\}$, we introduce the energy functional $J$ defined as follows: for all $x\in \HH^1\cap L^4$ set
\begin{equation}\label{eq:energy}
J(x)=\frac12\|x\|_{\HH^1}^2+\frac14\|x\|_{L^4}^4-\frac12\|x\|_{L^2}^2.
\end{equation}

In addition, for all $d\in\{1,2,3\}$, one has the Sobolev embedding properties $H^{\frac{d}{4}}\subset L^4$ and $H^1\subset L^6$ (see for instance~\cite[Theorem~6.7]{MR2944369}), and there exists $C_d\in(0,\infty)$ such that for all $x\in\HH^{\frac{d}{4}}$ one has
\begin{align}
\|x\|_{L^4}\le C\|x\|_{W^{\frac{d}{4},2}}\le C\|x\|_{\frac{d}{4}},\label{eq:L4-d}\\
\|x\|_{L^6}\le C\|x\|_{W^{1,2}}\le C\|x\|_{1}.\label{eq:L6-d}
\end{align}

The properties of the spaces $\HH^\alpha$ discussed above are now applied to state a few additional properties of the nonlinearity $F$. For all $d\in\{1,2,3\}$ and $\gamma\in(d/2,2]\setminus\{3/2\}$, there exists $C_\gamma\in(0,\infty)$ such that for all $x\in\HH^\gamma$ one has $F(x)\in\HH^\gamma$ and
\begin{equation}\label{eq:boundF1}
\|F(x)\|_\gamma\le C_\gamma \bigl(1+\|x\|_\gamma^3\bigr).
\end{equation}
The inequality~\eqref{eq:boundF1} is a straightforward consequence of the algebra property~\eqref{eq:algebra} for $\HH^\gamma$. In addition, for all $\gamma\in(d/2,2]\setminus\{3/2\}$, there exists $C_\gamma\in(0,\infty)$ such that for all $y\in L^4$ and $z\in\HH^\gamma$ one has
\begin{equation}\label{eq:boundF2}
\|F(y+z)-F(y)\|\le  \bigl(3\|yz^2\|+3\|zy^2\|+\|z^3\|\bigr)\le  C_\gamma\bigl(1+\|y\|_{L^4}^2\bigr)\bigl(1+\|z\|_{\gamma}^3\bigr).
\end{equation}
Finally, for all $\gamma\in(d/2,2]\setminus\{3/2\}$, there exists $C_\gamma\in(0,\infty)$ such that for all $x\in\HH^\gamma$ and all $h,k\in\HH$, one has
\begin{equation}\label{eq:Fderiv}
\begin{aligned}
\|F'(x).h\|&=\|f'(x(\cdot))h(\cdot)\|_{L^2}\le C\bigl(1+\|x\|_\gamma^2\bigr)\|h\|\\
\|F''(x).(h,k)\|_{L^1}&=\|f''(x(\cdot))h(\cdot)k(\cdot)\|_{L^1}\le C\bigl(1+\|x\|_\gamma\bigr)\|h\|\|k\|.
\end{aligned}
\end{equation}
The inequalities~\eqref{eq:Fderiv} are straightforward consequences of the algebra property~\eqref{eq:algebra} for $\HH^\gamma$.

\subsection{Well-posedness and moment bounds}\label{sec:wellposed}

Let us introduce the stochastic convolution process $\bigl(Z(t)\bigr)_{t\ge 0}$ defined by
\begin{equation}\label{eq:Z}
Z(t)=\int_0^t e^{-(t-s)A^2}dW^Q(s),
\end{equation}
which is understood as the unique mild solution of the stochastic evolution equation
\[
dZ(t)+A^2Z(t)dt=dW^Q(t)
\]
with initial value $Z(0)=0$.

Recall that the Wiener process $\bigl(W^Q(t)\bigr)_{t\ge 0}$ is defined by~\eqref{eq:WQ_2}, where two cases are considered for the covariance operator $Q$. In order to study the properties of the stochastic convolution process defined by~\eqref{eq:Z}, it is convenient to introduce the parameters $\Gamma$ and $\Gamma_0$ defined as follows.
\begin{ass}\label{ass:Q}
Let one of the conditions be satisfied.
\begin{enumerate}
\item[$(i)$] Let $d=1$ and $Q=I$. In that case, set $\Gamma=\frac32$ and $\Gamma_0=1$.
\item[$(ii)$] Let $d\in\{1,2,3\}$ and let $Q$ be a self-adjoint, nonnegative, trace-class linear operator on $\HH$. In that case, set $\Gamma=2$ and $\Gamma_0=1+\frac{d}{4}$.
\end{enumerate}
\end{ass}
Note that in both cases the condition $\gamma\in(\Gamma_0,\Gamma)$ implies $\gamma\in(d/2,2]\setminus\{3/2\}$. Imposing the condition $\gamma\in(\Gamma_0,\Gamma)$ is necessary in the proofs, however assuming that $\gamma>\Gamma_0$ is not restrictive since the orders of convergence are obtained when choosing $\gamma$ arbitrarily close to $\Gamma$.

One has the following result concerning the stochastic convolution (see, e.g., \cite[Section 2]{CHS21} and~\cite[Lemma 3.2]{QCW22}).
\begin{lemma}
Let Assumptions~\ref{ass:Q} be satisfied. For all $T\in(0,\infty)$, $m\in\N$ and all $\gamma\in[0,\Gamma)$, one has
\begin{equation}\label{eq:strong-moment-Zt}
\E[\underset{0\le t\le T}\sup~\|Z(t)\|_{L^\infty}^m]+ \E[\underset{0\le t\le T}\sup~\|Z(t)\|_\gamma^m]<\infty.
\end{equation}
\end{lemma}
The moment bounds in the $L^\infty$ norm (first term in the left-hand side of~\eqref{eq:strong-moment-Zt}) are a straightforward consequence of the moment bounds in the norm $\|\cdot\|_{\gamma}$ (second term in the left-hand side of~\eqref{eq:strong-moment-Zt}) and of the Sobolev inequality~\eqref{eq:Sobolev-norm_alpha}.

We are now in position to state a well-posedness result for mild solutions of the stochastic evolution equation
\begin{equation}\label{eq:SPDE}
dX(t)+A\bigl(AX(t)+F(X(t))\bigr)dt=dW^Q(t)
\end{equation}
with initial value $X(0)=x_0$. To indicate that the initial value of $X(0)$ is equal to $x_0$, the notation $\E_{x_0}[\cdot]$ is used in the sequel.

\begin{propo}\label{propo:wellposed}
Let Assumption~\ref{ass:Q} be satisfied and assume that $x_0\in \HH^\Gamma$. Then the stochastic evolution equation~\eqref{eq:SPDE} admits a unique global mild solution with initial value $X(0)=x_0$, namely a $\HH$-valued continuous stochastic process $\bigl(X(t)\bigr)_{t\ge 0}$ such that for all $t\ge 0$ one has
\begin{equation}\label{eq:mild}
X(t)=e^{-tA^2}x_0-\int_0^t e^{-(t-s)A^2}AF(X(s))ds+\int_0^t e^{-(t-s)A^2}dW^Q(s).
\end{equation}
Moreover, there exists $q\in\N$ such that for all $T\in(0,\infty)$, $m\in\N$ and all $\gamma\in(\Gamma_0,\Gamma)$, one has the moment bound
\begin{equation}\label{eq:momentsX}
\underset{x_0\in \HH^\gamma}\sup~\frac{\bigl(\E_{x_0}[\underset{0\le t\le T}\sup~\|X(t)\|_\gamma^m]\bigr)^{\frac1m}}{(1+\|x_0\|_\gamma^{q})}<\infty
\end{equation}
and the increment bound
\begin{equation}\label{eq:incrementsX}
\underset{x_0\in \HH^\gamma}\sup~\underset{0\le t_1<t_2\le T}\sup~\frac{\bigl(\E_{x_0}[\|X(t_2)-X(t_1)\|^m]\bigr)^{\frac1m}}{|t_2-t_1|^\frac{\gamma}{4}(1+\|x_0\|_\gamma^{q})}<\infty.
\end{equation}
\end{propo}
The proof of Proposition~\ref{propo:wellposed} is omitted. Under Assumption~\ref{ass:Q}-$(i)$ (space-time white noise case in dimension $d=1$), the strong moment bounds~\eqref{eq:momentsX} follow from~\cite[Proposition 1]{CHS21} for instance. Under Assumption~\ref{ass:Q}-$(ii)$ (trace-class noise case in dimensions $d = 1, 2, 3$),  the strong moment bounds~\eqref{eq:momentsX} follow from~\cite[Theorem 3.7]{QCW22} and \eqref{eq:strong-moment-Zt} for instance. The increment bounds~\eqref{eq:incrementsX} are then obtained by standard techniques. Observe that the Kolmogorov regularity criterion ensures that almost surely the trajectories $t\in[0,T]\mapsto X(t)\in\HH$ are $\frac{\gamma}{4}$-H\"older continuous for all $\gamma\in[0,\Gamma)$.

\section{Numerical methods and convergence results}\label{sec:main}

In this section, we first describe in Section~\ref{sec:Galerkin} the method employed for the spatial discretization: a spectral Galerkin approximation method is introduced. Our first main result is Theorem~\ref{theo:Galerkin}, which gives a weak order of convergence $\Gamma$ in terms of $\lambda_N$ when the approximation parameter $N$ goes to $\infty$ -- whereas the strong order of convergence is known to be equal to $\Gamma/2$, see for instance~\cite{CHS21,QCW22} and~\eqref{eq:strongerrorGalerkin} below. We then describe in Section~\ref{sec:full} the fully discrete method: a tamed exponential Euler scheme is employed for the temporal discretization. Our second main result is Theorem~\ref{theo:tamed}, which gives a weak order of convergence $\Gamma/2$ in terms of the time-step size $\Delta t$ -- whereas the strong order of convergence is known to be equal to $\Gamma/4$, see for instance~\cite{CQW22}), see also~\cite{CHS21} for a similar result for an implicit scheme. When an accelerated exponential Euler scheme is used, there is no approximation error for the contribution of the stochastic convolution and the temporal strong order of convergence may be improved, see~\cite{CHS21}.

\subsection{Spatial discretization}\label{sec:Galerkin}

Let $N\in\N$ be an integer. Introduce the $\HH_N$-valued process $\bigl(X_N(t)\bigr)_{t\ge 0}$ which is the solution of the stochastic evolution equation
\begin{equation}\label{eq:Galerkin}
dX_N(t)+A\bigl(AX_N(t)+P_NF(X_N(t))\bigr)dt=P_NdW^Q(t)
\end{equation}
with initial value $X_N(0)=P_Nx_0$. The process $X_N$ is an approximation of $X$ using the spectral Galerkin method. Note that the noise in the stochastic evolution equation~\eqref{eq:Galerkin} is a Wiener process with covariance operator $P_NQP_N$. If the covariance operator $Q$ and the linear operator $A$ commute (which is not assumed to the case in general), then one has $P_NdW^Q(t)=\sum_{j=0}^{N}\sqrt{q_j}\beta_j(t)e_j$.

Before stating Theorem~\ref{theo:Galerkin}, it is worth mentioning two auxiliary results. First, the results of Proposition~\ref{propo:wellposed} hold for $X_N$, uniformly with respect to $N\in\N$. Precisely, for all $N\in\N$, the stochastic evolution equation~\eqref{eq:Galerkin} admits a unique mild solution with values in $\HH_N$, defined for all times. Moreover, for all $T\in(0,\infty)$, $m\in\N$ and $\gamma\in(\Gamma_0,\Gamma)$, there exists $C_{\gamma,m}(T)\in(0,\infty)$ such that for all $x_0\in\HH^\gamma$ one has the moment bounds
\begin{equation}\label{eq:momentsX-Galerkin}
\underset{N\in\N}\sup~\bigl(\E_{P_Nx_0}[\underset{0\le t\le T}\sup~\|X_N(t)\|_\gamma^m]\bigr)^{\frac1m} \le C_{\gamma,m}(T)(1+\|x_0\|_\gamma^{q}),
\end{equation}
where the integer $q\in\N$ is given by Proposition~\ref{propo:wellposed} (in particular, it does not depend on the spatial approximation parameter $N\in\N$). Similarly, for all $x_0\in\HH^\gamma$ and $0\le t_1\le t_2\le T$, one has the increment bounds
\begin{equation}\label{eq:incrementsX-Galerkin}
\underset{N\in\N}\sup~\bigl(\E_{P_Nx_0}[\|X_N(t_2)-X_N(t_1)\|^m]\bigr)^{\frac1m} \le C_{\gamma,m}(T)|t_2-t_1|^\frac{\gamma}{4}(1+\|x_0\|_\gamma^{q}).
\end{equation}
Second, the strong error between $X_N(T)$ and $X(T)$ converges to $0$ when $N\to\infty$, with order $\Gamma/2$ with respect to $\lambda_N$, in the following sense, see~\cite[Theorem 1]{CHS21} and~\cite[Theorem~3.1]{CQW22}: there exists $q\in\N$ such that, for all $T\in(0,\infty)$, $m\in\N$ and $\gamma\in(\Gamma_0,\Gamma)$, there exists $C_{\gamma,m}(T)\in(0,\infty)$ such that for all $x_0\in\HH^\gamma$ and all $N\in\N$ one has the strong error estimate
\begin{equation}\label{eq:strongerrorGalerkin}
\bigl(\E[\|X_N(T)-X(T)\|^m]\bigr)^{\frac1m}\le C_{\gamma,m}(T)(1+\|x_0\|_\gamma^{q})\lambda_N^{-\frac{\gamma}{2}},
\end{equation}
with initial values $X(0)=x_0$ and $X_N(0)=P_Nx_0$.

We are now in position to state our first main result. Recall that $\vvvert\varphi\vvvert_2$ is given by~\eqref{eq:varphi-notation}.
\begin{theo}\label{theo:Galerkin}
Let Assumption~\ref{ass:Q} be satisfied. There exists $q\in\N$ such that the following holds. For all $T\in(0,\infty)$ and $\gamma\in(\Gamma_0,\Gamma)$, there exists $C_{\gamma}(T)\in(0,\infty)$ such that for all functions $\varphi:\HH\to\R$ of class $\mathcal{C}^2$ with bounded first and second order derivatives, $x_0\in\HH^\gamma$ and $N\in\N$, one has the weak error estimate
\begin{equation}\label{eq:weakerrorGalerkin}
\big|\E[\varphi(X_N(T))]-\E[\varphi(X(T))]\big|\le C_{\gamma}(T)(1+\|x_0\|_\gamma^{q})\vvvert\varphi\vvvert_2\lambda_N^{-\gamma},
\end{equation}
with initial values $X(0)=x_0$ and $X_N(0)=P_Nx_0$.
\end{theo}

The proof of Theorem~\ref{theo:Galerkin} is postponed to Section~\ref{sec:proofGalerkin}, however let us describe the most important arguments. Since the function $\varphi$ is globally Lipschitz continuous, owing to the strong error estimate~\eqref{eq:strongerrorGalerkin}, one has
\[
\big|\E[\varphi(X_N(T))]-\E[\varphi(X(T))]\big|=\underset{M\to\infty}\lim~\big|\E[\varphi(X_N(T))]-\E[\varphi(X_M(T))]\big|,
\]
therefore it suffices to prove the following weak error estimate:
\begin{equation}\label{eq:weakerrorGalerkinNM}
\underset{M\ge N}\sup~\big|\E[\varphi(X_N(T))]-\E[\varphi(X_M(T))]\big|\le C_{\gamma}(T)(1+\|x_0\|_\gamma^{q})\vvvert\varphi\vvvert_2\lambda_N^{-\gamma},
\end{equation}
The weak error appearing in the left-hand side of the inequality ~\eqref{eq:weakerrorGalerkinNM} can be written as
\begin{equation}\label{eq:decompweakerrorGalerkin}
\E[\varphi(X_N(T))]-\E[\varphi(X_M(T))]=\E[u_M(0,X_N(T))]-\E[u_M(T,X_M(0))],
\end{equation}
where $M\ge N$, and the function $u_M$ is defined by
\begin{equation}\label{eq:uM_results}
u_M(t,x)=\E_x[\varphi(X_M(t))]
\end{equation}
for all $t\ge 0$ and $x\in\HH_M$, see equation~\eqref{eq:uM}. Section~\ref{sec:Kolmogorov} below is devoted to the analysis of the regularity properties of the function $u_M$, which is the solution of a Kolmogorov equation associated with the stochastic evolution equation~\eqref{eq:Galerkin} with $N=M$. The delicate point of the analysis is to obtain suitable estimates on the first and second order derivatives of $u_M$ with respect to the variable $x$ which are uniform with respect to $M\in\N$: we refer to Theorem~\ref{theo:Kolmogorov}  and Section~\ref{sec:Kolmogorov} below for such results. The use of It\^o's formula and of the Kolmogorov equation and the application of the suitable regularity estimates for $u_M$ then provide the weak error estimate~\eqref{eq:weakerrorGalerkinNM}, see the details in Section~\ref{sec:proofGalerkin}. The strategy is standard in the literature, however one needs to perform the analysis in Section~\ref{sec:Kolmogorov} which is the main novelty in this work on numerical methods for the stochastic Cahn--Hilliard equation.

\subsection{Full discretization}\label{sec:full}

Let us now introduce the
full discretization which is performed using a tamed exponential Euler scheme, combined with the spectral Galerkin method described in Section~\ref{sec:Galerkin} (with discretization parameter still denoted by $N\in\N$). Let $\Delta t$ denote the time-step size. Without loss of generality, assume that $\Delta t=T/K$ where $T\in(0,\infty)$ is a fixed positive real number, and $K\in\N$ is an integer. For all $k\in\{0,\ldots,K\}$, set $t_k=k\Delta t$. 
The tamed exponential Euler fully discrete scheme we consider is defined as follows: for all $k=0,\ldots,K-1$, set
\begin{equation}\label{eq:tamed}
X_{N,k+1}
=
e^{-\Delta tA^2} X_{N,k}
+( A^2 )^{-1} ( I - e^{-\Delta tA^2} ) \frac {-A P_N F (X_{N,k} ) } { 1 + \Delta t \| P_N F(X_{N, k})\|}
+ e^{-\Delta tA^2} P_N \Delta W_k^Q,
\end{equation}
with initial value $X_{N,0}=P_Nx_0=X_N(0)$, and where
\[
\Delta W_k^Q=W^Q(t_{k+1})-W^Q(t_k).
\]
The taming procedure consists in introducing the factor $\frac{1}{ 1 + \Delta t \| P_N F(X_{N, k})\|}$ in the right-hand side of~\eqref{eq:tamed}. To understand the construction of the scheme, note that one has
\[
( A^2 )^{-1} ( I - e^{-\Delta tA^2} )=\int_{t_k}^{t_{k+1}}e^{-(t_{k+1}-t)A^2}dt.
\]
Before stating Theorem~\ref{theo:tamed}, it is worth mentioning two auxiliary results. First, one has moment bounds for $X_{N,k}$, which are uniform with respect to $N$ and $\Delta t$. Note that the scheme~\eqref{eq:tamed} is explicit, therefore the taming procedure plays a crucial role. Indeed, the moment bounds are not expected to hold if one employs a standard exponential Euler scheme.
\begin{theo}\label{theo:momentbounds-tamed}
Let Assumption~\ref{ass:Q} be satisfied. For all $T\in(0,\infty)$, $m\in\{1,\ldots\}$ and $\gamma\in(\Gamma_0,\Gamma)$, there exists $q\in\N$ and $C_{\gamma,m}(T)\in(0,\infty)$ such that for all $x_0\in\HH^\gamma$, all $N\in\N$ and all $\Delta t=T/K$ with $K\in\N$, one has
\begin{equation}\label{eq:momentbounds-tamed}
\bigl(\E_{P_Nx_0}[\underset{0\le k\le K}\sup~\|X_{N,k}\|_\gamma^m]\bigr)^{\frac1m}\le C_{\gamma,m}(T)(1+\|x_0\|_\gamma^{q}).
\end{equation}
\end{theo}
The proof of Theorem~\ref{theo:momentbounds-tamed} is postponed to Section~\ref{sec:proofmomentsfull}. A similar result has been obtained in a recent preprint \cite[Corollary~3.6]{CQW22}, where a special class of $Q$-Wiener process is chosen to ensure mass preservation of the solution. However, the present work considers a general $Q$-Wiener process and a detailed and slightly different proof is given below for completeness.

Second, the strong error between $X_{N,K}$ and $X_N(T)=X_N(K\Delta t)$ converges to $0$ when $\Delta t\to 0$ ($K\to\infty$), with order $\Gamma/4$, uniformly with respect to $\lambda_N$, in the following sense, see~\cite[Theorem~3.8]{CQW22}: there exists $q\in\N$ such that, for all $T\in(0,\infty)$, $m\in\N$ and $\gamma\in(\Gamma_0,\Gamma)$, there exists $C_{\gamma,m}(T)\in(0,\infty)$ such that for all $x_0\in\HH^\gamma$, $K\in\N$ and all $N\in\N$ one has the strong error estimate
\begin{equation}\label{eq:strongerrortamed}
\bigl(\E[\|X_{N,K}-X_N(T)\|^m]\bigr)^{\frac1m}\le C_{\gamma,m}(T)(1+\|x_0\|_\gamma^{q})\Delta t^{\frac{\gamma}{4}}.
\end{equation}

We are now in position to state the second main result of this article.
\begin{theo}\label{theo:tamed}
Let Assumption~\ref{ass:Q} be satisfied. 
For all $T\in(0,\infty)$, $\gamma\in(\Gamma_0,\Gamma)$ and $x_0\in\HH^\gamma$, there exists $C_{\gamma}(T,\|x_0\|_\gamma)\in(0,\infty)$ such that for all $N\in \mathbb N^+$, all $\Delta t=T/K$ with $K\in\N^+$, and all functions $\varphi:\HH\to\R$ of class $\mathcal{C}^2$ with bounded first and second order derivatives, one has the weak error estimate below for the tamed exponential Euler scheme~\eqref{eq:tamed} ($\Delta t\to 0$ with fixed $N\in\N$):
\begin{equation}\label{eq:weakerrortamed}
\Big|\E[\varphi(X_{N,K})]-\E[\varphi(X_N(T))]\Big|\le C(\gamma,T,\|x_0\|_\gamma)\vvvert\varphi\vvvert_2\Delta t^{\frac{\gamma}{2}}.
\end{equation}
Furthermore, one obtains the following weak error estimate for the fully discrete scheme ($\Delta t\to 0$ and $N\to\infty$): under the same assumptions, one has 
\begin{align}\label{eq:weakerrorfull}
\Big|\E[\varphi(X_{N,K})]-\E[\varphi(X(T))]\Big|\le C(\gamma ,T,\|x_0\|_\gamma)\vvvert\varphi\vvvert_2(\Delta t^{\frac{\gamma}{2}}+\lambda_N^{-\gamma }).
\end{align}
\end{theo}
Note that the weak error estimate~\eqref{eq:weakerrorfull} for the fully discrete scheme is a straightforward consequence of~\eqref{eq:weakerrorGalerkin} and~\eqref{eq:weakerrortamed}. The proof of Theorem~\ref{theo:tamed} is postponed to Section~\ref{sec:prooffull}, however let us describe the most important arguments. The weak error is written as
\begin{equation}\label{eq:decompweakerrortamed}
\E[\varphi(X_{N,K})]-\E[\varphi(X_{N}(T))]=\E[u_N(0,X_{N,K})]-\E[u_N(T,X_{N,0})],
\end{equation}
where $u_N$ is the auxiliary function defined by~\eqref{eq:uM_results}. For any $N\in\N$ and $\Delta t$, let us introduce a continuous-time process $\bigl(\tilde{X}_N^{\Delta t}(t)\bigr)_{0\le t\le T}$, such that $\tilde{X}_N^{\Delta t}(t_k)=X_{N,k}$ for all $k\in\{0,\ldots,K\}$. The process is defined as follows: for all $t\in[t_k,t_{k+1}]$, set
\begin{equation}\label{eq:Xtilde}
\tilde{X}_{N}^{\Delta t}( t )
=
e^{-(t-t_k)A^2}X_{N,k}
+
\int_{t_k}^{t} e^{-(t-s)A^2}\frac {-A P_N F(X_{N,k})}{1 + \Delta t \| P_N F(X_{N,k} )\|} 
\text{d} s
+
e^{-(t-t_k)A^2}P_N ( W^Q (t) - W^Q (t_k) ).
\end{equation}
Using a telescoping sum argument and the auxiliary process, one then has
\begin{align*}
\E[\varphi(X_{N,K})]-\E[\varphi(X_{N}(T))]&=\sum_{k=0}^{K-1}\bigl(\E[u_N(T-t_{k+1},X_{N,k+1})]-\E[u_N(T-t_k,X_{N,k})]\bigr)\\
&=\sum_{k=0}^{K-1}\bigl(\E[u_N(T-t_{k+1},\tilde{X}_{N}^{\Delta t}(t_{k+1}))]-\E[u_N(T-t_k,X_{N}^{\Delta t}(t_k))]\bigr).
\end{align*}
Like for the proof of Theorem~\ref{theo:Galerkin}, the use of It\^o's formula and of the Kolmogorov equation and the application of the suitable regularity estimates for $u_N$ then provide the weak error estimate~\eqref{eq:weakerrortamed}, see the details in Section~\ref{sec:prooffull}. The strategy of the proof described above is standard in the literature, the main novelty in this work is the application of the results on the Kolmogorov equation from Section~\ref{sec:Kolmogorov}.

\begin{rem}
If the cubic nonlinearity $F$ is replaced by a globally Lipschitz continuous mapping of class $\mathcal{C}^2$, with bounded derivatives, one can prove variants of Theorems~\ref{theo:Galerkin} and~\ref{theo:tamed} using similar arguments, and obtain again that the weak order is twice the strong order. In addition, considering the case $F=0$ would show that the strong and weak orders of convergence are optimal. In fact, in the globally Lipschitz case the proof of the moment bounds (Theorem~\ref{theo:momentbounds-tamed}) and of the regularity properties for solutions of Kolmorogov equations (Theorem~\ref{theo:Kolmogorov}) would be substantially simplified. Even if, to the best of our knowledge, this situation has not been treated in the literature, the details are omitted and we only consider the cubic nonlinearity case.
\end{rem}

\section{Regularity results for solutions of Kolmogorov equations}\label{sec:Kolmogorov}

Let $\varphi:\HH\to\R$ be a function of class $\mathcal{C}^2$, with bounded first and second order derivatives. As explained in Sections~\ref{sec:Galerkin} and~\ref{sec:full}, the mappings $u_M$ defined by~\eqref{eq:uM_results}, which are solutions of Kolmogorov equations, play an important role in the proof of weak error estimates studied in this article. The objective of this section is to study the regularity properties of $u_M$ and to prove bounds which are uniform with respect to $M\in\N$. To the best of our knowledge, such results have not been obtained yet for stochastic Cahn--Hilliard equations, and are one of the main novelties of this article.

For all $M\in\N$, define the function $u_M:[0,\infty)\times \HH\to\R$ such that
\begin{equation}\label{eq:uM}
u_M(t,x)=\E_{x}[\varphi(X_M(t))]
\end{equation}
for all $t\ge 0$ and $x\in\HH_M$, where $\bigl(X_M(t)\bigr)_{t\ge 0}$ is the solution of the stochastic evolution equation~\eqref{eq:Galerkin} with initial value $X_M(0)=x$ and $N=M$. We recall that the notation $\E_x[\cdot]$ means that $X_M(0)=x$. For all $M\in\N$ and all $t\ge 0$, the mapping $u_M(t,\cdot)$ is of class $\mathcal{C}^2$ on $\HH_M$, and for all $x,h,h_1,h_2\in\HH_M$, one has
\begin{equation}\label{eq:DuMD2uM}
\left\lbrace
\begin{aligned}
Du_M(t,x).h&=\E_x[D\varphi(X_M(t)).\eta_M^h(t)]\\
D^2u_M(t,x).(h_1,h_2)&=\E_x[D\varphi(X_M(t)).\zeta_M^{h_1,h_2}(t)]+\E_x[D^2\varphi(X_M(t)).(\eta_M^{h_1}(t),\eta_M^{h_2}(t))],
\end{aligned}
\right.
\end{equation}
where the processes $\bigl(\eta_M^h(t)\bigr)_{t\ge 0}$ and $\bigl(\zeta_M^{h_1,h_2}(t)\bigr)_{t\ge 0}$ are the solutions of the first variational equation
\begin{equation}\label{eq:eta}
\frac{d\eta_M^h(t)}{dt}=-A^2\eta_M^h(t)-AP_M\bigl(F'(X_M(t))\eta_M^h(t)\bigr),
\end{equation}
with initial value $\eta_M^h(0)=h$, and of the second variational equation
\begin{equation}\label{eq:zeta}
\frac{d\zeta_M^{h_1,h_2}(t)}{dt}=-A^2\zeta_M^{h_1,h_2}(t)-AP_M\bigl(F'(X_M(t))\zeta_M^{h_1,h_2}(t)\bigr)-AP_M\bigl(F''(X_M(t))\eta_M^{h_1}(t)\eta_M^{h_2}(t)\bigr),
\end{equation}
with initial value $\zeta_M^{h,k}(0)=0$, respectively. Note that $\eta_M^{h}(t)$ and $\zeta_M^{h_1,h_2}(t)$ also depend on the initial value $x=X_M(0)$, however this dependence is omitted to simplify the notation. Since the noise is additive in~\eqref{eq:Galerkin}, $\eta_M^{h}(\cdot)$ and $\zeta_M^{h_1,h_2}(\cdot)$ are solutions of evolution equations which are not driven by a Wiener process.

For all $M\in\N$, the mapping $u_M:[0,T]\times \HH_M\to \R$ is the solution of the Kolmogorov equation (see for instance the monograph~\cite{Cer01})
\begin{equation}\label{eq:Kolmogorov}
\left\lbrace
\begin{aligned}
\partial_tu_M(t,x)&=\mathcal{L}_Mu_M(t,x),\quad (t,x)\in[0,T]\times\HH_M,\\
u_M(0,x)&=\varphi(x),\quad x\in\HH_M,
\end{aligned}
\right.
\end{equation}
where the infinitesimal generator $\mathcal{L}_M$ associated with the stochastic evolution equation~\eqref{eq:Galerkin} (with $N=M$) on $\HH_M$ is given by
\begin{equation}\label{eq:generator}
\mathcal{L}_M\phi(x)=D\phi(x).\bigl(-A^2x-AP_MF(x)\bigr)+\frac12\sum_{j\ge 0}D^2\phi(x).\bigl(P_MQ^{\frac12}e_j,P_MQ^{\frac12}e_j\bigr)
\end{equation}
for all $x\in\HH^M$, if $\phi:\HH_M\to\R$ is of class $\mathcal{C}^2$.

\subsection{Statement of the regularity estimates}

The main challenge is to prove estimates on $Du_M(t,x).h$ and $D^2u_M(t,x).(h_1,h_2)$, with upper bounds independent of $M\in\N$, and which are suitable to obtain the weak error estimates in Theorems~\ref{theo:Galerkin} and~\ref{theo:tamed}. The main result of this section is Theorem~\ref{theo:Kolmogorov} below.
\begin{theo}\label{theo:Kolmogorov}
Let Assumption~\ref{ass:Q} be satisfied. There exists $q\in\N$ such that the following holds. For all $\alpha\in[0,2)$ and all $\alpha_1,\alpha_2\in[0,2)$ such that $\alpha_1+\alpha_2<2$, all $\gamma\in(\Gamma_0,\Gamma)$ and all $T\in(0,\infty)$, there exist positive real numbers $C_{\alpha,\gamma}(T)\in(0,\infty)$ and $C_{\alpha_1,\alpha_2,\gamma}(T)\in(0,\infty)$ such that for all functions $\varphi:\HH\to\R$ of class $\mathcal{C}^2$ with bounded first and second order derivatives, all $M\in\N$, all $x,h,h_1,h_2\in\HH_M$ and all $t\in(0,T]$, one has
\begin{equation}\label{eq:theoKolmogorov}
\begin{aligned}
\frac{|Du_M(t,x).h|}{\bigl(1+\|x\|_\gamma^q\bigr)\vvvert\varphi\vvvert_2}&\le C_{\alpha,\gamma}(T)t^{-\frac{\alpha}{2}}(|A^{-\alpha}\PP h|+|\langle h,e_0\rangle |)\\
\frac{|D^2u_M(t,x).(h_1,h_2)|}{\bigl(1+\|x\|_\gamma^q\bigr)\vvvert\varphi\vvvert_2}&\le C_{\alpha_1,\alpha_2,\gamma}(T)t^{-\frac{\alpha_1+\alpha_2}{2}}(|A^{-\alpha_1}\PP h_1|+|\langle h_1,e_0\rangle |)
(|A^{-\alpha_2}\PP h_2|+|\langle h_2,e_0\rangle |).
\end{aligned}
\end{equation}
\end{theo}
Let us emphasize that the real numbers $C_{\alpha,\gamma}(T)$ and $C_{\alpha_1,\alpha_2,\gamma}(T)$ appearing in the right-hand side of~\eqref{eq:theoKolmogorov} do not depend on $M\in\N$. Taking $\alpha=\alpha_1=\alpha_2=0$ in Theorem~\ref{theo:Kolmogorov} justifies that $u_M(t,\cdot)$ is of class $\mathcal{C}^2$ for all $t\ge 0$, and gives bounds on the first and second order derivatives which are uniform with respect to $M\in\N$. However, for the proof of weak error estimates in Theorems~\ref{theo:Galerkin} and~\ref{theo:tamed}, one requires the inequalities~\eqref{eq:theoKolmogorov} with $\alpha>0$ and $\alpha_1+\alpha_2>0$, more precisely $\alpha$ and $\alpha_1$ need to be chosen arbitrarily close to $2$ (with $\alpha_2=0$) to obtain a weak order of convergence which is twice the strong order. The upper bounds are not uniform with respect to $t\in[0,T]$ since the initial value $u_M(0,\cdot)=\varphi(\PP_M\cdot)$ is not assumed to satisfy this type of regularity estimate. The inequalities~\eqref{eq:theoKolmogorov} thus exhibit a regularization effect, which is due to the smoothing property~\eqref{eq:smoothing-HH}. The condition $\gamma\in(\Gamma_0,\Gamma)$ implies the condition $\gamma\in(\frac{d}{2},\Gamma)\setminus\{\frac32\}$, which is required due to the application of the properties from Section~\ref{sec:auxineq}, and ensures the validity of moment bounds~\eqref{eq:momentsX-Galerkin}.

Theorem~\ref{theo:Kolmogorov} is a variant of results previously obtained for solutions of Kolmogorov associated with parabolic semilinear stochastic evolution equations: see~\cite{Bre14} (additive noise) and~\cite{BD18} (multiplicative noise) in the case of globally Lipschitz nonlinearities, and~\cite{BG18,CHS21a} (Allen--Cahn equation) and~\cite{CH19,Bre22} (dissipative polynomial nonlinearities). In the references above, the spatial approximation parameter $M\in\N$ is sometimes omitted even if an auxiliary spectral Galerkin approximation is applied. We use arguments similar to those employed in the aforementioned references in order first to prove Theorem~\ref{theo:Kolmogorov} and second to apply this result to prove weak error estimates. To the best of our knowledge, Theorem~\ref{theo:Kolmogorov} is the first result in the literature giving regularity results for solutions of Kolmogorov equations associated with stochastic Cahn--Hilliard type equations driven by additive noise. Note that the case of globally Lipschitz nonlinearities $F$ has not been treated in the literature, however one can obtain a version of Theorem~\ref{theo:Kolmogorov} using similar arguments with substantial technical simplifications. Compared with the parabolic semilinear case, note that the power in the singularity $t^{-\frac{\alpha}{2}}$ is  related to the smoothing inequality~\eqref{eq:smoothing-HH}, and is thus different from the singularity appearing in the parabolic case.

Before proceeding with the proof of Theorem~\ref{theo:Kolmogorov}, let us mention that the expression of the inequalities~\eqref{eq:theoKolmogorov} requires to write $h=\PP h+\langle h,e_0\rangle$ and to treat differently the two contributions. Indeed, since $-A$ is the Laplace operator with Neumann boundary conditions, one has $Ae_0=0$, whereas if Dirichlet boundary conditions were considered all the eigenvalues of $A$ would be bounded positive and bounded away from $0$. In the setting considered in this article, $A^{-\alpha}$ for $\alpha>0$ is defined only on the subspace $\hh=\PP(\HH)$, which results in more complicated expressions in the right-hand sides of~\eqref{eq:theoKolmogorov}.

To understand the dependence in the right-hand sides of the inequalities~\eqref{eq:theoKolmogorov} with respect to $\alpha,\alpha_1,\alpha_2$, it is relevant to consider the case $F=0$, and to assume that $Q$ and $A$ commute. Instead of $u_M$, consider the mappings $v_M$ defined by
\[
v_M(t,x)=\E[\varphi(e^{-tA^2}P_Mx+P_MZ(t))],
\]
for all $t\ge 0$ and $x\in\HH_M$, where we recall that $\bigl(Z(t)\bigr)_{t\ge 0}$ is the stochastic convolution defined by~\eqref{eq:Z}. It is then straightforward to check that for all $M\in\N$, $x,h,h_1,h_2\in \HH_M$ and $t\in[0,T]$, one has
\begin{align*}
Dv_M(t,x).h&=\E[D\varphi(e^{-tA^2}P_Mx+P_MZ(t)).e^{-tA^2}h]\\
D^2v_M(t,x).(h_1,h_2)&=\E[D^2\varphi(e^{-tA^2}P_Mx+P_MZ(t)).(e^{-tA^2}h_1,e^{-tA^2}h_2)].
\end{align*}
As a consequence, one obtains the inequalities
\begin{align*}
|Dv_M(t,x).h|&\le \vvvert\varphi\vvvert_2 \|e^{-tA^2}h\|\\
&\le C_\alpha \vvvert\varphi\vvvert_2 t^{-\frac{\alpha}{2}}(|A^{-\alpha}\PP h|+|\langle h,e_0\rangle|)\\
|D^2v_M(t,x).(h_1,h_2)|&\le \vvvert\varphi\vvvert_2 \|e^{-tA^2}h_1\|\|e^{-tA^2}h_2\|\\
&\le C_{\alpha_1,\alpha_2}\vvvert\varphi\vvvert_2 t^{-\frac{\alpha_1+\alpha_2}{2}}(|A^{-\alpha_1}\PP h_1|+|\langle h_1,e_0\rangle |)(|A^{-\alpha_2}\PP h_2|+|\langle h_2,e_0\rangle |)
\end{align*}
as straightforward applications of the smoothing inequality~\eqref{eq:smoothing-HH} associated with the semigroup $\bigl(e^{-tA^2}\bigr)_{t\ge 0}$. Note that in the inequalities for $Dv_M(t,x).h$ and $D^2v_M(t,x).(h_1,h_2)$ one may choose arbitrary $\alpha,\alpha_1,\alpha_2\ge 0$. The study above explains the meaning of the inequalities~\eqref{eq:theoKolmogorov} from Theorem~\ref{theo:Kolmogorov} and why they may be seen as optimal in the treatment of the parameters $\alpha,\alpha_1,\alpha_2$. The conditions $\alpha<2$ and $\alpha_1+\alpha_2<2$ appear due to the presence of a nonlinearity $F$, and the expression $1+\|x\|_\gamma^q$ appears due to the fact that $F$ has polynomial growth.

\subsection{Auxiliary results}

In order to prove Theorem~\ref{theo:Kolmogorov}, it is convenient to introduce the family $\bigl(\Pi_M(t,s)\bigr)_{t\ge s\ge 0}$ of random linear operators on $\HH_M$, which is defined as follows. For all $M\in\N$, $s\ge 0$ and $h\in\HH_M$, $t\in[s,\infty)\mapsto \Pi_M(t,s)h\in\HH_M=\eta_M^h(t,s)$ is the solution of the linear random evolution equation
\begin{equation}\label{eq:PiM}
\frac{d\eta_M^h(t,s)}{dt}=-A^2\eta_M^h(t,s)-AP_M\bigl(F'(X_M(t))\eta_M^h(t,s)\bigr),
\end{equation}
for $t\ge s$, with initial value $\eta_M^h(s,s)=h$. The linear operators $\Pi_M(t,s)$ also depend on the initial value $x\in\HH_M$ of $X_M$. This dependence is omitted to simplify the notation. These two parameter family of linear operators have already been used in \cite{BD18,BG20,CH19} in order to obtain regularity properties of solutions of Kolmorov equations for parabolic semilinear SPDEs.

First, note that the processes $\bigl(\eta_M^h(t)\bigr)_{t\ge 0}$ and $\bigl(\zeta^{h_1,h_2}(t)\bigr)_{t\ge 0}$ which are solutions of~\eqref{eq:eta} and~\eqref{eq:zeta} respectively can be expressed using the linear operators $\Pi_M(t,s)$.
\begin{lemma}\label{lem:Kolmogorov1}
For all $M\in\N$, $x,h,h_1,h_2\in \HH_M$ and $t\ge 0$, one has
\[
\left\lbrace
\begin{aligned}
\eta_M^h(t)&=\Pi_M(t,0)h\\
\zeta_M^{h_1,h_2}(t)&=-\int_0^t\Pi_M(t,s)AP_M\bigl(F''(X_M(s))(\Pi_M(s,0)h_1,\Pi_M(s,0)h_2)\bigr)ds.
\end{aligned}
\right.
\]
\end{lemma}

\begin{proof}[Proof of Lemma~\ref{lem:Kolmogorov1}]
The identity $\eta_M^h(t)=\Pi_M(t,0)h$ is a straightforward application of the definition~\eqref{eq:PiM} of the linear operators $\bigl(\Pi_M(t,0)\bigr)_{t\ge 0}$ (with $s=0$).

To check that the second identity holds, it suffices to write the evolution equation satisfied by the right-hand side in terms of $t\ge 0$: it is also solution of~\eqref{eq:zeta} (using the first identity) and vanishes when $t=0$, therefore it is equal to $\zeta_M^{h_1,h_2}(t)$ at all times $t\ge 0$.

The proof of Lemma~\ref{lem:Kolmogorov1} is completed.
\end{proof}

The main technical result in this section, which is instrumental in the proof of Theorem~\ref{theo:Kolmogorov}, is Lemma~\ref{lem:Kolmogorov2}.
\begin{lemma}\label{lem:Kolmogorov2}
For all $\alpha\in[0,2)$, $\gamma\in(\Gamma_0,\Gamma)$ and all $T\in(0,\infty)$, there exists a real number $C_{\alpha,\gamma}(T)\in(0,\infty)$ such that for all $M\in\N$, all $x,h\in\HH_M$, all $0\le s<t\le T$, one has almost surely
\begin{equation}\label{eq:lemK2}
\|\Pi_M(t,s)h\|\le C_{\alpha,\gamma}(T)\bigl(1+\underset{r\in[s,T]}\sup~\|X_M(r)\|_\gamma^5\bigr)(t-s)^{-\frac{\alpha}{2}}(|A^{-\alpha}\PP h|+|\langle h,e_0\rangle|).
\end{equation}
\end{lemma}
Note that the real numbers $C_{\alpha,\gamma}(T)$ appearing in the statement of Lemma~\ref{lem:Kolmogorov2} are deterministic. Observe that if $F=0$, one has $\Pi_M(t,0)=e^{-tA^2}P_M$ for all $t\ge 0$, therefore the inequality~\eqref{eq:lemK2} may be interpreted as a variant of the smoothing inequality, this is consistent with the discussion above concerning the interpretation of Theorem~\ref{theo:Kolmogorov}.

The proof of Lemma~\ref{lem:Kolmogorov2} is the most delicate contribution of this section. The proof requires two steps: the inequality~\eqref{eq:lemK2} is proved first for $\alpha\in[0,1]$, and then for $\alpha\in(1,2)$. A key argument of the proof is to set
\begin{equation}\label{eq:tildeeta}
\tilde{\eta}_M^h(t,s)=\eta_M^h(t,s)-e^{-(t-s)A^2}h
\end{equation}
for all $M\in\N$, $x,h\in\HH_M$ and $t\ge s\ge 0$. Observe that $(I-\PP)\tilde{\eta}_M^h(t,s)=0$ for all $t\ge s$, and that $\tilde{\eta}_M^h(s,s)=0$. In addition, for all $t\ge s$ one has
\begin{equation}\label{eq:tildeeta2}
\frac{d\tilde{\eta}_M^h(t,s)}{dt}+A^2\tilde{\eta}_M^h(t,s)+AP_M\bigl(F'(X_M(t))\tilde{\eta}_M^h(t,s)\bigr)=-AP_M\bigl(F'(X_M(t))(e^{-(t-s)A^2}h)\bigr).
\end{equation}
A variant of Lemma~\ref{lem:Kolmogorov1} yields the following identity: for all $t\ge s$, one has
\begin{equation}\label{eq:tildeeta3}
\tilde{\eta}_M^h(t,s)=-\int_s^t \Pi_M(r,s)\Bigl(AP_M\bigl(F'(X_M(r))(e^{-(r-s)A^2}h)\bigr)\Bigr)dr.
\end{equation}
Since for all $t\ge s$ one has
\[
\Pi_M(t,s)h=\eta_M^h(t,s)=e^{-(t-s)A^2}h+\tilde{\eta}_M^h(t,s),
\]
owing to the smoothing inequality~\eqref{eq:smoothing-HH} it suffices to prove upper bounds for $\tilde{\eta}_M^h(t,s)$. When $\alpha\in[0,1]$ (first step), the proof of these upper bounds requires to exploit two energy estimates, in the $\|A^{-\frac12}\cdot\|$ and $\|\cdot\|$ norms, using the evolution equation~\eqref{eq:tildeeta2}. When $\alpha\in(1,2)$ (second step), the inequality~\eqref{eq:lemK2} is obtained by combining the identity~\eqref{eq:tildeeta3} and the inequality obtained when $\alpha=1$ (first step).

\begin{proof}[Proof of Lemma~\ref{lem:Kolmogorov2}]
In this proof, the value of $s\ge 0$ is fixed. Let $M\in\N$, and let $x,h\in\HH_M$ also be fixed.

Let us first prove the inequality~\eqref{eq:lemK2} when $\alpha\in[0,1]$. Recall that $\tilde{\eta}_M^h(t,s)\in\hh$ for all $t\ge s$. Due to~\eqref{eq:tildeeta2}, one obtains two energy estimates~\eqref{eq:prooflemK2-energy1} and~\eqref{eq:prooflemK2-energy2} in the $\|A^{-\frac12}\cdot\|$ and $\|\cdot\|$ norms respectively: for all $t\ge s$, one has
\begin{equation}\label{eq:prooflemK2-energy1}
\begin{aligned}
\frac{1}{2}\frac{d\|A^{-\frac12}\tilde{\eta}_M^h(t,s)\|^2}{dt}&+\|A^{\frac12}\tilde{\eta}_M^h(t,s)\|^2+\langle F'(X_M(t))\tilde{\eta}_M^h(t,s),\tilde{\eta}_M^h(t,s)\rangle\\
&=-\langle F'(X_M(t))(e^{-(t-s)A^2}h),\tilde{\eta}_M^h(t,s)\rangle
\end{aligned}
\end{equation}
and
\begin{equation}\label{eq:prooflemK2-energy2}
\begin{aligned}
\frac12\frac{d\|\tilde{\eta}_M^h(t,s)\|^2}{dt}&+\|A\tilde{\eta}_M^h(t,s)\|^2+\langle A(F'(X_M(t))\tilde{\eta}_M^h(t,s)),\tilde{\eta}_M^h(t,s)\rangle\\
&=-\langle A\bigl(F'(X_M(t))(e^{-(t-s)A^2}h)\bigr),\tilde{\eta}_M^h(t,s)\rangle.
\end{aligned}
\end{equation}
To exploit the first energy estimate~\eqref{eq:prooflemK2-energy1}, observe that for all $t\ge s$, one has
\begin{align*}
\langle F'(X_M(t))\tilde{\eta}_M^h(t,s),\tilde{\eta}_M^h(t,s)\rangle&=3\langle X_M(t)^2\tilde{\eta}_M^h(t,s),\tilde{\eta}_M^h(t,s)\rangle-\langle \tilde{\eta}_M^h(t,s),\tilde{\eta}_M^h(t,s)\rangle\\
&=3\|X_M(t)\tilde{\eta}_M^h(t,s)\|^2-\langle\tilde{\eta}_M^h(t,s),\tilde{\eta}_M^h(t,s)\rangle.
\end{align*}
As a consequence, for all $t\ge s$, one obtains
\begin{align*}
\frac{1}{2}\frac{d\|A^{-\frac12}\tilde{\eta}_M^h(t,s)\|^2}{dt}&+\|A^{\frac12}\tilde{\eta}_M^h(t,s)\|^2+3\|X_M(t)\tilde{\eta}_M^h(t,s)\|^2&\\
&\le \|\tilde{\eta}_M^h(t,s)\|^2+\|\tilde{\eta}_M^h(t,s)\|\|F'(X_M(t))(e^{-(t-s)A^2}h)\|\\
&\le \frac{3}{2}\|\tilde{\eta}_M^h(t,s)\|^2+C\|F'(X_M(t))(e^{-(t-s)A^2}h)\|^2
\end{align*}
using the Cauchy--Schwarz and Young inequalities, and the inequality~\eqref{eq:ineq2} in the last step. Using the inequality~\eqref{eq:ineq1} and the property $\tilde{\eta}_M^h(t,s)\in\hh$ for all $t\ge s$, one obtains
\[
\|\tilde{\eta}_M^h(t,s)\|^2\le \|A^{-\frac{1}{2}}\tilde{ \eta}_M^h(t,s)\| \|A^{\frac12}\tilde{\eta}_M^h(t,s)\|\le \frac12\|A^{-\frac12}\tilde{\eta}_M^h(t,s)\|^2+\frac12\|A^{\frac12}\tilde{\eta}_M^h(t,s)\|^2.
\]
Applying Gronwall's lemma, one obtains the following inequality: for all $T\in(0,\infty)$, there exists a deterministic real number $C(T)\in(0,\infty)$ such that for all $0\le s\le t\le T$ one has
\begin{equation}\label{eq:prooflemK2-1}
\begin{aligned}
\|A^{-\frac12}\tilde{\eta}_M^h(t,s)\|^2&+\int_s^t\|A^{\frac12}\tilde{\eta}_M^h(r,s)\|^2dr+\int_s^t\|X_M(r)\tilde{\eta}_M^h(r,s)\|^2dr\\
&\le C(T)\int_s^t\|F'(X_M(r))(e^{-(r-s)A^2}h)\|^2dr.
\end{aligned}
\end{equation}
The second energy estimate~\eqref{eq:prooflemK2-energy2} is treated as follows: using the facts that $A$ is self-adjoint and that $Ae_0=0$ and the Cauchy--Schwarz and Young inequalities, for all $t\ge s$, one has
\begin{align*}
-\langle A\bigl(F'(X_M(t))\tilde{\eta}_M^h(t,s)\bigr)&,\tilde{\eta}_M^h(t,s)\rangle=-\langle F'(X_M(t))\tilde{\eta}_M^h(t,s),A\tilde{\eta}_M^h(t,s)\rangle\\
&=-3\langle X_M(t)^2\tilde{\eta}_M^h(t,s),A\tilde{\eta}_M^h(t,s)\rangle+\langle \tilde{\eta}_M^h(t,s),A\tilde{\eta}_M^h(t,s)\rangle\\
&\le 3\|X_M(t)^2\tilde{\eta}_M^h(t,s)\| \|A\tilde{\eta}_M^h(t,s)\|+\|A^{\frac12}\tilde{\eta}_M^h(t,s)\|^2\\
&\le \frac{9}{2}\|X_M(t)^2\tilde{\eta}_M^h(t,s)\|^2+\frac12\|A\tilde{\eta}_M^h(t,s)\|^2+\|A^{\frac12}\tilde{ \eta}_M^h(t,s)\|^2.
\end{align*}
In addition, one has
\begin{align*}
-\langle A\bigl(F'(X_M(t))(e^{-(t-s)A^2}h)\bigr),\tilde{\eta}_M^h(t,s)\rangle&=-\langle \bigl(F'(X_M(t))(e^{-(t-s)A^2}h)\bigr),A\tilde{\eta}_M^h(t,s)\rangle\\
&\le \|A\tilde{\eta}_M^h(t,s)\|\|F'(X_M(t))(e^{-(t-s)A^2}h)\|\\
&\le \frac12\|A\tilde{\eta}_M^h(t,s)\|^2+C\|F'(X_M(t))(e^{-(t-s)A^2}h)\|^2.
\end{align*}

Using the Sobolev inequality~\eqref{eq:Sobolev-norm_alpha} under the condition $\gamma\in(\frac{d}{2},\Gamma)\setminus\{\frac32\}$,
one obtains
\[
\|X_M(t)^2\tilde{\eta}_M^h(t,s)\|^2\le \|X_M(t)\|_{L^\infty}^2\|X_M(t)\tilde{\eta}_M^h(t,s)\|^2\le C_\gamma \|X_M(t)\|_\gamma^2 \|X_M(t)\tilde{\eta}_M^h(t,s)\|^2.
\]
Therefore, one obtains for all $t\ge s$
\begin{align*}
\frac12\frac{d\|\tilde{\eta}_M^h(t,s)\|^2}{dt}&\le C_\gamma \|X_M(t)\|_\gamma^2 \|X_M(t)\tilde{\eta}_M^h(t,s)\|^2
+
\|A^{\frac12}\tilde{ \eta}_M^h(t,s)\|^2
\\
& +
C\|F'(X_M(t))(e^{-(t-s)A^2}h)\|^2
.
\end{align*}
Applying Gronwall's lemma, one obtains the following inequality: for all $T\in(0,\infty)$ and $\gamma\in(\Gamma_0,\Gamma)$, there exists a deterministic real number $C_\gamma(T)\in(0,\infty)$ such that for all $0\le s\le t\le T$, one has
\begin{equation}\label{eq:prooflemK2-2}
\begin{aligned}
\|\tilde{\eta}_M^h(t,s)\|^2&\le C_\gamma(T)\Bigl(\underset{r\in[s,T]}\sup~\|X_M(r)\|_\gamma^{2}\int_s^t\|X_M(r)\tilde{\eta}_M^h(r,s)\|^2dr+\int_s^t\|A^{\frac12}\tilde{\eta}_M^h(r,s)\|^2dr\Bigr)\\
&+C_\gamma(T)\int_s^t\|F'(X_M(r))(e^{-(r-s)A^2}h)\|^2dr.
\end{aligned}
\end{equation}
Using the smoothing inequality~\eqref{eq:smoothing-HH} and the Sobolev inequality~\eqref{eq:Sobolev-norm_alpha}, one obtains the upper bound
\begin{align*}
\int_s^t\|F'(X_M(r))(e^{-(r-s)A^2}h)\|^2dr&\le C_\gamma\bigl(1+\underset{r\in[s,T]}\sup~\|X_M(r)\|_\gamma^4\bigr)\int_s^t \|e^{-(r-s)A^2}h\|^2dr\\
&\le C_\gamma\bigl(1+\underset{r\in[s,T]}\sup~\|X_M(r)\|_\gamma^4\bigr)(\|A^{-1}\PP h\|^2+\langle h,e_0\rangle^2)\\
&\le C_\gamma\bigl(1+\underset{r\in[s,T]}\sup~\|X_M(r)\|_\gamma^4\bigr)(\|A^{-\alpha}\PP h\|^2+\langle h,e_0\rangle^2),
\end{align*}
using the condition $\alpha\in[0,1]$ in the last step. Therefore, combining the inequalities~\eqref{eq:prooflemK2-1} and~\eqref{eq:prooflemK2-2} with the upper bound above, one obtains the following inequality: for all $T\in(0,\infty)$, $\alpha\in[0,1]$ and $\gamma\in(\Gamma_0,\Gamma)$, there exists a deterministic real number $C_\gamma(T)\in(0,\infty)$ such that for all $0\le s\le t\le T$ one has
\begin{equation}\label{eq:prooflemK2-alpha1aux}
\|\tilde{\eta}_M^h(t,s)\|^2\le C_\gamma(T)\bigl(1+\underset{r\in[s,T]}\sup~\|X_M(r)\|_\gamma^6\bigr)\bigl(\|A^{-\alpha}\PP h\|^2+\langle h,e_0\rangle^2\bigr).
\end{equation}
Since $\eta_M^h(t,s)=\tilde{\eta}_M^h(t,s)+e^{-(t-s)A^2}h$, combining the inequality~\eqref{eq:prooflemK2-alpha1aux} and the smoothing inequality~\eqref{eq:smoothing-HH} then provides the inequality~\eqref{eq:lemK2} when $\alpha\in[0,1]$ (with a power $3$ instead of $5$ for $\|X_M(r)\|_\gamma$ in the right-hand side): for all $T\in(0,\infty)$, $\alpha\in(0,1]$ and $\gamma\in(\Gamma_0,\Gamma)$, there exists a deterministic real number $C_{\alpha,\gamma}(T)\in(0,\infty)$ such that for all $0\le s<t\le T$ one has
\begin{equation}\label{eq:prooflemK2-alpha1}
\|\eta_M^h(t,s)\|\le C_\gamma(T)(t-s)^{-\frac{\alpha}{2}}\bigl(1+\underset{r\in[s,T]}\sup~\|X_M(r)\|_\gamma^3\bigr)\bigl(\|A^{-\alpha}\PP h\|+|\langle h,e_0\rangle|\bigr).
\end{equation}

It remains to deal with the case $\alpha\in(1,2)$. Using the fact that $\tilde{\eta}_M^h(t,s)$ defined by~\eqref{eq:tildeeta} is solution of~\eqref{eq:tildeeta2} and the definition~\eqref{eq:PiM} of the random linear operators $\Pi_M(t,s)$, one obtains the expression~\eqref{eq:tildeeta3} for $\tilde{\eta}_M^h(t,s)$. Using the inequality~\eqref{eq:prooflemK2-alpha1} with $\alpha=1$, one then obtains for all $t\ge s$
\begin{align*}
\|\tilde{\eta}_M^h(t,s)\|&\le C_{1,\gamma}(T)\bigl(1+ \underset{r\in[s,T]}\sup~\|X_M(r)\|_\gamma^3 \bigr)\int_s^t (t-s)^{-\frac12}\|\bigl(F'(X_M(r))e^{-(r-s)A^2}h\bigr)\|dr\\
&\le C_{1,\gamma}(T)\bigl(1+
\underset{r\in[s,T]}\sup~\|X_M(r)\|_\gamma^5\bigr)\int_s^t (t-s)^{-\frac12}\|e^{-(r-s)A^2}h\|dr.
\end{align*}
Using the smoothing inequality~\eqref{eq:smoothing-HH} and the condition $\alpha\in(1,2)$, one obtains the inequalities
\begin{align*}
\int_s^t (t-s)^{-\frac12}\|e^{-(r-s)A^2}h\|dr&\le \int_s^t (t-s)^{-\frac12}(r-s)^{-\frac{\alpha}{2}}dr (\|A^{-\alpha}\PP h\|+\langle  h,e_0 \rangle)\\
&\le C_\alpha(T)\bigl(\|A^{-\alpha}\PP h\|+|\langle  h,e_0 \rangle|\bigr).
\end{align*}
As a consequence, using the identity $\eta_M^h(t,s)=\tilde{\eta}_M^h(t,s)+e^{-(t-s)A^2}h$, the inequality above and the smoothing inequality, one obtains 
the inequality~\eqref{eq:lemK2} when $\alpha\in(1,2)$: for all $T\in(0,\infty)$, $\alpha\in(1,2)$ and $\gamma\in(\Gamma_0,\Gamma)$, there exists a deterministic real number $C_{\alpha,\gamma}(T)\in(0,\infty)$ such that for all $0\le s<t\le T$ one has
\begin{equation}\label{eq:prooflemK2-alpha2}
\|\eta_M^h(t,s)\|\le C_\gamma(T)(t-s)^{-\frac{\alpha}{2}}\bigl(1+\underset{r\in[s,T]}\sup~\|X_M(r)\|_\gamma^5\bigr)\bigl(\|A^{-\alpha}\PP h\|+|\langle h,e_0\rangle|\bigr).
\end{equation}
Gathering the inequalities~\eqref{eq:prooflemK2-alpha1} (first step, $\alpha\in[0,1]$) and~\eqref{eq:prooflemK2-alpha2} (second step, $\alpha\in(1,2]$) then concludes the proof of Lemma~\ref{lem:Kolmogorov2}.
\end{proof}

\subsection{Proof of Theorem~\ref{theo:Kolmogorov}}

The proof of Theorem~\ref{theo:Kolmogorov} is a straightforward consequence of Lemmas~\ref{lem:Kolmogorov1} and~\ref{lem:Kolmogorov2}. The value of the integer $q$ appearing in the regularity estimates~\eqref{eq:theoKolmogorov} is not important to obtain the main results of this article: it may be possible to identify values of $q$ depending on the right-hand side of the moment bounds on the solution and of the inequality~\eqref{eq:lemK2} from Lemma~\ref{lem:Kolmogorov2}, however this is omitted to simplify the notation.

\begin{proof}[Proof of Theorem~\ref{theo:Kolmogorov}]
Let us first prove the upper bound for $Du_M(t,x).h$. Using the first equality in~\eqref{eq:DuMD2uM} and the first identity of Lemma~\ref{lem:Kolmogorov1}, one obtains the following inequality: for all $M\in\N$, $x,h\in\HH_M$ and $t\in[0,T]$, one has
\begin{align*}
\big|Du_M(t,x).h\big|&=\big|\E_x[D\varphi(X_M(t)).\eta_M^h(t)]\big|\\
&=\big|\E_x[D\varphi(X_M(t)).\Pi_M(t,0)h]\big|\\
&\le \vvvert\varphi\vvvert_2 \E_x[\|\Pi_M(t,0)h\|].
\end{align*}
Using the inequality~\eqref{eq:lemK2} from Lemma~\ref{lem:Kolmogorov2} and the moment bounds~\eqref{eq:momentsX-Galerkin}, one then obtains the following inequalities: for all $T\in(0,\infty)$, $\alpha\in[0,2)$ and $\gamma\in(\Gamma_0,\Gamma)$, there exists a real number $C_{\alpha,\gamma}(T)\in(0,\infty)$ such that for all $t\in(0,T]$ one has
\begin{align*}
\frac{\big|Du_M(t,x).h\big|}{(\|A^{-\alpha}\PP h\|+|\langle h, e_0\rangle| )}&\le C_{\alpha,\gamma}(T)\vvvert\varphi\vvvert_2 t^{-\frac{\alpha}{2}}\bigl(1+\E_x[\underset{r\in[0,T]}\sup~\|X_M(r)\|_\gamma^5]\bigr)\\
&\le C_{\alpha,\gamma}(T)\vvvert\varphi\vvvert_2 t^{-\frac{\alpha}{2}}\bigl(1+\|x\|_\gamma^{5q}\bigr).
\end{align*}
This yields the first inequality in~\eqref{eq:theoKolmogorov}. It remains to prove the second inequality. Recall that $D^2u_M(t,x).(h_1,h_2)$ is written as the sum of two terms, owing to the second equality in~\eqref{eq:DuMD2uM}. On the one hand, using the first identity of Lemma~\ref{lem:Kolmogorov1}, one obtains the following inequality: for all $M\in\N$, $x,h_1,h_2\in\HH_M$ and $t\in[0,T]$, one has
\begin{align*}
\big|\E_x[D^2\varphi(X_M(t)).(\eta_M^{h_1}(t),\eta_M^{h_2}(t))]\big|&=\big|\E_x[D^2\varphi(X_M(t)).(\Pi_M(t,0)h_1,\Pi_M(t,0)h_2)]\big|\\
&\le \vvvert\varphi\vvvert_2\bigl(\E_x[\|\Pi_M(t,0)h_1\|^2]\bigr)^{\frac12}\bigl(\E_x[\|\Pi_M(t,0)h_2\|^2]\bigr)^{\frac12},
\end{align*}
using the Cauchy--Schwarz inequality. Using the inequality~\eqref{eq:lemK2} from Lemma~\ref{lem:Kolmogorov2} and the moment bounds~\eqref{eq:momentsX-Galerkin}, one then obtains the following inequalities: for all $T\in(0,\infty)$, $\alpha_1,\alpha_2\in[0,2)$, with $\alpha_1+\alpha_2<2$ and $\gamma\in(\Gamma_0,\Gamma)$, there exists a real number $C_{\alpha_1,\alpha_2,\gamma}(T)\in(0,\infty)$ such that for all $t\in(0,T]$ one has
\begin{align*}
\frac{\big|\E_x[D^2\varphi(X_M(t)).(\eta_M^{h_1}(t),\eta_M^{h_2}(t))]\big|}{(\|A^{-\alpha_1}\PP h_1\|+|\langle h_1, e_0\rangle| )(\|A^{-\alpha_2}\PP h_2\|+|\langle h_2, e_0\rangle| )}\le C_{\alpha_1,\alpha_2,\gamma}(T)t^{-\frac{\alpha_1+\alpha_2}{2}}\vvvert\varphi\vvvert_2\bigl(1+\|x\|_\gamma^{10q}\bigr).
\end{align*}
On the other hand, using the second identity of Lemma~\ref{lem:Kolmogorov1}, one obtains the following inequality: for all $M\in\N$, $x,h_1,h_2\in\HH_M$ and $t\in[0,T]$, one has
\begin{align*}
\big|\E_x[D\varphi(X_M(t))&.\zeta_M^{h_1,h_2}(t)]\big|\le \vvvert\varphi\vvvert_2 \E_x[\|\zeta_M^{h_1,h_2}(t)\|]\\
&\le \vvvert\varphi\vvvert_2\int_0^t\E_x[\|\Pi_M(t,s)AP_M\bigl(F''(X_M(s))(\Pi_M(s,0)h_1,\Pi_M(s,0)h_2)\bigr)\|]ds.
\end{align*}
Recall that $\gamma\in(\frac{d}{2},\Gamma)\setminus\{\frac32\}$. Using the inequality~\eqref{eq:lemK2} from Lemma~\ref{lem:Kolmogorov2} with $\alpha=1+\frac{\gamma}{2}$, one obtains
\begin{align*}
&\big|\E_x[D\varphi(X_M(t)).\zeta_M^{h_1,h_2}(t)]\big|\\
&\le C_\gamma(T)\vvvert\varphi\vvvert_2\int_0^t(t-s)^{-\frac{\gamma}{4}-\frac12}\E_x[\bigl(1+\|X_M(s)\|_\gamma^5)\|\bigr)\|A^{-\frac{\gamma}{2}}\PP\bigl(F''(X_M(s))(\Pi_M(s,0)h_1,\Pi_M(s,0)h_2)\bigr)\|]ds\\
&\le C_\gamma(T)\vvvert\varphi\vvvert_2\int_0^t(t-s)^{-\frac{\gamma}{4}-\frac12}\E_x[\bigl(1+\|X_M(s)\|_\gamma^5)\|\bigr)\|\bigl(F''(X_M(s))(\Pi_M(s,0)h_1,\Pi_M(s,0)h_2)\bigr)\|_{L^1}]ds\\
&\le C_\gamma(T)\vvvert\varphi\vvvert_2\int_0^t(t-s)^{-\frac{\gamma}{4}-\frac12}\E_x[\bigl(1+\|X_M(s)\|_\gamma^6)\bigr)\|\Pi_M(s,0)h_1\|\|\Pi_M(s,0)h_2\|_{L^1}]ds.
\end{align*}
Using the inequality~\eqref{eq:lemK2} from Lemma~\ref{lem:Kolmogorov2} (with $\alpha_1,\alpha_2\in[0,2)$, and the condition $\alpha_1+\alpha_2<2$ to ensure integrability), the moment bounds~\eqref{eq:momentsX-Galerkin} and H\"older's inequality, one obtains the following inequalities: for all $T\in(0,\infty)$, $\alpha_1,\alpha_2\in[0,2)$, with $\alpha_1+\alpha_2<2$ and $\gamma\in(\Gamma_0,\Gamma)$, there exists a real number $C_{\alpha_1,\alpha_2,\gamma}(T)\in(0,\infty)$ such that for all $t\in(0,T]$ one has
\begin{align*}
&\frac{\big|\E_x[D\varphi(X_M(t)).\zeta_M^{h_1,h_2}(t)]\big|}{(\|A^{-\alpha_1}\PP h_1\|+|\langle h_1, e_0\rangle| )(\|A^{-\alpha_2}\PP h_2\|+|\langle h_2, e_0\rangle| )}\\
&\le C_{\alpha_1,\alpha_2,\gamma}(T)\vvvert\varphi\vvvert_2\int_{0}^{t}(t-s)^{-\frac{\gamma}{4}-\frac12}s^{-\frac{\alpha_1+\alpha_2}{2}}ds \bigl(1+\|x\|_\gamma^{16q}\bigr)\\
&\le C_{\alpha_1,\alpha_2,\gamma}(T)\vvvert\varphi\vvvert_2t^{-\frac \gamma 4+\frac 12-\frac {\alpha_1+\alpha_2}2}\bigl(1+\|x\|_\gamma^{16q}\bigr)\\
&\le C_{\alpha_1,\alpha_2,\gamma}(T)\vvvert\varphi\vvvert_2t^{-\frac {\alpha_1+\alpha_2}2}\bigl(1+\|x\|_\gamma^{16q}\bigr),
\end{align*}
using the condition $\gamma<\Gamma\le 2$ in the last step.

Gathering the estimates for the two expressions appearing in the right-hand side of the second identity of~\eqref{eq:DuMD2uM} then yields the second inequality in~\eqref{eq:theoKolmogorov}. This concludes the proof of Theorem~\eqref{eq:theoKolmogorov}.
\end{proof}

\section{Weak error estimates for the spatial Galerkin method}\label{sec:proofGalerkin}

The objective of this section is to prove Theorem~\ref{theo:Galerkin}. The proof is based on a standard approach described in Section~\ref{sec:Galerkin} the weak error is written as~\eqref{eq:decompweakerrorGalerkin} using the solutions $u_M$ of Kolmogorov equations~\eqref{eq:Kolmogorov}, and relevant error terms are identified using It\^o's formula (see below). Finally, upper bounds for the error terms are obtained applying the regularity estimates from Theorem~\ref{theo:Kolmogorov}. Recall that two cases are considered, see Assumption~\ref{ass:Q}: space-time white noise ($d=1$, $\Gamma=3/2$) and trace-class noise ($d\in\{1,2,3\}$, $\Gamma=2$). The general strategy of the proof is the same in the two cases, however the upper bounds are proved using different arguments.

\begin{proof}[Proof of Theorem~\ref{theo:Galerkin}]
Owing to the discussion in Section~\ref{sec:Galerkin}, in order to prove the weak error estimate~\eqref{eq:weakerrorGalerkin}, it suffices to prove the upper bound~\eqref{eq:weakerrorGalerkinNM}.

Recall that the condition $\gamma\in(\Gamma_0,\Gamma)$ is satisfied. Let $T\in(0,\infty)$ and $x_0\in\HH^\gamma$. In addition, let $\varphi:H\to\R$ be a function of class $\mathcal{C}^2$ with bounded first and second order derivatives. Recall that $\bigl(X_N(t)\bigr)_{t\ge 0}$ is defined by~\eqref{eq:Galerkin} for all $N\in\N$. Moreover, for all $M\in\N$, recall that the function $u_M:[0,\infty)\times\HH_M\to\R$ is defined by~\eqref{eq:uM}. Owing to~\eqref{eq:decompweakerrorGalerkin}, for all $M\ge N$, one has
\begin{align*}
\E[\varphi(X_N(T))]-\E[\varphi(X_M(T))]&=\E[u_M(0,X_N(T))]-\E[u_M(T,X_M(0))]\\
&=u_M(T,P_Nx_0)-u_M(T,P_Mx_0)\\
&+\E[u_M(0,X_N(T))]-\E[u_M(T,X_N(0))].
\end{align*}
Using It\^o's formula, the fact that $u_M$ is solution of the Kolmogorov equation~\eqref{eq:Kolmogorov} and the definition~\eqref{eq:generator} of the infinitesimal generator $\mathcal{L}_N$ associated with~\eqref{eq:Galerkin}, one obtains
\begin{align*}
\E[u_M(0,X_N(T))]&-\E[u_M(T,X_N(0))]
=
\int_0^T
\E[\bigl(\mathcal{L}_N-\partial_t\bigr)u_M(T-t,X_N(t))]dt\\
&
=
\int_0^T
\E[\bigl(\mathcal{L}_N-\mathcal{L}_M\bigr)u_M(T-t,X_N(t))]dt\\
&=\int_0^T\E\bigl[Du_M(T-t,X_N(t)).\bigl(A(P_M-P_N)F(X_N(t))\bigr)\bigr]dt\\
&+\frac12\int_0^T\E[\bigl[\sum_{j\ge 0}D^2u_M(T-t,X_N(t)).\bigl((P_N-P_M)Q^{\frac12}e_j,(P_N+P_M)Q^{\frac12}e_j\bigr)\bigr]dt.
\end{align*}
As a consequence, for all $M\ge N$, one has the decomposition
\begin{equation}\label{eq:decompweakerrorGalerkinProof}
\E[\varphi(X_N(T))]-\E[\varphi(X_M(T))]=e_{N,M,0}+e_{N,M,1}+e_{N,M,2},
\end{equation}
where the error terms are given by
\begin{align*}
e_{N,M,0}&=u_M(T,P_Nx_0)-u_M(T,P_Mx_0)\\
e_{N,M,1}&=\int_0^T\E\bigl[Du_M(T-t,X_N(t)).\bigl(A(P_M-P_N)F(X_N(t))\bigr)\bigr]dt\\
e_{N,M,2}&=\frac12\int_0^T\E[\bigl[\sum_{j\ge 0}D^2u_M(T-t,X_N(t)).\bigl((P_N-P_M)Q^{\frac12}e_j,(P_N+P_M)Q^{\frac12}e_j\bigr)\bigr]dt.
\end{align*}
Let us prove upper bounds for these three error terms. Using the first inequality of~\eqref{eq:theoKolmogorov} from Theorem~\ref{theo:Kolmogorov} with $\alpha=0$, one obtains for all $M\ge N$ the inequality
\begin{equation}\label{eq:proofGalerkin-eNM0}
|e_{N,M,0}|\le C_{\gamma}(T)\bigl(1+\|x_0\|_\gamma^q\bigr)|(P_N-P_M)x_0|\le C_{\gamma}(T)\bigl(1+\|x_0\|_\gamma^q\bigr)\lambda_N^{-\gamma}\|x_0\|_\gamma.
\end{equation}
Let us now obtain an upper bound for the error term $e_{N,M,1}$. Observe that one has $\langle A(P_M-P_N)F(X_N(t)),e_0\rangle=0$ for all $t\in[0,T]$ and all $M\ge N$. Using the first inequality of~\eqref{eq:Kolmogorov} from Theorem~\ref{theo:Kolmogorov} with $\alpha\in[0,2)$, one obtains the following inequality: for all $M\ge N$, one has
\[
|e_{N,M,1}|\le C_{\alpha,\gamma}(T)\int_0^T(T-t)^{-\frac{\alpha}{2}}\E\Bigl[\bigl(1+\|X_N(t)\|_\gamma^q\bigr)\|A^{1-\alpha}(P_M-P_N)F(X_N(t))\|\Bigr]dt.
\]
On the one hand, in the space-time white noise case (Assumption~\ref{ass:Q}-$(i)$, $d=1$, $\Gamma=3/2$), choosing $\alpha=\frac12+\gamma$, one has $1-\alpha=\frac12-\gamma$ and one obtains
\begin{align*}
|e_{N,M,1}|&\le C_{\gamma}(T)\lambda_{N}^{-\gamma}\int_0^T(T-t)^{-\frac{1}{4}-\frac{\gamma}{2}}\E\Bigl[\bigl(1+\|X_N(t)\|_\gamma^q\bigr)\|A^{\frac12}F(X_N(t))\|\Bigr]dt\\
&\le C_{\gamma}(T)\lambda_{N}^{-\gamma}\int_0^T(T-t)^{-\frac{1}{4}-\frac{\gamma}{2}}\E\Bigl[\bigl(1+\|X_N(t)\|_\gamma^q\bigr)\bigl(1+\|X_N(t)\|_1^3\bigr)\Bigr]dt,
\end{align*}
using the algebra property~\eqref{eq:algebra} for $\HH^1$ ($1>d/2=1/2$). Using the moment bounds~\eqref{eq:momentsX-Galerkin}, for all $\gamma\in(\Gamma_0,\Gamma)=(1,\frac32)$, one obtains
\begin{equation}\label{eq:proofGalerkin-eNM1-stwn}
|e_{N,M,1}|\le C_\gamma(T)\lambda_N^{-\gamma}\bigl(1+\|x_0\|_{\gamma}^q\bigr).
\end{equation}
On the other hand, in the trace-class noise case (Assumption~\ref{ass:Q}-$(ii)$, $d\in\{1,2,3\}$, $\Gamma=2$), choosing $\alpha=1+\frac{\gamma}{2}$
and noting $(P_N-P_M)F(X_N(t)) \in \hh$, one obtains
\begin{align*}
|e_{N,M,1}|&\le C_{\gamma}(T)\int_0^T(T-t)^{-\frac12-\frac{\gamma}{4}}\E\Bigl[\bigl(1+\|X_N(t)\|_\gamma^q\bigr)\|A^{-\frac{\gamma}{2}}(P_N-P_M)F(X_N(t))\|\Bigr]dt\\
&\le C_{\gamma}(T)\lambda_{N}^{-\gamma}\int_0^T(T-t)^{-\frac12-\frac{\gamma}{4}}\E\Bigl[\bigl(1+\|X_N(t)\|_\gamma^q\bigr)\|F(X_N(t))\|_\gamma\Bigr]dt\\
&\le C_{\gamma}(T)\lambda_{N}^{-\gamma}\int_0^T(T-t)^{-\frac12-\frac{\gamma}{4}}\E\Bigl[\bigl(1+\|X_N(t)\|_\gamma^q\bigr)(1+\|X_N(t))\|_\gamma^3\bigr)\Bigr]dt,
\end{align*}
since in this case $\HH^\gamma$ is an algebra (owing to the condition $\gamma\in(\frac{d}{2},\Gamma)\setminus\{\frac32\}$). Using the moment bounds~\eqref{eq:momentsX-Galerkin}, for all $\gamma\in(\Gamma_0,\Gamma)$, one obtains
\begin{equation}\label{eq:proofGalerkin-eNM1-trace}
|e_{N,M,1}|\le C_\gamma(T)\lambda_N^{-\gamma}\bigl(1+\|x_0\|_{\gamma}^q\bigr).
\end{equation}
It remains to deal with the error term $e_{N,M,2}$. Observe that one has $\langle (P_N-P_M)Q^{\frac12}e_j,e_0\rangle=0$ and $\|(P_M+P_N)Q^{\frac12}e_j\|\le 2\|Q^{\frac12}e_j\|$ for all $j\in\N$ and $M\ge N$. Using the second inequality of~\eqref{eq:theoKolmogorov} from Theorem~\ref{theo:Kolmogorov} with $\alpha_1=\alpha\in[0,2)$ and $\alpha_2=0$, one obtains the following inequality: for all $M\ge N$, one has
\[
|e_{N,M,2}|\le C_{\alpha,\gamma}(T)\int_{0}^{T}(T-t)^{-\frac{\alpha}{2}}\E\bigl[1+\|X_N(t)\|_\gamma^q\bigr]dt\sum_{j\ge 0}\|A^{-\alpha}(P_M-P_N)Q^{\frac12}e_j\| \|Q^{\frac12}e_j\|.
\]
On the one hand, in the space-time white noise case (Assumption~\ref{ass:Q}-$(i)$, $d=1$, $\Gamma=3/2$), one has $Q^{\frac12}e_j=e_j$, and choosing $\alpha=\frac12+\frac{\gamma+\Gamma}{2}$, one obtains
\begin{align*}
\sum_{j\ge 0}\|A^{-\alpha}(P_M-P_N)Q^{\frac12}e_j\| \|Q^{\frac12}e_j\|&=\sum_{j=N+1}^{M}\lambda_j^{-\alpha}\\
&=\sum_{j=N+1}^{M}\lambda_j^{-\frac12-\frac{\Gamma-\gamma}{2}-\gamma}\\
&\le \lambda_N^{-\gamma}\sum_{j\ge 1}\lambda_j^{-\frac12-\frac{\Gamma-\gamma}{2}}.
\end{align*}
Using the moment bounds~\eqref{eq:momentsX-Galerkin}, for all $\gamma\in(\frac12,\Gamma)=(\frac12,\frac32)$, one obtains
\begin{equation}\label{eq:proofGalerkin-eNM2-stwn}
|e_{N,M,2}|\le C_\gamma(T)\lambda_N^{-\gamma}\bigl(1+\|x_0\|_{\gamma}^q\bigr).
\end{equation}
On the other hand, in the trace-class noise case (Assumption~\ref{ass:Q}-$(ii)$, $d\in\{1,2,3\}$, $\Gamma=2$), choosing $\alpha=\gamma$, one obtains
\[
\sum_{j\ge 0}\|A^{-\alpha}(P_M-P_N)Q^{\frac12}e_j\| \|Q^{\frac12}e_j\|\le \lambda_N^{-\gamma}\sum_{j\ge 0}\|Q^{\frac12}e_j\|^2=\lambda_N^{-\gamma}{\rm Tr}(Q).
\]
Using the moment bounds~\eqref{eq:momentsX-Galerkin}, for all $\gamma\in(\Gamma_0,\Gamma)$, one obtains
\begin{equation}\label{eq:proofGalerkin-eNM2-trace}
|e_{N,M,2}|\le C_\gamma(T)\lambda_N^{-\gamma}\bigl(1+\|x_0\|_{\gamma}^q\bigr).
\end{equation}
We are now in position to conclude the proof of the inequality~\eqref{eq:weakerrorGalerkinNM}: using~\eqref{eq:decompweakerrorGalerkinProof}, it suffices to gather~\eqref{eq:proofGalerkin-eNM0},~\eqref{eq:proofGalerkin-eNM1-stwn} and~\eqref{eq:proofGalerkin-eNM2-stwn} in the space-time white noise case, and~\eqref{eq:proofGalerkin-eNM0},~\eqref{eq:proofGalerkin-eNM1-trace} and~\eqref{eq:proofGalerkin-eNM2-trace} in the trace-class noise case. Using the arguments explained above, the proof of Theorem~\ref{theo:Galerkin} is thus completed.
\end{proof}

\section{Weak error estimates for the tamed exponential Euler scheme}
\label{sec:prooffull}

The objective of this section is to prove Theorem~\ref{theo:tamed}. Before proceeding with the analysis of the weak error, let us give some auxiliary results.

\subsection{Auxiliary results}

Recall that the continuous-time process $\bigl(\tilde{X}_N^{\Delta t}(t)\bigr)_{0\le t\le T}$ is defined by~\eqref{eq:Xtilde} (see Section~\ref{sec:full}), and satisfies $\tilde{X}_N^{\Delta t}(t_k)=X_{N,k}$ for all $k\in\{0,\ldots,K\}$.

For all $k\in\{0,\ldots,K\}$ and $t\in[t_k,t_{k+1})$, set $\ell(t)=k$ and $t_{\ell(t)}=t_k$. The first auxiliary result considered in this section provides some moment bounds for the process $\tilde{X}_N^{\Delta t}$.
\begin{theo}\label{theo:momentbounds-Xtilde}
Let Assumption~\ref{ass:Q} be satisfied. For all $T\in(0,\infty)$, $m\in\{1,\ldots\}$ and $\gamma\in(\Gamma_0,\Gamma)$, there exists $q\in\N$ and $C_{\gamma,m}(T)\in(0,\infty)$ such that for all $x_0\in\HH^\gamma$, all $N\in\N$ and all $\Delta t=T/K$ with $K\in\N$, one has
\begin{equation}
\bigl(\E_{x_0}[\underset{0\le t\le T}\sup~\|\tilde{X}_N^{\Delta t}(t)\|_\gamma^m]\bigr)^{\frac1m}\le C_{\gamma,m}(T)(1+\|x_0\|_\gamma^{q}).
\end{equation}
\end{theo}
Note that Theorem~\ref{theo:momentbounds-tamed} (see Section~\ref{sec:full}) giving moment bounds for the tamed exponential Euler scheme is a straightforward corollary of Theorem~\ref{theo:momentbounds-Xtilde}. The proof of Theorem~\ref{theo:momentbounds-Xtilde} is postponed to Section~\ref{sec:proofmomentsfull}.

An immediate corollary of Theorem~\ref{theo:momentbounds-Xtilde} is Lemma~\ref{lem:incrementsXtilde} below.
\begin{lemma}\label{lem:incrementsXtilde}
For all $T\in(0,\infty)$, $m\in\{1,\ldots\}$ and $\gamma\in(\Gamma_0,\Gamma)$, there exists $q\in\N$ and $C_{\gamma,m}(T)\in(0,\infty)$ such that for all $x_0\in H^\gamma$, all $N\in\N$ and all $\Delta t=T/K$ with $K\in\N$, one has
\begin{equation}\label{eq:incrementsXtilde}
\underset{0\le k\le K-1}\sup~\underset{t\in[t_k,t_{k+1}]}\sup~\bigl(\E[\|\tilde{X}_N^{\Delta t}(t)-X_{N,k}\|^m]\bigr)^{\frac1m}\le  C_{\gamma,m}(T)(1+\|x_0\|_\gamma^{{q}})\Delta t^{\frac{\gamma}{4}}.
\end{equation}
\end{lemma}

\begin{proof}
For all $k\in\{0,\ldots,K-1\}$ and $t\in[t_k,t_{k+1}]$, one has
\begin{align*}
\tilde{X}_N^{\Delta t}(t)-X_{N,k}&=(e^{-(t-t_k)A^2}-I)X_{N,k}\\
&+\int_{t_k}^{t} e^{-(t-s)A^2}\frac {-A P_N F(X_{N,k})}{1 + \Delta t \| P_N F(X_{N,k} )\|}ds\\
&+e^{-(t-t_k)A^2}P_N ( W^Q (t) - W^Q (t_k) ).
\end{align*}
First, using the inequality~\eqref{eq:regul} and Theorem~\ref{theo:momentbounds-Xtilde}, one has
\[
\bigl(\E[\|(e^{-(t-t_k)A^2}-I)X_{N,k}\|^m]\bigr)^{\frac1m}\le C_\gamma\Delta t^{\frac{\gamma}{4}}\bigl(\E[\|X_{N,k}\|_\gamma^m]\bigr)^{\frac1m}\le C_{\gamma,m}(T)(1+\|x_0\|_\gamma^q)\Delta t^{\frac{\gamma}{4}}.
\]
Second, using the smoothing inequality~\eqref{eq:smoothing-HH}, the inequality~\eqref{eq:boundF1} and Theorem~\ref{theo:momentbounds-Xtilde}, one has
\begin{align*}
\bigl(\E[\|\int_{t_k}^{t} e^{-(t-s)A^2}\frac {-A P_N F(X_{N,k})}{1 + \Delta t \| P_N F(X_{N,k} )\|}ds\|^m]\bigr)^{\frac1m}&\le C_\gamma\int_{t_k}^{t}(t-s)^{-\frac12+\frac{\gamma}{4}}ds\bigl(\E[\|F(X_{N,k})\|_{\gamma}^m]\bigr)^{\frac1m}\\
&\le C_\gamma\Delta t^{\frac12+\frac{\gamma}{4}}\bigl(\E[\|F(X_{N,k})\|_\gamma^m]\bigr)^{\frac1m}\\
&\le C_\gamma\Delta t^{\frac{\gamma}{4}}\bigl(\E[(1+\|X_{N,k}\|_\gamma^{3m}]\bigr)^{\frac1m}\\
&\le C_{\gamma,m}(T)\Delta t^{\frac{\gamma}{4}}(1+\|x_0\|_\gamma^{3q}).
\end{align*}
It remains to deal with the third term. One needs different arguments to treat the two situations for the noise.
On the one hand, under Assumption~\ref{ass:Q}-$(i)$ ($d=1$, $Q=I$ and $\Gamma=\frac32$), with $\gamma\in(\Gamma_0,\Gamma)$, one obtains
\begin{align*}
\bigl(\E[\|e^{-(t-t_k)A^2}P_N ( W^Q (t) - W^Q (t_k) )\|^m]\bigr)^{\frac1m}&\le C_m\bigl(\E[\|e^{-(t-t_k)A^2}P_N ( W^Q (t) - W^Q (t_k) )\|^2]\bigr)^{\frac12}\\
&\le C_m\Bigl(\sum_{j\in\N}(t-t_k)e^{-2(t-t_k)\lambda_j^2}\Bigr)^{\frac12}\\
&\le C_m(t-t_k)^{\frac12-\frac{2-\gamma}{4}}\Bigl(\sum_{j\in\N}\lambda_j^{-2+\gamma}\Bigr)^{\frac12}\\
&\le C_{\gamma,m}\Delta t^{\frac{\gamma}{4}}.
\end{align*}
On the other hand, under Assumption~\ref{ass:Q}-$(ii)$ ($d\in\{1,2,3\}$, ${\rm Tr}(Q)<\infty$ and $\Gamma=2$), with $\gamma\in(\Gamma_0,\Gamma)$, one obtains
\begin{align*}
\bigl(\E[\|e^{-(t-t_k)A^2}P_N ( W^Q (t) - W^Q (t_k) )\|^m]\bigr)^{\frac1m}&\le C_m\bigl(\E[\|e^{-(t-t_k)A^2}P_N ( W^Q (t) - W^Q (t_k) )\|^2]\bigr)^{\frac12}\\
&\le C_m\Delta t^{\frac12}({\rm Tr}(Q))^{\frac12}\\
&\le C_{\gamma,m}\Delta t^{\frac{\gamma}{4}}.
\end{align*}

Gathering the estimates concludes the proof of Lemma~\ref{lem:incrementsXtilde}.
\end{proof}

Let us also state and prove the following result concerning the nonlinearity $F$.
\begin{lemma}\label{lem:F}
For all $\gamma\in(\frac{d}{2},\Gamma)\setminus\{\frac 32\}$, there exists $C_\gamma\in(0,\infty)$ such that for all $x\in \HH^\gamma$ and $y\in \HH$, one has
\begin{equation}\label{eq:lemF}
\|A^{-\frac{\gamma}{2}}\PP (F'(x).y)\|\le C_\gamma (1+\|x\|_\gamma^2)(\|A^{-\frac{\gamma}{2}}\PP y\|+|\langle y,e_0\rangle| ).
\end{equation}
\end{lemma}

\begin{proof}
Recall that, for $\gamma\in(\frac{d}{2},\Gamma)/\{\frac 32\}$, the space $\HH^\gamma$, equipped with the norm $\|\cdot\|_\gamma=\|A^{\frac{\gamma}{2}}\cdot\|$ is an algebra, see~\eqref{eq:algebra} in Section~\ref{sec:auxineq}. Let $x\in\HH^\gamma$ and $y\in\HH$. For all $z\in\HH$, one has
\begin{align*}
\langle A^{-\frac{\gamma}{2}}\PP (F'(x).y),\PP z\rangle&=\langle F'(x).y,A^{-\frac{\gamma}{2}}\PP z\rangle\\
&=\langle y,F'(x).\bigl(A^{-\frac{\gamma}{2}}\PP z\bigr)\rangle\\
&\le \|A^{-\frac{\gamma}{2}}\PP y\| \|F'(x).\bigl(A^{-\frac{\gamma}{2}}\PP z\bigr)\|_\gamma+ |\langle y,e_0\rangle|\langle e_0, F'(x).\bigl(A^{-\frac{\gamma}{2}}\PP z \rangle|  \\
&\le \|A^{-\frac{\gamma}{2}}\PP y\|(1+\|x\|_\gamma^2)\|A^{-\frac{\gamma}{2}}\PP z\|_\gamma+ |\langle y,e_0\rangle| \| F'(x).\bigl(A^{-\frac{\gamma}{2}}\PP z\bigr) \|_{L^1}  \\
&\le (\|A^{-\frac{\gamma}{2}}\PP y\|+|\langle y,e_0\rangle |)(1+\|x\|_\gamma^2)\|\PP z\|.
\end{align*}
Using the identity
\[
\|A^{-\frac{\gamma}{2}}(\PP F'(x).y)\|=\underset{z\in \HH\setminus\{0\}}\sup~\frac{\langle A^{-\frac{\gamma}{2}}F'(x).y,\PP z\rangle}{\|\PP z\|}
\]
then concludes the proof of Lemma~\ref{lem:F}.
\end{proof}

\subsection{Proof of Theorem~\ref{theo:tamed}}

\begin{proof}[Proof of Theorem~\ref{theo:tamed}]
Recall that the mapping $u_N$ is defined by~\eqref{eq:Kolmogorov}, and satisfies the regularity estimates stated in Theorem~\ref{theo:Kolmogorov} (uniformly with respect to the spatial discretization parameter $N$), see Section~\ref{sec:Kolmogorov}. As explained in Section~\ref{sec:full}, it suffices to prove the weak error estimate~\eqref{eq:weakerrortamed}. Recall that the condition $\gamma\in(\Gamma_0,\Gamma)$ is satisfied. Let $T\in(0,\infty)$ and $x_0\in H^\gamma$. Recall that using a telescoping sum argument and the auxiliary process $\tilde{X}_N^{\Delta t}$ defined by~\eqref{eq:Xtilde}, the weak error is written as
\begin{align*}
\E[\varphi(X_{N,K})]-\E[\varphi(X_N(T))]&=\E[u_N(0,X_{N,K})]-\E[u_N(t_K,X_{N,0})]\\
&=\sum_{k=0}^{K-1}\bigl(\E[u_N(T-t_{k+1},X_{N,k+1})]-\E[u_N(T-t_k,X_{N,k})]\bigr)\\
&=\sum_{k=0}^{K-1}\bigl(\E[u_N(T-t_{k+1},\tilde{X}_N^{\Delta t}(t_{k+1}))]-\E[u_N(T-t_k,\tilde{X}_N^{\Delta t}(t_k))]\bigr),
\end{align*}
using the equalities $T=K\Delta t$, $t_k=k\Delta t$, $u_N(0,\cdot)=\varphi$, and $X_{N,k}=\tilde{X}_{N}^{\Delta t}(t_k)$ for all $k\in\{0,\ldots,K\}$. For all $k\in[t_k,t_{k+1}]$, the auxiliary process $\bigl(\tilde{X}_N^{\Delta t}(t)\bigr)_{t\in[t_k,t_{k+1}]}$ is solution of the auxiliary stochastic evolution equation
\[
d\tilde{X}_N^{\Delta t}(t)=-A^2\tilde{X}_N^{\Delta t}(t)dt-\frac{1}{1+\Delta t\|P_NF(X_{N,k})\|}AP_NF(X_{N,k})dt+e^{-(t-t_k)A^2}dW^Q(t).
\]
Using It\^o's formula, the fact that $u_N$ is solution of the Kolmogorov equation~\eqref{eq:Kolmogorov}, one obtains the decomposition
\begin{equation}\label{eq:decompweakerrortamedProof}
\E[\varphi(X_{N,K})]-\E[\varphi(X_N(T))]=\sum_{k=0}^{K-1}\bigl(e_k^1+e_k^2\bigr),
\end{equation}
where, for all $k\in\{0,\ldots,K-1\}$, the error terms are given by
\begin{align*}
e_k^1&=\int_{t_k}^{t_{k+1}}\E\bigl[\langle Du_N(T-t,\tilde{X}_N^{\Delta t}(t)),\frac{AP_NF(X_{N,k})}{1+\Delta t\|P_NF(X_{N,k})\|}-AP_NF(\tilde{X}_N^{\Delta t}(t))\rangle\bigr]dt
\\
e_k^2&=\frac12\int_{t_k}^{t_{k+1}}\E\bigl[{\rm Tr}\Bigl(D^2u_N(T-t,\tilde{X}_N^{\Delta t}(t))\bigl(e^{-(t-t_k)A^2}P_NQP_Ne^{-(t-t_k)A^2}-P_NQP_N\Bigr)\bigr]dt.
\end{align*}
The two error terms $e_k^1$ and $e_k^2$ are decomposed as follows: for all $k\in\{0,\ldots,K-1\}$
\begin{align*}
e_k^1&=e_k^{1,1}+e_k^{1,2}\\
e_k^2&=e_k^{2,1}+e_k^{2,2}
\end{align*}
where
\begin{align*}
e_k^{1,1}&=\int_{t_k}^{t_{k+1}}\E\bigl[\langle Du_N(T-t,\tilde{X}_N^{\Delta t}(t)),AP_NF(X_{N,k})-AP_NF(\tilde{X}_N^{\Delta t}(t))\rangle\bigr]dt\\
e_k^{1,2}&=-\Delta t\int_{t_k}^{t_{k+1}}\E\bigl[\langle Du_N(T-t,\tilde{X}_N^{\Delta t}(t)),\frac{\|P_NF(X_{N,k})\|}{1+\Delta t\|P_NF(X_{N,k})\|}AP_NF(X_{N,k})\rangle\bigr]dt
\end{align*}
and
\begin{align*}
e_k^{2,1}&=\frac12\int_{t_k}^{t_{k+1}}\E\bigl[{\rm Tr}\Bigl(D^2u_N(T-t,\tilde{X}_N^{\Delta t}(t))\bigl(e^{-(t-t_k)A^2}P_NQP_N(e^{-(t-t_k)A^2}-I)\Bigr)\bigr]dt\\
e_k^{2,2}&=\frac12\int_{t_k}^{t_{k+1}}\E\bigl[{\rm Tr}\Bigl(D^2u_N(T-t,\tilde{X}_N^{\Delta t}(t))\bigl((e^{-(t-t_k)A^2}-I)P_NQP_N\Bigr)\bigr]dt.
\end{align*}

The weak error estimate~\eqref{eq:weakerrortamed} is a straightforward consequence of the decomposition~\eqref{eq:decompweakerrortamedProof} and of the following claims: for all $\gamma\in(\Gamma_0,\Gamma)$ and $T\in(0,\infty)$, there exists $C_\gamma(T)\in(0,\infty)$ such that for all $k\in\{0,\ldots,K-1\}$ one has
\begin{align}
|e_k^{1,1}|&\le C_\gamma(T)\vvvert\varphi\vvvert_2(1+\|x_0\|_\gamma^{q})\Delta t^{\frac{\gamma}{2}}\int_{t_k}^{t_{k+1}}(T-t)^{-\frac12-\frac{\gamma}{4}}dt,\label{eq:claim11}\\
|e_k^{1,2}|&\le C_\gamma(T)\vvvert\varphi\vvvert_2(1+\|x_0\|_\gamma^{q})\Delta t^{\frac{\gamma}{2}}\int_{t_k}^{t_{k+1}}(T-t)^{-\frac12}dt,\label{eq:claim12}\\
|e_k^{2,1}|&\le C_\gamma(T)\vvvert\varphi\vvvert_2(1+\|x_0\|_\gamma^{q})\Delta t^{\frac{\gamma}{2}}\int_{t_k}^{t_{k+1}}(T-t)^{-\frac{\alpha}{2}}dt,\label{eq:claim21}\\
|e_k^{2,2}|&\le C_\gamma(T)\vvvert\varphi\vvvert_2(1+\|x_0\|_\gamma^{q})\Delta t^{\frac{\gamma}{2}}\int_{t_k}^{t_{k+1}}(T-t)^{-\frac{\alpha}{2}}dt,\label{eq:claim22}
\end{align}
where the value of the parameter $\alpha\in(0,2]$ depends on the assumptions on the covariance $Q$: $\alpha=\frac12+\frac{\gamma+\Gamma}{2}$ under Assumption~\ref{ass:Q}-$(i)$ and $\alpha=\gamma$ under Assumption~\ref{ass:Q}-$(ii)$. The most difficult part of the proof is to establish~\eqref{eq:claim11}.

$\bullet$ Proof of the inequality~\eqref{eq:claim12}.

Using Taylor's formula, one has
\begin{align*}
e_k^{1,1}&=\int_{t_k}^{t_{k+1}}\E\bigl[\langle Du_N(T-t,\tilde{X}_N^{\Delta t}(t)),AP_NF'(X_{N,k}).\bigl(X_{N,k}-\tilde{X}_N^{\Delta t}(t)\bigr)\rangle\bigr]dt\\
&-\int_{t_k}^{t_{k+1}}\E\bigl[\langle Du_N(T-t,\tilde{X}_N^{\Delta t}(t)),\int_{0}^{1}(1-\xi)AP_NF''(\varrho_{N,k,t}(\xi)).\bigl(\tilde{X}_N^{\Delta t}(t)-X_{N,k},\tilde{X}_N^{\Delta t}(t)-X_{N,k}\bigr)d\xi\rangle\bigr]dt,
\end{align*}
where $\varrho_{N,k,t}(\xi)=X_{N,k}+\xi(\tilde{X}_N^{\Delta t}(t)-X_{N,k})$. 

Writing $Du_N(T-t,\tilde{X}_N^{\Delta t}(t))=Du_N(T-t,X_{N,k})+Du_N(T-t,\tilde{X}_N^{\Delta t}(t))-Du_N(T-t,X_{N,k})$ and using a conditional expectation argument to get
\[
\E\bigl[\langle Du_N(T-t,X_{N,k}),AP_NF'(X_{N,k}).\bigl(e^{-(t-t_k)A^2}(W^Q(t)-W^Q(t_k))\bigr)\rangle\bigr]=0,
\]
the error term $e_k^{1,1}$ is decomposed as
\[
e_k^{1,1}=e_k^{1,1,1}+e_k^{1,1,2}+e_k^{1,1,3}+e_k^{1,1,4},
\]
where
\begin{align*}
e_k^{1,1,1}&=-\int_{t_k}^{t_{k+1}}\E\bigl[\langle Du_N(T-t,X_{N,k}),AP_NF'(X_{N,k}).\bigl((e^{-(t-t_k)A^2}-I)X_{N,k}\bigr)\rangle\bigr]dt\\
e_k^{1,1,2}&=\int_{t_k}^{t_{k+1}}\E\bigl[\langle Du_N(T-t,X_{N,k}),AP_NF'(X_{N,k}).\int_{t_k}^{t}e^{-(t-s)A^2}\frac{AP_NF(X_{N,k})}{1+\Delta t\|P_NF(X_{N,k})\|}ds\rangle\bigr]dt\\
e_k^{1,1,3}&=\int_{t_k}^{t_{k+1}}\E\bigl[\langle Du_N(T-t,\tilde{X}_N^{\Delta t}(t))-Du_N(T-t,X_{N,k}),AP_NF'(X_{N,k}).\bigl(X_{N,k}-\tilde{X}_N^{\Delta t}(t)\bigr)\rangle\bigr]dt\\
e_k^{1,1,4}&=-\int_{t_k}^{t_{k+1}}\E\bigl[\langle Du_N(T-t,\tilde{X}_N^{\Delta t}(t)),\int_{0}^{1}(1-\xi)AP_NF''(\varrho_{N,k,t}(\xi)).\bigl(\tilde{X}_N^{\Delta t}(t)-X_{N,k},\tilde{X}_N^{\Delta t}(t)-X_{N,k}\bigr)d\xi\rangle\bigr]dt.
\end{align*}

For the error term $e_k^{1,1,1}$, using the regularity estimate on $Du_N(T-t,X_{N,k})$ from Theorem~\ref{theo:Kolmogorov} (with $\alpha=1+\frac{\gamma}{2}<2$), Lemma~\ref{lem:F} and the inequality~\eqref{eq:regul}, one obtains
\begin{align*}
|e_k^{1,1,1}|&\le C_\gamma(T)\vvvert\varphi\vvvert_2\int_{t_k}^{t_{k+1}}(T-t)^{-\frac12-\frac{\gamma}{4}}\E[(1+\|X_{N,k}\|_\gamma^q)\|A^{-\frac{\gamma}{2}}\PP F'(X_{N,k}).\bigl((e^{-(t-t_k)A^2}-I)X_{N,k}\bigr)\|]dt\\
&\le C_\gamma(T)\vvvert\varphi\vvvert_2\int_{t_k}^{t_{k+1}}(T-t)^{-\frac12-\frac{\gamma}{4}}\E[(1+\|X_{N,k}\|_\gamma^{q+2})\|A^{-\frac{\gamma}{2}}\bigl((e^{-(t-t_k)A^2}-I)X_{N,k}\bigr)\|]dt\\
&\le C_\gamma(T)\vvvert\varphi\vvvert_2\int_{t_k}^{t_{k+1}}(T-t)^{-\frac12-\frac{\gamma}{4}}dt\E[(1+\|X_{N,k}\|_\gamma^{q+3})] \Delta t^{\frac{\gamma}{2}}\\
&\le C_\gamma(T)\vvvert\varphi\vvvert_2(1+\|x_0\|_\gamma^{q(q+3)})\int_{t_k}^{t_{k+1}}(T-t)^{-\frac12-\frac{\gamma}{4}}dt \Delta t^{\frac{\gamma}{2}},
\end{align*}
using the moment bounds from Theorem~\ref{theo:momentbounds-tamed} in the last step.

For the error term $e_k^{1,1,2}$, using the regularity estimate on $Du_N(T-t,X_{N,k})$ from Theorem~\ref{theo:Kolmogorov} (with $\alpha=1+\frac{\gamma}{2}<2$) and Lemma~\ref{lem:F}, one obtains
\begin{align*}
&|e_k^{1,1,2}|\\
&\le C_\gamma(T)\vvvert\varphi\vvvert_2\int_{t_k}^{t_{k+1}}(T-t)^{-\frac12-\frac{\gamma}{4}}\E[(1+\|X_{N,k}\|_\gamma^q)\|A^{-\frac{\gamma}{2}}\PP F'(X_{N,k}).\int_{t_k}^{t}e^{-(t-s)A^2}\frac{AP_NF(X_{N,k})}{1+\Delta t\|P_NF(X_{N,k})\|}ds\|]dt\\
&\le C_\gamma(T)\vvvert\varphi\vvvert_2\int_{t_k}^{t_{k+1}}(T-t)^{-\frac12-\frac{\gamma}{4}}\E[(1+\|X_{N,k}\|_\gamma^{q+2})\|A^{-\frac{\gamma}{2}}\PP \int_{t_k}^{t}e^{-(t-s)A^2}AF(X_{N,k})ds\|]dt\\
&\le C_\gamma(T)\vvvert\varphi\vvvert_2\int_{t_k}^{t_{k+1}}(T-t)^{-\frac12-\frac{\gamma}{4}}\E[(1+\|X_{N,k}\|_\gamma^{q+2})\int_{t_k}^{t}\|A^{\frac{2-\gamma}{2}}\PP F(X_{N,k})\|]dsdt\\
&\le C_\gamma(T)\vvvert\varphi\vvvert_2\int_{t_k}^{t_{k+1}}(T-t)^{-\frac12-\frac{\gamma}{4}}dt\E[(1+\|X_{N,k}\|_\gamma^{q+5})\Delta t\\
&\le C_\gamma(T)\vvvert\varphi\vvvert_2(1+\|x_0\|_\gamma^{q(q+5)})\int_{t_k}^{t_{k+1}}(T-t)^{-\frac12-\frac{\gamma}{4}}dt \Delta t,
\end{align*}
using the fact that $\gamma$ can be chosen such that $\gamma>1$ since $\gamma>\Gamma_0$, and the moment bounds from Theorem~\ref{theo:momentbounds-tamed} in the last step.

For the error term $e_k^{1,1,3}$, using the regularity estimate on $D^2u_N(T-t,X_{N,k})$ from Theorem~\ref{theo:Kolmogorov} (with $\alpha_1=0$ and $\alpha_2=1$), one obtains
\begin{align*}
&|e_k^{1,1,3}|\\
&\le C_\gamma(T)\vvvert\varphi\vvvert_2\int_{t_k}^{t_{k+1}}(T-t)^{-\frac12}\E[(1+\|X_{N,k}\|_\gamma^q+\|\tilde{X}_N^{\Delta t}(t)\|_\gamma^q)\|\tilde{X}_N^{\Delta t}-X_{N,k}\| \|F'(X_{N,k}).(\tilde{X}_N^{\Delta t}-X_{N,k})\|]dt\\
&\le C_\gamma(T)\vvvert\varphi\vvvert_2\int_{t_k}^{t_{k+1}}(T-t)^{-\frac12}\E[(1+\|X_{N,k}\|_\gamma^{q+2}+\|\tilde{X}_N^{\Delta t}(t)\|_\gamma^{q+2})\|\tilde{X}_N^{\Delta t}-X_{N,k}\|^2]dt\\
&\le C_\gamma(T)\vvvert\varphi\vvvert_2(1+\|x_0\|_\gamma^{q(q+3)})\int_{t_k}^{t_{k+1}}(T-t)^{-\frac12}dt \Delta t^{\frac{\gamma}{2}},
\end{align*}
using the moment bounds from Theorem~\ref{theo:momentbounds-Xtilde}, the bounds on the increments from Lemma~\ref{lem:incrementsXtilde}, and H\"older's inequality in the last step.

For the error term $e_k^{1,1,4}$, using the regularity estimate on $D^2u_N(T-t,X_{N,k})$ from Theorem~\ref{theo:Kolmogorov} (with $\alpha=1+\frac{\gamma}{2}$), one obtains
\begin{align*}
&|e_k^{1,1,4}|\\
&\le C_\gamma(T)\vvvert\varphi\vvvert_2\int_{t_k}^{t_{k+1}}
(T-t)^{-\frac12 
- \frac{\gamma}{4}}
\E[(1+\|X_{N,k}\|_\gamma^q)\|\int_0^1 (1-\xi) \\
&\qquad\|A^{-\frac{\gamma}{2}} \PP F''(\varrho_{N,k,t}(\xi).\bigl(\tilde{X}_N^{\Delta t}(t)-X_{N,k},
 \tilde{X}_N^{\Delta t}(t)-X_{N,k}\bigr)\|]d\xi dt\\
&\le C_\gamma(T)\vvvert\varphi\vvvert_2\int_{t_k}^{t_{k+1}}(T-t)^{-\frac12
- \frac{\gamma}{4}
}\E[(1+\|X_{N,k}\|_\gamma^{q+1}+\|\tilde{X}_N^{\Delta t}(t)\|_\gamma^{q+1})\|\tilde{X}_N^{\Delta t}(t)-X_{N,k}\|^2]dt\\
&\le C_\gamma(T)\vvvert\varphi\vvvert_2(1+\|x_0\|_\gamma^{q(q+3)})\int_{t_k}^{t_{k+1}}
(T-t)^{-\frac12
- \frac{\gamma}{4}}dt \Delta t^{\frac{\gamma}{2}},
\end{align*}
using the moment bounds from Theorem~\ref{theo:momentbounds-Xtilde}, the inequality~\eqref{eq:ineq3}, the bounds on the increments from the Lemma~\ref{lem:incrementsXtilde}, and H\"older's inequality in the last step.

Gathering the estimates on the error terms $e_k^{1,1,1},e_k^{1,1,2},e_k^{1,1,3},e_k^{1,1,4}$, one obtains
\[
|e_k^{1,1}|\le C_\gamma(T)\vvvert\varphi\vvvert_2(1+\|x_0\|_\gamma^{q})\int_{t_k}^{t_{k+1}}(T-t)^{-\frac12-\frac{\gamma}{4}}dt \Delta t^{\frac{\gamma}{2}}.
\]
This provides the inequality~\eqref{eq:claim11}.

$\bullet$ Proof of the inequality~\eqref{eq:claim12}.

Owing to the regularity property of $Du_N(T-t,\cdot)$ from Theorem~\ref{theo:Kolmogorov} (with $\alpha=1$) and to the moment bounds on the auxiliary process $\tilde{X}_N^{\Delta t}$ from Theorem~\ref{theo:momentbounds-Xtilde}, for all $k\in\{0,\ldots,K-1\}$, one obtains the error estimate
\begin{align*}
|e_k^{1,2}|&\le C(T)\vvvert\varphi\vvvert_2\Delta t\int_{t_k}^{t_{k+1}}(T-t)^{-\frac12}\E\bigl[(1+\|\tilde{X}_N^{\Delta t}(t)\|_\gamma^q)\|P_NF(X_{N,k})\|^2\bigr]dt\\
&\le C_\gamma(T)\vvvert\varphi\vvvert_2\Delta t\bigl(1+\|x_0\|_\gamma^{q+6}\bigr)\int_{t_k}^{t_{k+1}}(T-t)^{-\frac12}dt,
\end{align*}
using the inequality~\eqref{eq:boundF1} and the moment bounds from Theorem~\ref{theo:momentbounds-Xtilde}. This provides the inequality~\eqref{eq:claim12}.

$\bullet$ Proof of the inequality~\eqref{eq:claim21}.

Owing to the regularity property of $D^2u_N(T-t,\cdot)$ from Theorem~\ref{theo:Kolmogorov} (with $\alpha_1=0$ and $\alpha_2=\alpha$) and to the moment bounds on the auxiliary process $\tilde{X}_N^{\Delta t}$ from Theorem~\ref{theo:momentbounds-Xtilde}, for all $k\in\{0,\ldots,K-1\}$, one obtains the upper bounds
\begin{align*}
|e_k^{2,1}|&\le \frac12\int_{t_k}^{t_{k+1}}\Big|\sum_{j\ge 1}D^2u_N(T-t,\tilde{X}_N^{\Delta t}(t)).\bigl(e^{-(t-t_k)A^2}P_NQ^{\frac12}e_j,(e^{-(t-t_k)A^2}-I)P_NQ^{\frac12}e_j\bigr)\Big|dt\\
&\le C_{\gamma}(T)\vvvert\varphi\vvvert_2\int_{t_k}^{t_{k+1}}(T-t)^{-\frac{\alpha}{2}}\E[(1+\|\tilde{X}_N^{\Delta t}(t)\|_\gamma^q)]\sum_{j\ge 1}\|Q^{\frac12}e_j\| \|A^{-\alpha}(e^{-(t-t_k)A^2}-I)P_NQ^{\frac12}e_j\|dt\\
&\le C_{\gamma}(T)\vvvert\varphi\vvvert_2(1+\|x_0\|_\gamma^{q^2})\int_{t_k}^{t_{k+1}}(T-t)^{-\frac{\alpha}{2}}\sum_{j\ge 1 }\|Q^{\frac12}e_j\| \|A^{-\alpha}(e^{-(t-t_k)A^2}-I)P_NQ^{\frac12}e_j\|dt.
\end{align*}

On the one hand, in the space-time white noise case (Assumption~\ref{ass:Q}-$(i)$, $d=1$, $Q=I$ and $\Gamma=\frac32$), choosing $\alpha=\frac{1}{2}+\frac{\gamma+\Gamma}{2}$, using the inequality~\eqref{eq:regul}, one obtains
\begin{align*}
|e_k^{2,1}|&\le C_{\gamma}(T)\vvvert\varphi\vvvert_2(1+\|x_0\|_\gamma^{q^2})\int_{t_k}^{t_{k+1}}(T-t)^{-\frac{\alpha}{2}} \sum_{j\ge 1}\lambda_j^{-\frac12-\frac{\Gamma-\gamma}{2}}(t-t_k)^{\frac{\gamma}{2}}dt\\
&\le C_{\gamma}(T)\vvvert\varphi\vvvert_2(1+\|x_0\|_\gamma^{q^2})\int_{t_k}^{t_{k+1}}(T-t)^{-\frac{\alpha}{2}}dt \Delta t^{\frac{\gamma}{2}}.
\end{align*}

On the other hand, in the trace-class noise case (Assumption~\ref{ass:Q}-$(ii)$, $d\in\{1,2,3\}$, ${\rm Tr}(Q)<\infty$ and $\Gamma=2$), choosing $\alpha=\gamma$, using the inequality~\eqref{eq:regul}, one obtains
\begin{align*}
|e_k^{2,1}|&\le C_{\gamma}(T)\vvvert\varphi\vvvert_2(1+\|x_0\|_\gamma^{q^2})\int_{t_k}^{t_{k+1}}(T-t)^{-\frac{\alpha}{2}}\sum_{j\in\N}\|Q^{\frac12}e_j\|^2 (t-t_k)^{\frac{\gamma}{2}}dt\\
&\le C_{\gamma}(T)\vvvert\varphi\vvvert_2(1+\|x_0\|_\gamma^{q^2})\int_{t_k}^{t_{k+1}}(T-t)^{-\frac{\gamma}{2}}dt \Delta t^{\frac{\gamma}{2}}.
\end{align*}
This provides the inequality~\eqref{eq:claim21}.

$\bullet$ Estimate for $e_k^{2,2}$.

Writing
\[
|e_k^{2,2}|\le \frac12\int_{t_k}^{t_{k+1}}\Big|\sum_{j\in\N}D^2u_N(T-t,\tilde{X}_N^{\Delta t}(t)).\bigl((e^{-(t-t_k)A^2}-I)P_NQ^{\frac12}e_j,P_NQ^{\frac12}e_j\bigr)\Big|dt
\]
it is straightforward to obtain~\eqref{eq:claim22} using the same arguments as in the proof of~\eqref{eq:claim21} above. The details are omitted.

We are now in position to conclude: as explained above, it suffices to combine the error estimates~\eqref{eq:claim11},~\eqref{eq:claim12},
~\eqref{eq:claim21} and~\eqref{eq:claim22} with the decomposition~\eqref{eq:decompweakerrortamedProof} to obtain~\eqref{eq:weakerrortamed}. The proof of Theorem~\ref{theo:tamed} is thus completed.
\end{proof}

\section{Proof of Theorem~\ref{theo:momentbounds-tamed}}\label{sec:proofmomentsfull}

The objective of this section is to provide the proof of Theorem~\ref{theo:momentbounds-tamed}, giving moment bounds for $X_{N,k}$ with $0\le t_k\le T$. In fact, one proves the stronger version of the result, see Theorem~\ref{theo:momentbounds-Xtilde}, giving moment bounds for $\tilde{X}_N^{\Delta t}(t)$ with $0\le t\le T$ (recall that $\tilde{X}_N^{\Delta t}(t_k)=X_{N,k}$).

The proof requires several steps and the introduction of auxiliary processes. First, introduce Gaussian auxiliary processes $\tilde{Z}_N^{\Delta t}$ defined as follows: for all $N\in\N$, $\Delta t=T/K$, $k\in\{0,\ldots,K-1\}$ and $t\in[t_k,t_{k+1}]$, set
\begin{equation}\label{eq:Ztilde}
\begin{aligned}
Z_{N,k+1}&=e^{-\Delta tA^2}\bigl(Z_{N,k}+P_N\Delta W_k^Q\bigr),\\
\tilde{Z}_N^{\Delta t}(t)&=e^{-(t-t_k)A^2}\bigl(Z_{N,k}+P_N(W^Q(t)-W^Q(t_k))\bigr),
\end{aligned}
\end{equation}
with initial values $Z_{N,0}=\tilde{Z}_N^{\Delta t}(0)=0$. One has the following auxiliary result. 
\begin{lemma}\label{lem:tamedGaussian}
For all $T\in(0,\infty)$, $m\in\{1,\ldots\}$ and $\gamma\in[0,\Gamma)$, there exists $q\in\N$ and $C_{\gamma,m}(T)\in(0,\infty)$ such that for all $N\in\N$ and all $\Delta t=T/K$ with $K\in\N$, one has
\begin{equation}\label{eq:lemtamedGaussian}
\bigl(\E[\underset{0\le t\le T}\sup~\|\tilde{Z}_N^{\Delta t}(t)\|_\gamma^m]\bigr)^{\frac1m}\le C_{\gamma,m}(T).
\end{equation}
\end{lemma}

\begin{proof}
The moment bound is a consequence of the following auxiliary bounds on the increments of the process $\tilde{Z}_N^{\Delta t}$, using the fact that this process is Gaussian and the Kolmogorov regularity criterion: for all $\gamma\in[0,\Gamma)$, there exists $C_{\gamma}(T)\in(0,\infty)$ such that for all $N\in\N$ and all $\Delta t=T/K$ with $K\in\N$, one has
\begin{equation}\label{eq:incrementstamedGaussian}
\underset{0\le s<t\le T}\sup~\frac{\E[\|\tilde{Z}_N^{\Delta t}(t)-\tilde{Z}_N^{\Delta t}(s)\|_\gamma^2]}{|t-s|^{\frac{\Gamma-\gamma}{4}}}\le C_{\gamma}(T).
\end{equation}

In the sequel, the objective is to prove the inequality~\eqref{eq:incrementstamedGaussian}. Note that $\tilde{Z}_N^{\Delta t}$ is {the} solution of the stochastic evolution equation
\[
d\tilde{Z}_N^{\Delta t}(t)=-A^2\tilde{Z}_N^{\Delta t}(t)dt+e^{-(t-t_{\ell(t)})A^2}dW^Q(t),
\]
where we recall that $\ell(t)=k$ if $t\in[t_k,t_{k+1})$. Therefore for all $t\in[0,T]$ one has
\[
\tilde{Z}_N^{\Delta t}(t)=\int_{0}^{t}e^{-(t-t_{\ell(r)})A^2}P_NdW^Q(r),
\]
and for all $0\le s<t\le T$, one has
\[
\tilde{Z}_N^{\Delta t}(t)-\tilde{Z}_N^{\Delta t}(s)=(e^{-(t-s)A^2}-I)\tilde{Z}_N^{\Delta t}(s)+\int_{s}^{t}e^{-(t-t_{\ell(r)})A^2}P_NdW^Q(r).
\]
We claim that the following auxiliary moment bound holds: there exists $C_\gamma(T)\in(0,\infty)$ such that for all $N\in\N$ and $\Delta t=T/K$ one has
\begin{equation}\label{eq:boundZNDeltat}
\underset{0\le s\le T}\sup~\E[\|\tilde{Z}_N^{\Delta t}(s)\|_\gamma^2]\le C_{\gamma}(T).
\end{equation}
Let us prove~\eqref{eq:boundZNDeltat}. Applying It\^o's isometry formula, one obtains for all $s\in[0,T]$
\[
\E[\|\tilde{Z}_N^{\Delta t}(s)\|_\gamma^2]=\int_0^s \|A^{\frac{\gamma}{2}}e^{-(s-t_{\ell(r)})A^2}P_NQ^{\frac12}\|_{\mathcal{L}_2(H)}^2 dr.
\]
On the one hand, in the space-time white noise case (Assumption~\ref{ass:Q}-$(i)$, $d=1$, $Q=I$ and $\Gamma=\frac32$), for all $s\in[0,T]$, one has
\begin{align*}
\E[\|\tilde{Z}_N^{\Delta t}(s)\|_\gamma^2]&=\int_0^s \sum_{j\in\N}\lambda_j^{\gamma}e^{-2(s-t_{\ell(r)})\lambda_j^2} dr\\
&\le \int_0^s \sum_{j\in\N}\lambda_j^{-\frac12-\frac{\Gamma-\gamma}{2}} \lambda_j^{\frac12+\frac{\gamma+\Gamma}{2}}e^{-2(s-r)\lambda_j^2} dr\\
&\le C_{\gamma}\int_0^s (s-r)^{-\frac14-\frac{\gamma+\Gamma}{4}}dr\\
&\le C_{\gamma}(T).
\end{align*}
On the other hand, in the trace-class noise case (Assumption~\ref{ass:Q}-$(ii)$, $d\in\{1,2,3\}$, ${\rm Tr}(Q)<\infty$ and $\Gamma=2$), for all $s\in[0,T]$, one has
\begin{align*}
\E[|\tilde{Z}_N^{\Delta t}(s)|_\gamma^2]&\le \int_0^s \|A^{\frac{\gamma}{2}}e^{-(s-t_{\ell(r)})A^2}P_N\|_{\mathcal{L}(H)}^2\|Q^{\frac12}\|_{\mathcal{L}_2(H)}^2dr\\
&\le C_\gamma \int_0^s (s-r)^{-\frac{\gamma}{2}}ds {\rm Tr}(Q)\\
&\le C_{\gamma}(T),
\end{align*}
using the smoothing inequality~\eqref{eq:smoothing-HH}. This concludes the proof of the auxiliary moment bound~\eqref{eq:boundZNDeltat}. We are now in position to prove the claim.

For all $0\le s<t\le T$, applying It\^o's isometry formula, one has
\begin{align*}
\E[\|\tilde{Z}_N^{\Delta t}(t)-\tilde{Z}_N^{\Delta t}(s)\|_\gamma^2]&=\E[\|(e^{-(t-s)A^2}-I)\tilde{Z}_N^{\Delta t}(s)\|_\gamma^2]\\
&+\E[\|\int_{s}^{t}e^{-(t-t_{\ell(r)})A^2}P_NdW^Q(r)\|_\gamma^2]\\
&\le \E[\|(e^{-(t-s)A^2}-I)\tilde{Z}_N^{\Delta t}(s)\|_\gamma^2]\\
&+\int_s^t \|A^{\frac{\gamma}{2}}e^{-(t-t_{\ell(r)})A^2}P_NQ^{\frac12}\|_{\mathcal{L}_2(H)}^2 dr.
\end{align*}

On the one hand, using the inequality~\eqref{eq:regul} and applying the auxiliary moment bound~\eqref{eq:boundZNDeltat} gives
\[
\E[\|(e^{-(t-s)A^2}-I)\tilde{Z}_N^{\Delta t}(s)\|_\gamma^2]\le C_{\gamma}(T)|t_2-t_1|^{\frac{\Gamma-\gamma}{4}}\E[\|\tilde{Z}_N^{\Delta t}(s)\|_{\frac{\gamma+\Gamma}{2}}^2]\le C_{\gamma}(T)|t_2-t_1|^{\frac{\Gamma-\gamma}{4}}.
\]
On the other hand, like in the proof of the inequality~\eqref{eq:boundZNDeltat}, two cases need to be considered. First, in the space-time white noise case (Assumption~\ref{ass:Q}-$(i)$, $d=1$, $\Gamma=\frac32$), applying It\^o's isometry formula, one has for all $0\le s<t\le T$,
\begin{align*}
\E[\|\int_{s}^{t}e^{-(t-t_{\ell(r)})A^2}P_NdW^Q(r)\|_\gamma^2]&=\int_s^t \|A^{\frac{\gamma}{2}}e^{-(t-t_{\ell(r)})A^2}P_NQ^{\frac12}\|_{\mathcal{L}_2(H)}^2 dr\\
&=\int_0^s \sum_{j\in\N}\lambda_j^{\gamma}e^{-2(s-t_{\ell(r)})\lambda_j^2} dr\\
&\le \int_s^t \sum_{j\in\N}\lambda_j^{-\frac12-\frac{\Gamma-\gamma}{2}} \lambda_j^{\frac12+\frac{\gamma+\Gamma}{2}}e^{-2(t-t_{\ell(r)})\lambda_j^2} dr\\
&\le C_{\gamma}\int_s^t (t-r)^{-\frac14-\frac{\gamma+\Gamma}{2}}dr\\
&\le C_{\gamma}(T)(t-s)^{\frac{\Gamma-\gamma}{4}},
\end{align*}
using the identity $\frac34-\frac{\gamma+\Gamma}{4}=\frac38-\frac{\gamma}{4}=\frac{\Gamma-\gamma}{4}$ since $\Gamma=\frac32$.

Second, in the trace-class noise case (Assumption~\ref{ass:Q}-$(ii)$, $d\in\{1,2,3\}$, $\Gamma=2$), applying It\^o's isometry formula, one has for all $0\le s<t\le T$,
\begin{align*}
\E[\|\int_{s}^{t}e^{-(t-t_{\ell(r)})A^2}P_NdW^Q(r)\|_\gamma^2]&=\int_s^t \|A^{\frac{\gamma}{2}}e^{-(t-t_{\ell(r)})A^2}P_NQ^{\frac12}\|_{\mathcal{L}_2(H)}^2 dr\\
&\le \int_s^t \|A^{\frac{\gamma}{2}}e^{-(t-t_{\ell(r)})A^2}P_N\|_{\mathcal{L}(H)}^2\|Q^{\frac12}\|_{\mathcal{L}_2(H)}^2dr\\
&\le C_\gamma \int_s^t (t-r)^{-\frac{\gamma}{2}}ds {\rm Tr}(Q)\\
&\le C_{\gamma}(T)|t-s|^{\frac{\Gamma-\gamma}{2}},
\end{align*}
using the identity $1-\frac{\gamma}{2}=\frac{\Gamma-\gamma}{2}$ since $\Gamma=2$.

The proof of the inequality~\eqref{eq:incrementstamedGaussian} is thus completed. Applying the Kolmogorov regularity criterion then concludes the proof of Lemma~\ref{lem:tamedGaussian}.
\end{proof}

Recall that $\tilde{X}_N^{\Delta t}(0)=X_{N,0}=P_Nx_0$. Let us now define additional auxiliary processes as follows. For all $t\ge 0$, define
\begin{align*}
&\tilde{Z}_N^{\Delta t;x_0}(t)=e^{-tA^2}P_Nx_0+\tilde{Z}_N^{\Delta t}(t)\\
&\tilde{Y}_N^{\Delta t}(t)=\tilde{X}_N^{\Delta t}(t)-\tilde{Z}_N^{\Delta t;x_0}(t).
\end{align*}
Using the definition~\eqref{eq:Xtilde} of $\tilde{X}_N^{\Delta t}$ and the definition~\eqref{eq:Ztilde} of $\tilde{Z}_N^{\Delta t}$, observe that for all $t\in[0,T]$, one has
\begin{equation}\label{eq:Ytilde}
\tilde{Y}_N^{\Delta t}(t)=\int_0^t e^{-(t-s)A^2}\frac{-AP_NF(X_{\ell(s)})}{1+\Delta t\|P_NF(X_{\ell(s)})\|}ds.
\end{equation}
Finally, define the auxiliary processes $\tilde{R}_N^{\Delta t}$ and $\tilde{r}_N^{\Delta t}$ as follows: for all $t\in[0,T]$, set
\begin{equation}\label{eq:Rrtilde}
\begin{aligned}
&\tilde{R}_N^{\Delta t}(t)=-\int_0^t e^{-(t-s)A^2}AP_NF\bigl(\tilde{Y}_N^{\Delta t}(s)+\tilde{Z}_N^{\Delta t;x_0}(t_{\ell(s)})\bigr)ds\\
&\tilde{r}_N^{\Delta t}(t)=\tilde{Y}_N^{\Delta t}(t)-\tilde{R}_N^{\Delta t}(t).
\end{aligned}
\end{equation}

The strategy of the proof of Theorem~\ref{theo:momentbounds-Xtilde} is straightforward. First, one proves Lemma~\ref{lem:tamed-r} and Lemma~\ref{lem:tamed-R} below to obtain some moment bounds for the auxiliary processes $\tilde{r}_{N}^{\Delta t}$ and $\tilde{R}_N^{\Delta t}$. Combining these results gives moment bounds for $\tilde{Y}_{N}^{\Delta t}$, and using Lemma~\ref{lem:tamedGaussian} one obtains moment bounds for $\tilde{X}_{N}^{\Delta t}$. However, one cannot directly prove the moment bounds, and first one needs to prove moment bounds on some well-chosen events $\Omega_k^{\Delta t}$ defined below. The proof of Theorem~\ref{theo:momentbounds-Xtilde} below shows how to remove the indicator functions of those events. The arguments above are standard in the proof of moment bounds for tamed Euler schemes applied to SDEs and SPDEs. The arguments which are specific to the Cahn--Hilliard equation are given in the proof of Lemma~\ref{lem:tamed-R}, which is the most technical part of the analysis. Note that the techniques used to prove Lemma~\ref{lem:tamed-R} are similar to those needed to prove the moment bounds~\eqref{eq:momentsX-Galerkin} for $X_N(t)$ (see Section~\ref{sec:Galerkin}).

Let $\gamma\in(\Gamma_0,\Gamma)$ be given. Recall that this ensures that the conditions $\gamma\in(\frac{d}{2},\Gamma)\setminus\{\frac32\}$ and $\gamma\ge 1+\frac{d}{4}$ if $\Gamma=2$ are satisfied. Let $\theta\in(0,1)$ be a sufficiently small auxiliary parameter, upper bounds on $\theta$ are given below in the analysis. For all $k\in\{0,\ldots,K\}$, define the event
\[
\Omega_{N,k}^{\Delta t}=\{\underset{0\le \ell \le k}\sup~\|X_{N,\ell}\|\le \Delta t^{-\theta}\},
\]
and define the associated indicator functions
\[
\chi_{N,k}^{\Delta t}=\mathds{1}_{\Omega_{N,k}^{\Delta t}}.
\]

By convention, set $\Omega_{N,-1}^{\Delta t}=\Omega$. Let us now state the auxiliary moment bounds for $\tilde{r}_N^{\Delta t}$ and $\tilde{R}_N^{\Delta t}$.

\begin{lemma}\label{lem:tamed-r}
For all $\gamma\in(\Gamma_0,\Gamma)$, there exists a sufficiently small $\theta\in(0,1)$ such that
\begin{equation}
\underset{0\le k\le K}\sup~\bigl(\E[\chi_{N,k-1}^{\Delta t}\underset{0\le t \le t_k}\sup~\|\tilde{r}_N^{\Delta t}(t)\|_\gamma^m]\bigr)^{\frac1m}\le C_m(T)(1+\|x_0\|_\gamma^q).
\end{equation}
\end{lemma}

\begin{lemma}\label{lem:tamed-R}
For all $\gamma\in(\Gamma_0,\Gamma)$, there exists a sufficiently small $\theta\in(0,1)$ such that
\begin{equation}
\underset{0\le k\le K}\sup~\bigl(\E[\chi_{N,k-1}^{\Delta t}\underset{0\le t \le t_k}\sup~\|\tilde{R}_N^{\Delta t}(t)\|_\gamma^m]\bigr)^{\frac1m}\le C_m(T)(1+\|x_0\|_\gamma^q).
\end{equation}
\end{lemma}

The value of the integer $q$ in the two statements below does not depend on the value of the regularity parameter $\gamma$.

To simplify the notation in the proofs below, set $\chi_k=\chi_{N,k}^{\Delta t}$, $\tilde{X}(t)=\tilde{X}_N^{\Delta t}(t)$, $X_k=X_{N,k}$, $\tilde{Z}^{x_0}(t)=\tilde{Z}_N^{\Delta t;x_0}(t)$, $\tilde{Y}(t)=\tilde{Y}_N^{\Delta t}(t)$, $\tilde{R}(t)=\tilde{R}_N^{\Delta t}(t)$ and $\tilde{r}(t)=\tilde{r}_N^{\Delta t}(t)$. All the upper bounds are uniform with respect to the omitted parameters $N\in\N$ and $\Delta t=T/K$, $K\in\N$.

\begin{proof}[Proof of Lemma~\ref{lem:tamed-r}]
Note that for all $t\in[0,T]$, one has
\[
\tilde{r}(t)=\tilde{r}_1(t)+\tilde{r}_2(t),
\]
where
\begin{align*}
\tilde{r}_1(t)&=-\int_0^te^{-(t-s)A^2}\frac{-\Delta t\|F(X_\ell(s))\|}{1+\Delta t\|P_NF(X_{\ell(s)})\|}AP_NF(X_{\ell(s)})ds\\
\tilde{r}_2(t)&=-\int_0^te^{-(t-s)A^2}AP_N\Bigl(F(\tilde{Y}(t_{\ell(s)})+\tilde{Z}^{x_0}(t_{\ell(s)}))-F(\tilde{Y}(s)+\tilde{Z}^{x_0}(t_{\ell(s)}))\Bigr)ds.
\end{align*}

Assume that the condition $\theta<1/6$ is satisfied. First, for all $k\in\{0,\ldots,K\}$ and all $t\in[0,t_k]$, one has
\begin{align*}
\chi_{k-1}\|\tilde{r}_1(t)\|_\gamma&\le\chi_{k-1} \int_0^t (t-s)^{-\frac12}\Delta t\|F(X_{\ell(s)})\|_\gamma^2 ds\\
&\le C\chi_{k-1} \int_0^t (t-s)^{-\frac12}\Delta t(1+\|X_{\ell(s)}\|_\gamma^6) ds\\
&\le Ct^{\frac12} \Delta t^{1-6\theta}\\
&\le C(T).
\end{align*}
This yields the inequality
\[
\underset{0\le k\le K}\sup~\bigl(\E[\chi_{k-1}\underset{0\le t \le t_k}\sup~\|\tilde{r}_1(t)\|_\gamma^m]\bigr)^{\frac1m}\le C_m(T).
\]
It remains to prove a similar upper bound for $\tilde{r}_2(t)$. Introduce an auxiliary parameter $\kappa\in(0,\gamma-\frac{d}{2})$. One then has, for all $t\le t_k$,
\begin{align*}
\chi_{k-1}\|\tilde{r}_2(t)\|_\gamma&\le C_\kappa\chi_{k-1}\int_0^t (t-s)^{-\frac12-\frac{\kappa}{4}}\|F(\tilde{Y}(t_{\ell(s)})+\tilde{Z}^{x_0}(t_{\ell(s)}))-F(\tilde{Y}(s)+\tilde{Z}^{x_0}(t_{\ell(s)}))\|_{\gamma-\kappa}ds\\
&\le C_\kappa\chi_{k-1}\int_{0}^{t}(t-s)^{-\frac12-\frac{\kappa}{4}}\|\tilde{Y}(t_{\ell(s)})-\tilde{Y}(s)\|_{\gamma-\kappa}\mathcal{M}(s)ds,
\end{align*}
where
\[
\mathcal{M}(s)=1+\|\tilde{Y}(t_{\ell(s)})\|_{\gamma-\kappa}^2+\|\tilde{Y}(s)\|_{\gamma-\kappa}^2+\|\tilde{Z}^{x_0}(t_{\ell(s)})\|_{\gamma-\kappa}^2.
\]

On the one hand, one has
\begin{align*}
\chi_{k-1}\|\tilde{Y}(t_{\ell(s)})-\tilde{Y}(s)\|_{\gamma-\kappa}&\le \chi_{k-1}\|(e^{-(s-t_{\ell(s)})A^2}-I)\tilde{Y}(t_{\ell(s)})\|_{\gamma-\kappa}\\
&+\chi_{k-1}\int_{t_{\ell(s)}}^{s}\frac{\|e^{-(s-r)A^2}AP_NF(X_{\ell(r)})\|_\gamma}{1+\Delta t\|P_NF(X_{\ell(r)})\|}dr\\
&\le C_\kappa\Delta t^{\frac{\kappa}{4}}\chi_{k-1}\|\tilde{Y}(t_{\ell(s)})\|_{\gamma}+\Delta t^{\frac12}\chi_{k-1}(1+\|X_{\ell(s)}\|_\gamma^3)\\
&\le C_\kappa\Delta t^{\min(\frac{\kappa}{4},\frac12)-3\theta}+C_\kappa\Delta t^{\frac{\kappa}{4}}\|\tilde{Z}^{x_0}(t_{\ell(s)})\|_\gamma.
\end{align*}
On the other hand, one has for all $s\in[0,t_k)$
\begin{align*}
\chi_{k-1}\mathcal{M}(s)&\le C\chi_{k-1}\bigl(1+\|\tilde{Y}(t_{\ell(s)})\|_{\gamma-\kappa}^2+\|\tilde{Y}(s)-\tilde{Y}(t_{\ell(s)})\|_{\gamma-\kappa}^2+\|\tilde{Z}^{x_0}(t_{\ell(s)})\|_{\gamma-\kappa}^2\bigr)\\
&\le C_\kappa+C_\kappa\Delta t^{-2\theta}+C_\kappa\Delta t^{2\min(\frac{\kappa}{4},\frac12)-6\theta}+ C_\kappa\|\tilde{Z}^{x_0}(t_{\ell(s)})\|_\gamma^2.
\end{align*}
Using the Cauchy--Schwarz inequality and Lemma~\ref{lem:tamedGaussian}, if $\theta$ is sufficiently small, one then obtains
\[
\underset{0\le k\le K}\sup~\bigl(\E[\chi_{k-1}\underset{0\le s \le t_k}\sup~\|\tilde{Y}(t_{\ell(s)})-\tilde{Y}(s)\|_{\gamma-\kappa}^m\mathcal{M}(s)^m]\bigr)^{\frac1m}\le C(1+\|x_0\|_\gamma^3),
\]
which then gives
\[
\underset{0\le k\le K}\sup~\bigl(\E[\chi_{k-1}\underset{0\le t \le t_k}\sup~\|\tilde{r}_2(t)\|_\gamma^m]\bigr)^{\frac1m}\le C_m(T)(1+\|x_0\|_\gamma^3).
\]
Gathering the estimates the concludes the proof of Lemma~\ref{lem:tamedGaussian}.
\end{proof}

\begin{proof}[Proof of Lemma~\ref{lem:tamed-R}]
Note that the process $\tilde{R}$ defined by~\eqref{eq:Rrtilde} is solution of the evolution equation
\[
\frac{d\tilde{R}(t)}{dt}+A^2\tilde{R}(t)+AP_NF(\tilde{R}(t)+\tilde{r}(t)+\tilde{Z}^{x_0}(t_{\ell(t)}))=0,
\]
with initial value $\tilde{R}(0)=0$.

The proof of Lemma~\ref{lem:tamed-R} requires to complete three steps, which use different arguments, similar to those that would be needed to prove moment bounds~\eqref{eq:momentsX-Galerkin}: first, energy estimates in the $\|A^{-\frac12}\cdot\|$ and $\|\cdot\|$ norms; second, arguments depending on dimension $d$ to obtain estimates in the $\|\cdot\|_{L^6}$ norm; third, the mild formulation to get estimates in the $\|\cdot\|_{\gamma}$ norm.

$\bullet$ Step $1$: let us prove that
\begin{equation}\label{eq:claim1}
\underset{0\le k\le K}\sup~\bigl(\E[\chi_{k-1}\underset{0\le t \le t_k}\sup~\|\tilde{R}(t)\|^m]\bigr)^{\frac1m}\le C_m(T)(1+\|x_0\|_\gamma^q).
\end{equation}

First, observe that $\widetilde R(t) \in \hh$, therefore one has the following energy estimate in the $\|A^{-\frac12}\PP \cdot\|$ norm:
\begin{align*}
&\frac12\frac{d}{dt}\|A^{-\frac12} \tilde{R}(t)\|^2+\|A^{\frac12} \tilde{R}(t)\|^2\\
&=-\langle F(\tilde{R}(t)+\tilde{r}(t)+\tilde{Z}^{x_0}(t)), \tilde{R}(t)\rangle\\
&=-\langle F(\tilde{R}(t)+\tilde{r}(t)+\tilde{Z}^{x_0}(t))-F(\tilde{R}(t)), \tilde{R}(t)\rangle+\langle F(\tilde{R}(t)), \tilde{R}(t)\rangle
\\
&=\langle \tilde{r}(t)+\tilde{Z}^{x_0}(t),\tilde{R}(t)\rangle-\langle (\tilde{r}(t)+\tilde{Z}^{x_0}(t))^3,\tilde{R}(t)\rangle\\
&-3\langle (\tilde{r}(t)+\tilde{Z}^{x_0})^2,\tilde{R}(t)\rangle-3\langle \tilde{r}(t)+\tilde{Z}^{x_0}(t),\tilde{R}(t)^2\rangle\\
&-\|\tilde{R}(t)\|_{L^4}^{4}+\|\tilde{R}(t)\|^2-\langle F(\tilde{R}(t)+\tilde{r}(t)+\tilde{Z}^{x_0}(t)),  \tilde{R}(t)\rangle\\
&\le -\frac12\|\tilde{R}(t)\|_{L^4}^4+\|\tilde{R}(t)\|^2+C\|\tilde{r}(t)+\tilde{Z}^{x_0}(t)\|_{L^4}^4\\
&\le -\frac12\|\tilde{R}(t)\|_{L^4}^4+\frac12\|A^{-\frac12}\tilde{R}(t)\|^2+\frac12\|A^{\frac12}\tilde{R}(t)\|^2+C\|\tilde{r}(t)+\tilde{Z}^{x_0}(t)\|_\gamma^4.
\end{align*}
Applying Gronwall's lemma yields the following inequality: for all $t\in[0,T]$,
\[
\int_0^t \|A^{\frac12}\tilde{R}(s)\|^2ds+\int_0^t \|\tilde{R}(s)\|_{L^4}^4ds\le C(T)\underset{0\le s\le t}\sup~(\|\tilde{r}(s)\|_\gamma^4+\|\tilde{Z}^{x_0}(s)\|_{\gamma}^4).
\]
Second, one has the energy estimate in the $\|\cdot\|$ norm:
\begin{align*}
\frac12\frac{d}{dt}\|\tilde{R}(t)\|^2+\|A\tilde{R}(t)\|^2&=-\langle AF(\tilde{R}(t)+\tilde{r}(t)+\tilde{Z}^{x_0}(t)),\tilde{R}(t)\rangle\\
&=-\langle F(\tilde{R}(t)+\tilde{r}(t)+\tilde{Z}^{x_0}(t)),A\tilde{R}(t)\rangle\\
&=-\langle F(\tilde{R}(t)+\tilde{r}(t)+\tilde{Z}^{x_0})-F(\tilde{R}(t)),A\tilde{R}(t)\rangle\\
&-\langle F(\tilde{R}(t)),A\tilde{R}(t)\rangle\\
&\le \frac12\|A\tilde{R}(t)\|^2+\frac12\|F(\tilde{R}(t)+\tilde{r}(t)+\tilde{Z}^{x_0}(t))-F(\tilde{R}(t))\|^2\\
&+\|\nabla\tilde{R}(t)\|^2-3\|\nabla\tilde{R}(t) \tilde{R}(t)\|^2\\
&\le \frac 12\|A\tilde{R}(t)\|^2+\|\nabla\tilde{R}(t)\|^2\\
&+C\bigl(1+\|\tilde{r}(t)\|_\gamma^6+\|\tilde{Z}^{x_0}(t)\|_\gamma^6\bigr)\bigl(1+\|\tilde{R}(t)\|_{L^4}^4\bigr)-3\|\nabla \tilde{R}(t) \tilde{R}(t)\|^2,
\end{align*}
using Young's inequality and the identity
\[
\langle \tilde{R}^3(t),A\tilde{R}(t)\rangle=-3\langle\nabla \tilde{R}(t) \tilde{R}^2,\nabla \tilde{R}(t)\rangle=-3\|\nabla \tilde{R}(t) \tilde{R}(t)\|^2.
\]

Using the inequality above, one obtains, for all $t\in[0,T]$,
\begin{align*}
\|\tilde{R}(t)\|^2&+\int_0^t\|A\tilde{R}(s)\|^2ds+\int_0^t\|\nabla \tilde{R}(s) \tilde{R}(s)\|^2ds\\
&\le C(T) \int_0^t \|A^{\frac12}\tilde{R}(s)\|^2ds+C(T)\int_0^t\bigl(1+\|\tilde{r}(s)\|_\gamma^6+\|\tilde{Z}^{x_0}(s)\|_\gamma^6\bigr)\bigl(1+\|\tilde{R}(s)\|_{L^4}^4\bigr)ds\\
&\le C(T)\underset{0\le s_le t}\sup~(\|\tilde{r}(s)\|_\gamma^4+\|\tilde{Z}^{x_0}(s)\|_{\gamma}^4)\\
&+C(T)\underset{0\le s\le t}\sup~\bigl(1+\|\tilde{r}(s)\|_\gamma^6+\|\tilde{Z}^{x_0}(s)\|_\gamma^6\bigr)\int_{0}^{t}\bigl(1+\|\tilde{R}(s)\|_{L^4}^4\bigr)ds\\
&\le C(T)\underset{0\le s\le t}\sup~\bigl(1+\|\tilde{r}(s)\|_\gamma^{10}+\|\tilde{Z}^{x_0}(s)\|_\gamma^{10}\bigr).
\end{align*}
Using the moment bound from Lemma~\ref{lem:tamed-r} then concludes the proof of the claim~\eqref{eq:claim1}.

Note that one also obtains the inequalities
\begin{align*}
\underset{0\le k\le K}\sup~\bigl(\E[\chi_{k-1}\bigl(\int_0^{t_k}\|A\tilde{R}(t)\|^2dt\bigr)^m]\bigr)^{\frac1m}&\le C_m(T)(1+\|x_0\|_\gamma^q)\\
\underset{0\le k\le K}\sup~\bigl(\E[\chi_{k-1}\bigl(\int_0^{t_k}\|\nabla \tilde{R}(t) \tilde{R}(s)\|^2ds\bigr)^m]\bigr)^{\frac1m}&\le C_m(T)(1+\|x_0\|_\gamma^q)
\end{align*}
which are used below.

$\bullet$ Step $2$: let us prove that
\begin{equation}\label{eq:claim2}
\underset{0\le k\le K}\sup~\bigl(\E[\chi_{k-1}\underset{0\le t \le t_k}\sup~\|\tilde{R}(t)\|_{L^6}^m]\bigr)^{\frac1m}\le C_m(T)(1+\|x_0\|_\gamma^q).
\end{equation}
Different strategies are necessary to treat the cases $d=1$ and $d=2,3$.

First, assume that $d=1$. Using the definition~\eqref{eq:Rrtilde} of $\tilde{R}(t)$, the Sobolev inequality~\eqref{eq:Sobolev-norm_alpha}, the smoothing property~\eqref{eq:smoothing-HH}, and the fact that $F$ is a polynomial function of degree $3$, one obtains
\begin{align*}
\|\tilde{R}(t)\|_{L^6}&\le \int_0^{t}\|e^{-(t-s)A^2}AF(\tilde{R}(s)+\tilde{r}(s)+\tilde{Z}^{x_0}(t_{\ell(s)}))\|_{L^6}ds\\
&\le C\int_0^{t}\|e^{-(t-s)A^2}A^{1+\frac16}F(\tilde{R}(s)+\tilde{r}(s)+\tilde{Z}^{x_0}(t_{\ell(s)}))\|ds\\
&\le C\int_0^t (t-s)^{-\frac{7}{12}}\|F(\tilde{R}(s)+\tilde{r}(s)+\tilde{Z}^{x_0}(t_{\ell(s)})\|ds\\
&\le C\int_0^t (t-s)^{-\frac{7}{12}}\bigl(1+\|\tilde{R}(s)\|_{L^6}^3+\|\tilde{r}(s)\|_{L^6}^3+\|\tilde{Z}^{x_0}(t_{\ell(s)})\|_{L^6}^3\bigr)ds.
\end{align*}
On the one hand, one has $\|\tilde{r}(s)\|_{L^6}\le \|\tilde{r}(s)\|_{\gamma}$ and $\|\tilde{Z}^{x_0}(t_{\ell(s)})\|_{L^6}\le \|\tilde{Z}^{x_0}(t_{\ell(s)})\|_{\gamma}$ for all $s\le t\le t_k$.

On the other hand, using the Gagliardo-Nirenberg inequality~\eqref{eq:L6-d1-GN} and Young's inequality, one obtains
\begin{align*}
(t-s)^{-\frac{7}{12}}\|\tilde{R}(s)\|_{L^6}^3&\le C\|A\tilde{R}(s)\|^{\frac12}(t-s)^{-\frac{7}{12}}\|\tilde{R}(s)\|^{\frac52}\\
&\le C\|A\tilde{R}(s)\|^2+C(t-s)^{-\frac{7}{9}}\|\tilde{R}(s)\|^{\frac{10}{3}}.
\end{align*}
Using the inequalities proved above then concludes the proof of the claim~\eqref{eq:claim2} in the case $d=1$.

Second, assume that $d=2$ or $d=3$. The mapping $t\mapsto J(\tilde{R}(t))=\frac12\|\tilde{R}(t)\|_1^2+\frac14\|\tilde{R}(t)\|_{L^4}^4-\frac12\|\tilde{R}(t)\|_{L^2}^{2}$, where $J$ is the energy functional defined by~\eqref{eq:energy}, satisfies the evolution equation
\begin{align*}
\frac{dJ(\tilde{R}(t))}{dt}&=\langle A\tilde{R}(t)+F(\tilde{R}(t)),\frac{d\tilde{R}(t)}{dt}\rangle\\
&=-\langle A\tilde{R}(t),A^2\tilde{R}(t)\rangle-\langle A\tilde{R}(t),AP_NF(\tilde{R}(t)+\tilde{r}(t)+\tilde{Z}^{x_0}(t_{\ell(t)}))\rangle\\
&-\langle F(\tilde{R}(t)),A^2\tilde{R}(t)\rangle-\langle F(\tilde{R}(t)),AP_NF(\tilde{R}(t)+\tilde{r}(t)+\tilde{Z}^{x_0}(t_{\ell(t)}))\rangle\\
&=-\langle A^{\frac32}\tilde{R}(t),A^{\frac32}\tilde{R}(t)\rangle-\langle A^{\frac32}\tilde{R}(t),A^{\frac 12} P_NF(\tilde{R}(t)+\tilde{r}(t)+\tilde{Z}^{x_0}(t_{\ell(t)}))\rangle\\
&-\langle A^{\frac12}P_NF(\tilde{R}(t)),A^{\frac32}\tilde{R}(t)\rangle-\langle A^{\frac12}\tilde{R}(t),A^{\frac12}P_NF(\tilde{R}(t)+\tilde{r}(t)+Z^{x_0}(t_{\ell(t)}))\rangle\\
&=-\|A^{\frac32}\tilde{R}(t)+A^{\frac12}P_NF(\tilde{R}(t)+\tilde{r}(t)+\tilde{Z}^{x_0}(t_{\ell(t)}))\|^2\\
&+\langle A^{\frac12}P_NF(\tilde{R}(t)+\tilde{r}(t)+\tilde{Z}^{x_0}(t_{\ell(t)})),A^{\frac12}P_NF(\tilde{R}(t)+\tilde{r}(t)+\tilde{Z}^{x_0}(t_{\ell(t)}))\rangle\\
&+\langle A^{\frac32}\tilde{R}(t),A^{\frac12}P_NF(\tilde{R}(t)+\tilde{r}(t)+\tilde{Z}^{x_0}(t_{\ell(t)}))\rangle\\
&-\langle A^{\frac32}\tilde{R}(t),A^{\frac12}P_NF(\tilde{R}(t))\rangle\\
&-\langle A^{\frac12}P_NF(\tilde{R}(t)),A^{\frac12}P_NF(\tilde{R}(t)+\tilde{r}(t)+\tilde{Z}^{x_0}(t_{\ell(t)}))\rangle\\
&=-\|A^{\frac32}\tilde{R}(t)+A^{\frac12}P_NF(\tilde{R}(t)+\tilde{r}(t)+\tilde{Z}^{x_0}(t_{\ell(t)}))\|^2\\
&+\langle F(\tilde{R}(t)+\tilde{r}(t)+\tilde{Z}^{x_0}(t_{\ell(t)}))-F(\tilde{R}(t)),A^2\tilde{R}(t)+AP_NF(\tilde{R}(t)+\tilde{r}(t)+\tilde{Z}^{x_0}(t_{\ell(t)}))\rangle\\
&=-\|A^{\frac32}\tilde{R}(t)+A^{\frac12}P_NF(\tilde{R}(t)+\tilde{r}(t)+\tilde{Z}^{x_0}(t_{\ell(t)}))\|^2\\
&+\langle A^{\frac12}F(\tilde{R}(t)+\tilde{r}(t)+\tilde{Z}^{x_0}(t_{\ell(t)}))-A^{\frac12}F(\tilde{R}(t)),A^{\frac32}\tilde{R}(t)+A^{\frac12}P_NF(\tilde{R}(t)+\tilde{r}(t)+\tilde{Z}^{x_0}(t_{\ell(t)}))\rangle.
\end{align*}
Using Young's inequality, and the fact that $J(\tilde{R}(0))=0$ and {$A^{\frac 12}(F(\tilde{R}(s)+\tilde{r}(s)+\tilde{Z}^{x_0}(t_{\ell(s)}))-F(\tilde{R}(s)))\in \hh$}, one obtains the inequality
\[
J(\tilde{R}(t))\le \frac12\int_0^t \|F(\tilde{R}(s)+\tilde{r}(s)+\tilde{Z}^{x_0}(t_{\ell(s)}))-F(\tilde{R}(s))\|_{1}^2 ds
\]
for all $t\in[0,T]$. Using the identity $\nabla F(R)=(2R^2-1)\nabla R$, one obtains for all $s\in[0,T]$
\begin{align*}
\|F(\tilde{R}(s)&+\tilde{r}(s)+\tilde{Z}^{x_0}(t_{\ell(s)}))-F(\tilde{R}(s))\|_{ 1}^2\\
&=\|(2\tilde{R}(s)^2-1)\nabla \tilde{R}(s)-2\bigl((\tilde{R}(s)+\tilde{r}(s)+\tilde{Z}^{x_0}(t_{\ell(s)}))^2-1\bigr)\nabla \bigl(\tilde{R}(s)+\tilde{r}(s)+\tilde{Z}^{x_0}(s)\bigr)\|^2\\
&\le C\|\tilde{r}(s)+\tilde{Z}^{x_0}(t_{\ell(s)})\|_{L^\infty}^2\|\tilde{R}(s)\nabla \tilde{R}(s)\|^2\\
&+C\|\tilde{r}(s)+\tilde{Z}^{x_0}(t_{\ell(s)})\|_{L^\infty}^4\|\nabla \tilde{R}(s)\|^2\\
&+C\|\tilde{r}(s)+\tilde{Z}^{x_0}(t_{\ell(s)})\|_{L^\infty}^2\|\nabla(\tilde{r}(s)+\tilde{Z}^{x_0}(t_{\ell(s)}))\|_{L^4}^{2}\\
&+C\|\nabla(\tilde{r}(s)+\tilde{Z}^{x_0}(t_{\ell(s)}))\|_{L^4}^{2}\|\tilde{R}(s)\|_{L^4}^4\\
&+C\|\tilde{r}(s)+\tilde{Z}(s)\|_{L^\infty}^4\|\nabla(\tilde{r}(s)+\tilde{Z}^{x_0}(t_{\ell(s)})\|^2\\
&\le C\bigl(1+\|\tilde{r}(s)\|_{\gamma}^4+\|\tilde{Z}(t_{\ell(s)})\|_{\gamma}^4\bigr)\bigl(1+\|\nabla\tilde{R}(s)\|^2+\|\tilde{R}(s)\nabla\tilde{R}(s)\|^2+\|\tilde{R}(s)\|_{L^4}^4\bigr),
\end{align*}
using the inequality $\|\nabla \cdot\|_{L^4}\le C\|\nabla \cdot\|_{\frac{d}{4}}\le C\|\cdot\|_{1+\frac{d}{4}}\le C\|\cdot\|_\gamma$, which follows from the Sobolev embedding $H^{\frac{d}{4}}\subset L^4$, see the inequality~\eqref{eq:L4-d}, and from the condition $\gamma\ge 1+\frac{d}{4}$.

One then obtains
\begin{align*}
\chi_{k-1}\underset{0\le t\le t_k}\sup~J(\tilde{R}(t))\le C\chi_{k-1}\underset{0\le t\le t_k}\sup~&\bigl(1+\|\tilde{r}(t)\|_{\gamma}^4+\|\tilde{Z}(t)\|_{\gamma}^4\bigr)\\
&\cdot\chi_{k-1}\int_{0}^{t_k}\bigl(1+\|\nabla\tilde{R}(t)\|^2+\|\tilde{R}(t)\nabla\tilde{R}(t)\|^2+\|\tilde{R}(t)\|_{L^4}^4\bigr)dt.
\end{align*}
Using the previous inequalities and H\"older's inequality, one obtains
\[
\underset{0\le k\le K}\sup~\bigl(\E[\chi_{k-1}\underset{0\le t \le t_k}\sup~|J(\tilde{R}(t))|^m]\bigr)^{\frac1m}\le C_m(T)(1+\|x_0\|_\gamma^q).
\]
Finally, using the Sobolev embedding $H^{1}\subset L^6$, for all $d\in\{2,3\}$, see the inequality~\eqref{eq:L6-d}, and the lower bound $\frac14\|\cdot\|_{L^4}^4-\frac12\|\cdot\|_{L^2}^2\ge -C$, for all $t\in[0,T]$ one has
\[
\|\tilde{R}(t)\|_{L^6}^2\le C\|\tilde{R}(t)\|_{H^1}^2\le C\bigl(J(\tilde{R}(t))+1\bigr).
\]
Using the inequalities proved above then concludes the proof of the claim~\eqref{eq:claim2} in the case $d\in\{2,3\}$.

$\bullet$ Step $3$: let us finally prove the required estimate. Using the condition $\gamma\in(\frac{d}{2},\Gamma)$, one has $\frac12+\frac{\gamma}{4}<1$. Using the definition of $\tilde{R}(t)$ and the smoothing inequality, for all $t\in[0,T]$, one has
\begin{align*}
\|\tilde{R}(t)\|_\gamma&\le \int_{0}^{t}\|e^{-(t-s)A^2}A  P^N  F(\tilde{R}(s)+\tilde{r}(s)+\tilde{Z}^{x_0}(t_{\ell(s)}))\|_{\gamma}ds\\
&\le C\int_{0}^{t}(t-s)^{-\frac12-\frac{\gamma}{4}}\| F(\tilde{R}(s)+\tilde{r}(s)+\tilde{Z}^{x_0}(t_{\ell(s)}))\|ds\\
&\le C\int_{0}^{t}(t-s)^{-\frac12-\frac{\gamma}{4}}\bigl(1+\|\tilde{R}(s)\|_{L^6}^3+\|\tilde{r}(s)\|_{L^6}^3+\|\tilde{Z}^{x_0}(t_{\ell(s)}))\|_{L^6}^3\bigr)ds\\
&\le C\int_{0}^{t}(t-s)^{-\frac12-\frac{\gamma}{4}}\bigl(1+\|\tilde{R}(s)\|_{L^6}^3+\|\tilde{r}(s)\|_{\gamma}^3+\|\tilde{Z}^{x_0}(t_{\ell(s)}))\|_{\gamma}^3\bigr)ds
\end{align*}
since $F$ is a polynomial of degree $3$. One then obtains the inequality
\[
\chi_{k-1}\underset{0\le t\le t_k}\sup~\|\tilde{R}(t)\|_\gamma\le C(T)\chi_{k-1}\underset{0\le t\le t_k}\sup~\bigl(\|\tilde{R}(t)\|_{L^6}^3+\|\tilde{r}(t)\|_{\gamma}^3\bigr)+\underset{0\le t\le T}\sup~\|\tilde{Z}^{x_0}(t)\|_{\gamma}^3.
\]
It remains to use the inequalities proved above to conclude the proof of Lemma~\ref{lem:tamed-R}.
\end{proof}

We are now in position to provide the proof of Theorem~\ref{theo:momentbounds-tamed}.
\begin{proof}[Proof of Theorem~\ref{theo:momentbounds-tamed}]
Using the estimates from Lemmas~\ref{lem:tamedGaussian}, \ref{lem:tamed-r} and~\ref{lem:tamed-R}, and the identity
\[
\tilde{X}(t)=\tilde{R}(t)+\tilde{r}(t)+\tilde{Z}^{x_0}(t),
\]
one obtains
\begin{equation}
\underset{0\le k\le K}\sup~\bigl(\E[\chi_{k-1}\underset{0\le t \le t_k}\sup~\|\tilde{X}(t)\|_\gamma^m]\bigr)^{\frac1m}\le C_m(T)(1+\|x_0\|_\gamma^q).
\end{equation}
Let us now prove the following inequality:
\begin{equation}
\underset{0\le k\le K}\sup~\bigl(\E[\bigl(1-\chi_{k}\bigr)\underset{0\le t \le t_k}\sup~\|\tilde{X}(t)\|_\gamma^m]\bigr)^{\frac1m}\le C_m(T)(1+\|x_0\|_\gamma^{q}).
\end{equation}
Using the identity
\[
\chi_{k-1}-\chi_{k}=\mathds{1}_{\underset{0\le \ell\le k-1}\sup~\|X_\ell\|_\gamma\le \Delta t^{-\theta}}-\mathds{1}_{\underset{0\le \ell\le k}\sup~\|X_\ell\|_\gamma\le \Delta t^{-\theta}}=\chi_{k-1}\mathds{1}_{\|X_k\|_\gamma> \Delta t^{-\theta}},
\]
and the convention $\chi_{-1}=1$, one has
\[
1-\chi_k=\sum_{\ell=0}^{k}\bigl(\chi_{\ell-1}-\chi_{\ell}\bigr)=\sum_{\ell=0}^{k}\chi_{\ell-1}\mathds{1}_{\|X_\ell\|_\gamma>\Delta t^{-\theta}}.
\]
for all $k\in\{0,\ldots,K\}$. As a consequence, one has
\begin{align*}
\bigl(\mathbb{E}[(1-\chi_k)\underset{0\le t \le t_k}\sup~\|\tilde{X}(t)\|_\gamma^m]\bigr)^{\frac1m}&\le \sum_{\ell=0}^{n}\bigl(\mathbb{E}[\chi_{\ell-1}\mathds{1}_{\|X_\ell\|_\gamma>\Delta t^{-\theta}}\underset{0\le t \le t_k}\sup~\|\tilde{X}(t)\|_\gamma^m]\bigr)^{\frac1m}\\
&\le \sum_{\ell=0}^{n}\bigl(\mathbb{E}[\chi_{\ell-1}\mathds{1}_{\|X_\ell\|_\gamma>\Delta t^{-\theta}}]\bigr)^{\frac{1}{2m}}\bigl(\mathbb{E}[\underset{0\le t \le t_k}\sup~\|\tilde{X}(t)\|_\gamma^{2m}]\bigr)^{\frac{1}{2m}}\\
&\le \sum_{\ell=0}^{n}\bigl(\mathbb{E}[\chi_{\ell-1}\Delta t^{m(\theta)\theta}\|X_\ell\|_\gamma^{m(\theta)}]\bigr)^{\frac{1}{2m}}\bigl(\mathbb{E}[\underset{0\le t \le t_k}\sup~\|\tilde{X}(t)\|_\gamma^{2m}]\bigr)^{\frac{1}{2m}},
\end{align*}
using Markov's inequality in the last step, with the auxiliary parameter $m(\theta)\ge 1$ chosen below.

On the one hand, one has for all $\ell\in\{0,\ldots,K\}$
\[
\bigl(\mathbb{E}[\chi_{\ell-1}\Delta t^{m(\theta)\theta}\|X_\ell\|_\gamma^{m(\theta)}]\bigr)^{\frac{1}{2m}}\le C_{m,M,\theta}(T)\Delta t^{\frac{M}{2m(\theta)}\theta}(1+\|x_0\|_\gamma^q)^{\frac{m(\theta)}{2m}}.
\]
On the other hand, for all $t\in[0,T]$, one has $\tilde{X}(t)=\tilde{Y}(t)+\tilde{Z}^{x_0}(t)$, therefore
\begin{align*}
\|\tilde{X}(t)\|_\gamma&\le \|x_0\|_\gamma+\|\tilde{Z}(t)\|_\gamma+\int_{0}^{t}\frac{\|e^{-(t-s)A^2}AP_N F(X_{\ell(s)})\|_\gamma}{1+\Delta t\|P_NF(X_{\ell(s)})\|}ds\\
&\le \|x_0\|_\gamma+\|\tilde{Z}(t)\|_\gamma+\int_{0}^{t}(t-s)^{-\frac{2+\gamma}{4}}\frac{\|P_NF(X_{\ell(s)})\|}{1+\Delta t\|P_NF(X_{\ell(s)})\|}ds\\
&\le \|x_0\|_\gamma+\|\tilde{Z}(t)\|_\gamma+\frac{1}{\Delta t}\int_{0}^{t}(t-s)^{-\frac{2+\gamma}{4}}ds.
\end{align*}
This gives
\[
\bigl(\mathbb{E}[\underset{0\le t \le t_k}\sup~\|\tilde{X}(t)\|_\gamma^{2m}]\bigr)^{\frac{1}{2m}}\le \bigl(\mathbb{E}[\underset{0\le t \le T}\sup~\|\tilde{X}(t)\|_\gamma^{2m}]\bigr)^{\frac{1}{2m}}\le C_m(T)\Delta t^{-1}(1+\|x_0\|_\gamma).
\]
Gathering the two estimates and using the inequality $k\le K=\frac{T}{\Delta t}$ then gives
\begin{align*}
\bigl(\mathbb{E}[(1-\chi_k)\underset{0\le t \le t_k}\sup~\|\tilde{X}(t)\|_\gamma^m]\bigr)^{\frac1m}&\le C_{m,\theta}(T)\Delta t^{-2}(1+\|x_0\|_\gamma)\Delta t^{\frac{m(\theta)}{2m}\theta}(1+\|x_0\|_\gamma^q)^{\frac{m(\theta)}{2m}}.
\end{align*}
Choosing $m(\theta)=\max(\frac {4m}\theta,1)$ then gives the required estimate.

Finally, using the inequality $\chi_K\le \chi_{K-1}$, one obtains
\begin{align*}
\bigl(\E[\underset{0\le t \le T}\sup~\|\tilde{X}(t)\|_\gamma^m]\bigr)^{\frac1m}&\le \bigl(\E[\bigl(1-\chi_{K}\bigr)\underset{0\le t \le t_K}\sup~\|\tilde{X}(t)\|_\gamma^m]\bigr)^{\frac1m}\\
&+\bigl(\E[\chi_{K}\underset{0\le t \le t_K}\sup~\|\tilde{X}(t)\|_\gamma^m]\bigr)^{\frac1m}\\
&\le C_m(T)(1+\|x_0\|_\gamma^{q}).
\end{align*}
This concludes the proof of Theorem~\ref{theo:momentbounds-tamed}.
\end{proof}

\section*{Acknowledgements}

The work of CEB is partially supported by the following project SIMALIN (ANR-19-CE40-0016) operated by the French National Research Agency.

%\bibliographystyle{plain}
%\bibliography{bib}

\end{document}